\documentclass[a4paper,11pt]{article}

\usepackage{amsmath}
\usepackage{amsthm}
\usepackage{amsfonts}
\usepackage{amssymb}
\usepackage{t-angles}

\usepackage[all]{xy}
\usepackage{amscd}

\newcommand{\si}{\sigma}

\newcommand{\tht}{\theta}

\newcommand{\id}{\mathrm{id}}

\newcommand{\ot}{\otimes}

\newcommand{\trl}{\triangleleft}
\newcommand{\trr}{\triangleright}

\def\ppr{\rightharpoonup}
\def\ppl{\leftharpoonup}

\newcommand{\li}{{}_{1}}
\newcommand{\lii}{{}_{2}}

\newcommand{\lmo}{{}_{(0)}} 

\newcommand{\loo}{{}_{(0)}}
\newcommand{\loi}{{}_{(-1)}}

\newcommand{\lmoo}{{}_{(0)}}

\newcommand{\lmi}{{}_{(1)}}

\newcommand{\lmoi}{{}_{(-1)}}

\newcommand{\mo}{{}_{(0)}}
\newcommand{\mi}{{}_{(1)}}

\newcommand{\moi}{{}_{(-1)}}

\newcommand{\boo}{{}_{[0]}}
\newcommand{\bi}{{}_{[1]}}
\newcommand{\bii}{{}_{[2]}}
\newcommand{\boi}{{}_{[-1]}}

\newcommand{\poo}{{}_{[0]}}
\newcommand{\poi}{{}_{[-1]}}
\newcommand{\ppi}{{}_{<1>}}
\newcommand{\pii}{{}_{<2>}}
\newcommand{\qi}{{}_{\{1\}}}
\newcommand{\qii}{{}_{\{2\}}}

\def\rbiprod{{\cdot\kern-.33em\triangleright\!\!\!<}}
\def\lbiprod{{>\!\!\!\triangleleft\kern-.33em\cdot\, }}

\def\lrbiprod{{\ \cdot\kern-.60em\triangleright\kern-.33em\triangleleft\kern-.33em\cdot\, }}

\def\lprod{{>\!\!\!\triangleleft\kern-.33em\ \, }}

\newcommand{\lrcoprod}{{\,\blacktriangleright\!\!\blacktriangleleft\, }}

\allowdisplaybreaks
\newtheorem{theorem}{Theorem}[section]
\newtheorem{lemma}[theorem]{Lemma}

\newtheorem{corollary}[theorem]{Corollary}
\theoremstyle{definition}
\newtheorem{definition}[theorem]{Definition}
\newtheorem{remark}[theorem]{Remark}

\title{Extending structures for left-symmetric bialgebras}
\author{Tao Zhang, Hui-Jun Yao}
\date{}

\begin{document}
 \maketitle

 \setcounter{section}{0}

\begin{abstract}
We introduce the concept of braided left-symmetric bialgebras  and construct cocycle bicrossproduct left-symmetric bialgebras.
As an application, we solve the extending problem for left-symmetric bialgebras by using some non-abelian cohomology theory.

\end{abstract}

\tableofcontents

\section{Introduction}
Left-symmetric algebras (also called pre-Lie algebras, quasi-associative algebras, Vinberg algebras and so on),  as a class of nonassociative algebras, are arising from the study of
convex homogenous cones,  affine manifolds and affine structures on Lie groups (\cite{Ki01}, \cite{P01}, \cite{Vi01}). It also plays an important role  in many fields of mathematics and mathematical physics. Recently, left-symmetric algebras are widely developed in many papers. The concept of  Hom-left-symmetric algebras was introduced  in \cite{Mak001} and played important roles in the study of Hom-Lie bialgebras and Hom-Lie
2-algebras  \cite{sh001,sh002}. Hom-left-symmetric algebras were studied from several aspects. The geometrization of Hom-left-symmetric algebras was studied in  \cite{ZQ001}.
The concept of  left-symmetric bialgebras was provided by Bai in \cite{Bai08}, where the left-symmetric analogue of the classical Yang-Baxter equation was investigated in details.

The theory of extending structure for many types of algebras were well  developed by A. L. Agore and G. Militaru in \cite{AM1,AM2,AM3,AM5,AM6}.
Let $A$ be an algebra and $E$ a vector space containing $A$ as a subspace.
The extending problem is to describe and classify all algebra structures on $E$ such that $A$ is a subalgebra of $E$.
They show that associated to any extending structure of $A$ by a complement space $V$, there is an  unified product on the direct sum space  $E\cong A\oplus V$.
Recently, extending structures for 3-Lie algebras, Lie bialgebras,  infinitesimal bialgebras, anti-flexible bialgebras and Lie conformal superalgebras were studied  in \cite{Z2,Z3,Z4,ZY,ZCY}.

In this paper we introduced the concept of braided  left-symmetric bialgebras and the construction of cocycle bicrossproduct left-symmetric bialgebras.
It is proved that  braided left-symmetric bialgebras give rise to braided Lie bialgebras.
We will show that this new concept will play a key role in considering extending problem for    left-symmetric bialgebras.
As an application, we solve the extending problem for    left-symmetric bialgebras by using some non-abelian cohomology theory.

This paper is organized as follows. In Section 2, we recall some
definitions and fixed some notations  of left-symmetric algebras. In Section 3, we introduce the concept of braided  left-symmetric bialgebras and provethe bosonisation
theorem associating braided left-symmetric bialgebras to ordinary left-symmetric bialgebras. At the end of this section, we also show the connection between braided left-symmetric bialgebras and braided Lie bialgebras.
In  section 4, we define the notion of matched pairs of  braided    left-symmetric bialgebras.
Besides, we construct the cocycle bicrossproduct left-symmetric bialgebras through two generalized braided left-symmetric bialgebras.
In section 5, we study the extending problems for left-symmetric bialgebras and show that they can be classified by some non-abelian cohomology theory.

Throughout the following of this paper, all vector spaces will be over a fixed field of character zero.
An algebra  $(A, \cdot)$  is always a left-symmetric algebra and  a coalgebra $(A, \Delta)$ is always a  left-symmetric coalgebra.
The identity map of a vector space $V$ is denoted by $\id_V: V\to V$ or simply $\id: V\to V$.
The flip map $\tau: V\ot V\to V\ot V$ is defined by $\tau(u\ot v)=v\ot u$ for all $u, v\in V$.

\section{Preliminaries}

\begin{definition} A  left-symmetric algebra $(A, \cdot)$ is a vector space equipped
 with a  product $\cdot: A\otimes A\rightarrow A$ such that the following left-symmetric condition is satisfied:
\begin{equation}
(a \cdot b)\cdot c-a\cdot (b \cdot c)=(b \cdot a)\cdot c-b\cdot (a \cdot c).
 \end{equation}
\end{definition}
In the following, we always omit $ ``\cdot" $ and write the product by $ab$ for simplicity.
It is well known that a left-symmetric algebra $(A, \cdot)$ give rise to a Lie algebra $\mathfrak{g}(A)$ with commutator
$[a,b]=ab-ba$.

\begin{definition}A  left-symmetric coalgebra $A$ is a vector space equipped
 with a coproduct $\Delta: A\rightarrow A\otimes A$ such that the following left-symmetric condition is satisfied:
\begin{equation}
(\Delta \otimes \id)\Delta(a)-(\id\otimes\Delta)\Delta (a)=\tau_{12}\left( (\Delta \otimes \id)\Delta(a)-(\id\otimes\Delta)\Delta (a)\right),
 \end{equation}
 for  any  $a \in A$, where $\tau_{13}(a\ot b\ot c)=c\ot b\ot a$.

 We denote a left-symmetric coalgebra by $(A, ~\Delta)$.
\end{definition}

\begin{definition}\cite{Bai08}
Let $A$ be a vector space. A left-symmetric bialgebra structure on $A$ is a pair of linear maps $(\alpha, \beta)$ such that $\alpha: A\rightarrow A \otimes A$,  $\beta: A ^{*}\rightarrow A ^{*} \otimes A ^{*}$ and
\begin{enumerate}
\item[(1)]$\alpha^{*}: A^{*} \otimes A^{*}\rightarrow  A^{*}$ is a left-symmetric algebra structure on $A^{*}$;

\item[(2)]$\beta^{*}: A  \otimes A \rightarrow  A $ is a left-symmetric algebra structure on $A $;

\item[(3)] $\alpha$ is a 1-cocycle of $\mathfrak{g}(A)$ associated to $L\otimes 1+1\otimes ad$ with values in $A\otimes A$,
that is
\begin{equation}
 \alpha ([a, b])=ab_{1}\otimes b_{2}+b_{1}\otimes [a, b_{2}]-ba_{1}\otimes a_{2}-a_{1}\otimes [b, a_{2}],
 \end{equation}
 for any $a, b \in A$, $\alpha(a)=a_{1}\otimes a_{2}\in A\otimes A$.
\item[(4)] $\beta$ is a 1-cocycle of $\mathfrak{g}(A^{*})$ associated to $L\otimes 1+1\otimes ad$ with values in $A^{*}\otimes A^{*}$,
   that is
\begin{equation}\label{wyyy01}
 \beta([f, g])=fg_{1}\otimes g_{2}+g_{1}\otimes [f, g_{2}]-gf_{1}\otimes f_{2}-f_{1}\otimes [g, f_{2}],
 \end{equation}
 for any $f, g \in A^{*}$, $\beta(f)=f_{1}\otimes f_{2}\in A^{*}\otimes A^{*}$.
\end{enumerate}
\end{definition}

\begin{remark}
Since $(A^{*}, ~\alpha^{*})$ is a left-symmetric algebra,  using the duality between $A$ and $A^*$, we obtain
 \begin{equation*}
  <\alpha^{*}(f\otimes g),  a>=<f\otimes g, \alpha(a)>=<f\otimes g, a_{1}\otimes a_{2}> =f(a_{1})g(a_{2}),
 \end{equation*}
 for any $f, g \in A^{*}$, $a \in A$. Thus, for any $a, b \in A$,  the left hand side of equation \eqref {wyyy01} is equal to
 $$
\begin{aligned}
& <\beta([f, g]), a\otimes b>=<[f, g], \beta^{*}(a\otimes b)>=<(\alpha^{*}-\alpha^{*}\tau)(f\otimes g), ab>=<f\otimes g, (\alpha -\tau\alpha )(ab)>,
\end{aligned}
$$
 and the right hand side is equal to
\begin{eqnarray*}
&&<fg_{1}\otimes g_{2}+g_{1}\otimes [f, g_{2}]-gf_{1}\otimes f_{2}-f_{1}\otimes [g, f_{2}], a\otimes b>\\
&=&<fg_{1}\otimes g_{2}, a\otimes b>+<g_{1}\otimes [f, g_{2}], a\otimes b>-<gf_{1}\otimes f_{2}, a\otimes b>-<f_{1}\otimes [g, f_{2}], a\otimes b>\\
&&=<(\alpha^{*}\otimes \id)(\id \otimes \beta)(f\otimes g), a\otimes b>+<(\id\otimes (\alpha^{*}-\alpha^{*}\tau)) \tau_{12}(\id \otimes \beta)(f\otimes g), a\otimes b>\\
&& -<(\alpha^{*}\otimes \id)(\id \otimes \beta)\tau(f\otimes g), a\otimes b>-<(\id\otimes (\alpha^{*}-\alpha^{*}\tau)) \tau_{12}(\id \otimes \beta)\tau(f\otimes g), a\otimes b>\\
&=&<f\otimes g, (\id \otimes \beta^{*})(\alpha \otimes \id)(a\otimes b)>+<f\otimes g, (\id \otimes \beta^{*})\tau_{12}(\id \otimes( \alpha -\tau\alpha ))(a\otimes b)>\\
&&-<f\otimes g, \tau(\id \otimes \beta^{*})(\alpha \otimes \id)(a\otimes b)>-<f\otimes g, \tau(\id \otimes \beta^{*})\tau_{12}(\id \otimes( \alpha -\tau\alpha ))(a\otimes b)>\\
&=&<f\otimes g, (\id \otimes \beta^{*})(a_{1}\otimes a_{2}\otimes b)>+<f\otimes g, (\id \otimes \beta^{*})\tau_{12}  (a\otimes b_{1}\otimes b_{2})>\\
&&-<f\otimes g, (\id \otimes \beta^{*})\tau_{12}(a\otimes b_{2}\otimes b_{1})>-<f\otimes g,  \tau(\id \otimes \beta^{*})(a_{1}\otimes a_{2}\otimes b)>\\
&&-<f\otimes g, \tau(\id \otimes \beta^{*})\tau_{12}(a\otimes b_{1}\otimes b_{2})>+<f\otimes g, \tau(\id \otimes \beta^{*})\tau_{12}(a\otimes b_{2}\otimes b_{1})>\\
&=&<f\otimes g,  a_{1}\otimes a_{2} b >+<f\otimes g,   b_{1}\otimes a  b_{2}> -<f\otimes g,  b_{2}\otimes ab_{1} >\\
&&-<f\otimes g, a_{2} b\otimes a_{1}>-<f\otimes g, a b_{2}\otimes b_{1}>+<f\otimes g,   a b_{1}\otimes b_{2}>\\
&=&<f\otimes g,  a_{1}\otimes a_{2} b+ b_{1}\otimes a  b_{2}-b_{2}\otimes ab_{1}- a_{2} b\otimes a_{1}-a b_{2}\otimes b_{1}+a b_{1}\otimes b_{2}>\\
&=&<f\otimes g,  (\id-\tau)(a_{1}\otimes a_{2} b+ b_{1}\otimes a  b_{2} +a b_{1}\otimes b_{2})>.
\end{eqnarray*}
Thus, we have equation \eqref {wyyy01} is equal to
 \begin{equation}\label{wyyy02}
 \alpha(ab)-\tau\alpha (ab)=(\id-\tau)\Big( a_{1}\otimes a_{2} b+ b_{1}\otimes a  b_{2} +a b_{1}\otimes b_{2}  \Big).
 \end{equation}
\end{remark}

If we denote $\alpha:=\Delta$,  then we can redefine the left-symmetric bialgebra  as follows.

\begin{definition} \label{dfnlb}
A  left-symmetric bialgebra $A$ is a vector space equipped
simultaneously with a left-symmetric algebra structure $(A, \cdot)$  and a left-symmetric  coalgebra structure $(A, \Delta)$ such that the following
compatibility conditions are satisfied:
\begin{equation}\label{eq:LB0}
 \Delta([a, b])=\sum  a b\li\ot b\lii+b\li\ot [a, b\lii]-b a\li\ot a\lii -a\li \ot [b, a\lii] ,
 \end{equation}
 \begin{equation}\label{eq:LB1}
 (\id-\tau )\Delta(ab)=\sum (\id-\tau )\left( a\li\ot a\lii b+ab\li\ot b\lii+b\li\ot ab\lii \right)
 \end{equation}
where $[a, b]$ is abbreviated as $ab-ba$ and we denote this left-symmetric bialgebra by $(A, \cdot, \Delta)$.
\end{definition}

For convenience, we would like to denote $\Delta (a)\cdot b: = \sum a\li\ot a\lii b$,   $a\cdot \Delta(b): =\sum  a b\li\ot b\lii$ and $a\bullet \Delta(b): =\sum  b\li\ot a b\lii$.  Then we can also write the compatibility conditions  as
\begin{equation}\label{bialg01}
\Delta ([a, b])=\Delta (a)\cdot b+a\cdot \Delta(b)+a\bullet \Delta(b)-\Delta (b)\cdot a-b\cdot \Delta(a)-b\bullet \Delta(a),
\end{equation}
\begin{equation}\label{bialg02}
(\id-\tau )\Delta(ab)=(\id-\tau)\Big(\Delta (a)\cdot b+a\cdot \Delta(b)+a\bullet \Delta(b)\Big),
\end{equation}


\begin{definition}
Let ${H}$ be a left-symmetric algebra and $V$ be a vector space. Then $V$ is called an ${H}$-bimodule if there is a pair of linear maps $ \trr: {H}\otimes V \to V, (x, v) \to x \trr v$ and $\trl: V\otimes {H} \to V, (v, x) \to v \trl x$  such that the following conditions hold:
\begin{eqnarray}
&&  (xy) \trr v- x\trr (y\trr v)=(yx)\trr v-y\trr (x\trr v), \\
&&  (v \trl x) \trl y-v \trl (xy) =(x\trr v)\trl y-x \trr (v\trl y),
\end{eqnarray}
for all $x, y\in {H}$ and $v\in V$.
\end{definition}
The category of  bimodules over $H$ is denoted  by ${}_{H}\mathcal{M}{}_{H}$.

\begin{definition}
Let ${H}$ be a left-symmetric coalgebra and  $V$  a vector space. Then $V$ is called an ${H}$-bicomodule if there is a pair of linear maps $\phi: V\to {H}\otimes V$ and $\psi: V\to V\otimes {H}$  such that the following conditions hold:
\begin{eqnarray}
 &&\left(\Delta_{H} \otimes \id _{V}\right)\phi(v)-\left(\id _{H} \otimes \phi\right) \phi(v)=\tau_{12}\Big(\left(\Delta_{H} \otimes \id _{V}\right)\phi(v)-\left(\id _{H} \otimes \phi\right) \phi(v)\Big), \\
 &&\left( \psi\otimes \id _{H} \right) \psi(v)-\left(\id_{V}\ot\Delta_{H} \right)\psi(v)=\tau_{12}\Big((\phi\ot \id_{H})\psi(v)-(\id_{H} \ot \psi)\phi(v)\Big).
\end{eqnarray}
If we denote by  $\phi(v)=v\moi\ot v\mo$ and $\psi(v)=v\mo\ot v\mi$, then the above equations can be written as:
\begin{eqnarray}
  &&\Delta_{H}\left(v_{(-1)}\right) \otimes v_{(0)}-v_{(-1)} \otimes \phi\left(v_{(0)}\right)=\tau_{12}\Big(\Delta_{H}\left(v_{(-1)}\right) \otimes v_{(0)}-v_{(-1)} \otimes \phi\left(v_{(0)}\right)\Big),\\
  &&\psi\left(v_{(0)}\right) \otimes v_{(1)}-v_{(0)} \otimes \Delta_{H}\left(v_{(1)}\right)=\tau_{12}\Big(\phi(v_{(0)})\ot v_{(1)}-v_{(-1)}\ot \psi(v_{(0)})\Big).
\end{eqnarray}
\end{definition}
The category of  bicomodules over $H$ is denoted by ${}^{H}\mathcal{M}{}^{H}$.

\begin{definition}
Let ${H}$ and  ${A}$ be left-symmetric algebras. An action of ${H}$ on ${A}$ is a pair of linear maps $\trr: {H}\otimes {A} \to {A}, (x,  a) \to x \trr a$ and $\trl: {A}\otimes {H} \to {A}, (a,  x) \to a \trl x$  such that $A$ is an $H$-bimodule and  the following conditions hold:
\begin{eqnarray}
  &&x\trr (ab)-(x\trr a)b=a(x\trr b)-(a \trl x)b, \\
  &&(ab)\trl x-a (b\trl x)=(ba)\trl x-b(a\trl x),
\end{eqnarray}
for all $x\in {H}$ and $a, b\in {A}$. In this case, we call $(A,  \trr,  \trl)$ to be an $H$-bimodule left-symmetric algebra.
\end{definition}

\begin{definition}
Let ${H}$ and  ${A}$ be left-symmetric coalgebras. An coaction of ${H}$ on ${A}$ is a pair of linear maps $\phi: {A}\to {H}\otimes {A}$ and $\psi: {A}\to {A}\otimes {H}$ such that $A$ is an $H$-bicomodule and  the following conditions hold:
\begin{eqnarray}
  &&(\id_H \otimes \Delta_{A})\phi(a)-(\phi\otimes\id_A) \Delta_{A}(a)=\tau_{12}\Big((\id_A\otimes \phi) \Delta_{A}(a)-(\psi\otimes\id_A)\Delta_{A}(a)\Big),\\
  &&(\Delta_{A}\otimes \id_H)\psi(a)-(\id_A\otimes \psi) \Delta_{A}(a)=\tau_{12}\Big((\Delta_{A}\otimes \id_H)\psi(a)-(\id_A\otimes \psi) \Delta_{A}(a)\Big).
\end{eqnarray}
If we denote by  $\phi(a)=a\moi\ot a\mo$ and $\psi(a)=a\mo\ot a\mi$, then the above equation can be written as
\begin{eqnarray}
  &&a_{(-1)} \otimes \Delta_{A}\left(a_{(0)}\right)-\phi\left(a_{1}\right) \otimes a_{2}=\tau_{12}\Big(a_{1}\ot \phi(a_{2})-\psi(a_{1})\ot a_{2}\Big),\\
    &&\Delta_{A}\left(a_{(0)}\right) \otimes a_{(1)}-a_{1} \otimes \psi\left(a_{2}\right)=\tau_{12}\Big( \Delta_{A}\left(a_{(0)}\right) \otimes a_{(1)}-a_{1} \otimes \psi\left(a_{2}\right)\Big)
\end{eqnarray}
for all $a\in {A}$. In this case,  we call $(A,  \phi,  \psi)$ to be an $H$-bicomodule left-symmetric coalgebra.
\end{definition}

\begin{definition}
Let $({A}, ~\cdot)$ be a given   left-symmetric  algebra (left-symmetric coalgebra,   left-symmetric   bialgebra) and  $E$ be a vector space.
An extending system of ${A}$ through $V$ is a left-symmetric  algebra  (left-symmetric coalgebra,   left-symmetric  bialgebra) on $E$
such that $V$ a complement subspace of ${A}$ in $E$, the canonical injection map $i: A\to E, a\mapsto (a, 0)$  or the canonical projection map $p: E\to A, (a, x)\mapsto a$ is a left-symmetric  algebra (left-symmetric coalgebra,    left-symmetric bialgebra) homomorphism.
The extending problem is to describe and classify up to an isomorphism  the set of all  left-symmetric   algebra  (left-symmetric coalgebra,   left-symmetric  bialgebra) structures that can be defined on $E$.
\end{definition}

We remark that our definition of extending system of ${A}$ through $V$ contains not only extending structure in \cite{AM1,AM2,AM3}
but also the global extension structure in \cite{AM5}.
In fact, the canonical injection map $i: A\to E$ is a left-symmetric   (co)algebra homomorphism if and only if $A$ is a   left-symmetric  sub(co)algebra of $E$.

\begin{definition}
Let ${A} $ be a   left-symmetric   algebra (left-symmetric  coalgebra, left-symmetric    bialgebra)and  $E$  be a left-symmetric  algebra  (left-symmetric coalgebra,   left-symmetric bialgebra) such that
${A}$ is a subspace of $E$ and $V$ a complement of
${A}$ in $E$. For a linear map $\varphi: E \to E$ we
consider the diagram:
\begin{equation}\label{eq:ext1}
\xymatrix{
   0  \ar[r]^{} &A \ar[d]_{\id_A} \ar[r]^{i} & E \ar[d]_{\varphi} \ar[r]^{\pi} &V \ar[d]_{\id_V} \ar[r]^{} & 0 \\
   0 \ar[r]^{} & A \ar[r]^{i'} & {E} \ar[r]^{\pi'} & V \ar[r]^{} & 0.
   }
\end{equation}
where $\pi: E\to V$ are the canonical projection maps and $i: A\to E$ are the inclusion maps.
We say that $\varphi: E \to E$ \emph{stabilizes} ${A}$ if the left square of the diagram \eqref{eq:ext1} is  commutative.
Let $(E,  \cdot)$ and $(E,  \cdot')$ be two   left-symmetric  algebra (left-symmetric coalgebra,   left-symmetric  bialgebra) structures on $E$. $(E,  \cdot)$ and $(E,  \cdot')$ are called \emph{equivalent}, and we denote this by $(E, \cdot) \equiv (E,  \cdot')$, if there exists a left-symmetric  algebra (left-symmetric coalgebra,  left-symmetric   bialgebra) isomorphism $\varphi: (E, \cdot)
\to (E, \cdot')$ which stabilizes ${A}$. Denote by $Extd(E, {A} )$ ($CExtd(E, {A} )$, $BExtd(E, {A} )$) the set of equivalent classes of   left-symmetric   algebra (left-symmetric coalgebra,  left-symmetric   bialgebra) structures on $E$.
\end{definition}

\section{Braided   left-symmetric bialgebras}
In this section, we introduce  the concept of   left-symmetric Hopf bimodule and braided    left-symmetric bialgebra which will be used in the following sections.

\subsection{Left-symmetric Hopf bimodule and braided   left-symmetric bialgebra}


\begin{definition}
 Let $H$ be a    left-symmetric bialgebra. A left-symmetric Hopf bimodule over $H$ is a space $V$ endowed with maps
\begin{align*}
&\trr: H\otimes V \to V, \quad \trl: V\otimes H \to V, \\
&\phi: V \to H \otimes V, \quad  \psi: V \to V\otimes H,
\end{align*}
such that $V$ is simultaneously a bimodule, a  bicomodule over $H$ and satisfying
 the following compatibility conditions:
 \begin{enumerate}
\item[(HM1)]$\phi(v \trl x)-\phi(x\trr v)$\\
$=v_{(-1)} \otimes\left(v_{(0)} \trl x\right)+x_{1}\ot (v\trl x_{2})-x_{1}\ot (x_{2}\trr v)-xv_{(-1)}\ot v_{(0)}-v_{(-1)}\ot (x \trr v_{(0)})$,
\item[(HM2)]
$\psi(x \trr v)-\psi(v\trl x)=\left(x \trr  v_{(0)}\right) \otimes v_{(1)}+v_{(0)}\otimes [x, v_{(1)}] -(v\trl x_{1})\ot x_{2},$
\item[(HM3)] $\phi(x \trr v)-\tau\psi(x\trr v)$\\
$=x_{1} \otimes\left(x_{2}\trr v\right)+x v_{(-1)}\otimes v_{(0)}+v_{(-1)}\ot( x\trr v_{(0)})-v_{(1)}\ot (x\trr v_{(0)})-xv_{(1)} \ot v_{(0)}$,
\item[(HM4)] $\phi(v \trl x)-\tau\psi(v \trl x ) =v_{(-1)} \otimes\left(v_{(0)} \trl x\right)+x_{1}\ot (v\trl x_{2})-v_{(1)}x\ot v_{(0)}- x_{2}\ot (v \trl  x_{1})$,
\end{enumerate}
then $V$ is called a left-symmetric Hopf bimodule over $H$.
\end{definition}
We denote  the  category of  left-symmetric Hopf bimodules over $H$ by ${}^{H}_{H}\mathcal{M}{}^{H}_{H}$.

\begin{definition} Let $H$  be a  left-symmetric bialgebra.
If $A$ is a left-symmetric algebra and a left-symmetric coalgebra in ${}^{H}_{H}\mathcal{M}{}^{H}_{H}$,  we call $A$ a \emph{braided    left-symmetric bialgebra}  if the following conditions are satisfied:
\begin{enumerate}
\item[(BB1)]
$\Delta([a, b])$\\
$=  a b\li\ot b\lii+b\li\ot [a, b\lii]-b a\li\ot a\lii -a\li \ot [b, a\lii]+a_{(0)} \otimes\left(a_{(1)} \trr b\right)$\\
$+(a\trl b_{(-1)}) \otimes b_{(0)}+b_{(0)} \otimes (a\trl b_{(1)}) -b_{(0)} \otimes\left(b_{(1)} \trr a\right)-(b\trl a_{(-1)}) \otimes a_{(0)}-a_{(0)} \otimes (b\trl a_{(1)})$,
\end{enumerate}
\begin{enumerate}
\item[(BB2)]
$(\id-\tau) \Delta(ab) $\\
$=(\id-\tau)\Big(a_{1} \otimes a_{2} b+a b_{1} \otimes b_{2}+ b_{1} \otimes a b_{2}+a_{(0)} \otimes\left(a_{(1)} \trr b\right)+(a\trl b_{(-1)}) \otimes b_{(0)}+b_{(0)} \otimes (a\trl b_{(1)})\Big)$.
\end{enumerate}
\end{definition}
Here we say $A$ to be a left-symmetric algebra and a left-symmetric  coalgebra in ${}^{H}_{H}\mathcal{M}{}^{H}_{H}$ means that $A$ is simultaneously an $H$-bimodule left-symmetric  algebra (left-symmetric coalgebra) and $H$-bicomodule  left-symmetric algebra (left-symmetric coalgebra).

Now we construct left-symmetric bialgebra from braided left-symmetric bialgebra.

\begin{theorem} Let  $H$ be a   left-symmetric bialgebra, $A$ be a left-symmetric  algebra and a  left-symmetric coalgebra in ${}^{H}_{H}\mathcal{M}{}^{H}_{H}$.
We define product and coproduct on the direct sum vector space $E:=A \oplus H$ by
$$
\begin{aligned}
&(a, x)(b, y):=(a b+x\trr b+a \trl y, \, x y), \\
&\Delta_{E}(a, x):=\Delta_{A}(a)+\phi(a)+\psi(a)+\Delta_{H}(x).
\end{aligned}
$$
Then there is a left-symmetric bialgebra structure on $E$  if and only if  $A$ is a braided    left-symmetric bialgebra in ${}^{H}_{H}\mathcal{M}{}^{H}_{H}$. We call $E$ the biproduct of ${A}$ and ${H}$ which will be  denoted   by $A\lbiprod H$.
\end{theorem}

\begin{proof}
First, we need to verify whether the product is  left-symmetric. For $\forall a, b, c\in A$, and $\forall x, y, z\in H$, we will check that
$$((a, x) (b, y)) (c, z)-(a, x)((b, y) (c, z))=((b, y)(a, x)) (c, z)-(b, y)( (a, x)(c, z)).$$
By definition, the left hand side is equal to
$$\begin{aligned}
&\big((a, x) (b, y)\big) (c, z)-(a, x)\big((b, y)(c, z)\big)\\
=&\big(a b+x \trr b+a \trl y, x y\big) (c, z)-(a, x)(b c+y\trr c+b \trl z, y z)\\
=&\big((a b) c+(x \trr b) c+(a \trl y) c+(x y)\trr c+(ab)\trl z+(x \trr b) \trl z+(a \trl y) \trl z, (x y) z\big)\\
&-\big((a(b c)+a(y\trr c)+a(b \trl z)+x \trr (b c)+x \trr (y\trr c)+x\trr (b \trl z)+a \trl(y z), x(y z)\big),
\end{aligned}
$$
and the right hand side is equal to
$$\begin{aligned}
&\big((b, y) (a, x)\big) (c, z)-(b, y)\big((a, x)(c, z)\big)\\
=&\big(b a+y \trr a+b \trl x, y x\big) (c, z)-(b, y)(a c+x\trr c+a \trl z, x z)\\
=&\big((b a) c+(y \trr a) c+(b \trl x) c+(y x)\trr c+(ba)\trl z+(y \trr a) \trl z+(b \trl x) \trl z, \, (y x) z\big)\\
&-\big((b(a c)+b(x\trr c)+b(a \trl z)+y \trr (a c)+y \trr (x\trr c)+y\trr (a \trl z)+b \trl(x z), \, y(x z)\big).
\end{aligned}
$$
Thus the two sides are equal to each other if and only if $(A, \trr, \trl)$ is a bimodule left-symmetric algebra over $H$.

Next, we need to verify  that the coproduct is left-symmetric . For all $(a, x)\in A \oplus H$. we have to prove
$$(\Delta_E \otimes \id) \Delta_E(a, x)-(\id \otimes \Delta_E) \Delta_E(a, x)=\tau_{12}((\Delta_E \otimes \id) \Delta_E(a, x)-(\id \otimes \Delta_E) \Delta_E(a, x)).$$
 By definition, the left hand side is equal to
$$
\begin{aligned}
&(\Delta_E \otimes \id) \Delta_E(a, x)-(\id \otimes \Delta_E) \Delta_E(a, x)\\
=&(\Delta_E\otimes \id)\left(a_{1} \otimes a_{2}+a_{(-1)} \otimes a_{(0)}+a_{(0)} \otimes a_{(1)}+x_{1} \otimes x_{2}\right)\\
&-(\id \otimes \Delta_E)\left(a_{1} \otimes a_{2}+a_{(-1)} \otimes a_{(0)}+a_{(0)} \otimes a_{(1)}+x_{1} \otimes x_{2}\right)\\
=&\Delta_{A}\left(a_{1}\right) \otimes a_{2}+\phi\left(a_{1}\right) \otimes a_{2}+\psi\left(a_{1}\right) \otimes a_{2}
+\Delta_{H}\left(a_{(-1)}\right) \otimes a_{(0)}\\
&+\Delta_{A}\left(a_{(0)}\right) \otimes a_{(1)}+\phi\left(a_{(0)}\right) \otimes a_{(1)}
+\psi\left(a_{(0)}\right) \otimes a_{(1)}+\Delta_{H}\left(x_{1}\right) \otimes x_{2}\\
&-a_{1} \otimes \Delta_{A}\left(a_{2}\right)-a_{1} \otimes \phi\left(a_{2}\right)-a_{1} \otimes \psi\left(a_{2}\right)-a_{(-1)} \otimes \Delta_{A}\left(a_{(0)}\right)\\
&-a_{(-1)} \otimes \phi\left(a_{(0)}\right)-a_{(-1)} \otimes \psi\left(a_{(0)}\right)-a_{(0)} \otimes \Delta_{H}\left(a_{(1)}\right)-x_{1} \otimes \Delta_{H}\left(x_{2}\right).
\end{aligned}
$$
The right hand side can be computed similarly.
Thus the two sides are equal to each other if and only if $(A, \phi, \psi)$ is a bicomodule left-symmetric coalgebra over $H$.

 Finally, we show  the first compatibility condition:
 $$\begin{aligned}
&\Delta_E([(a, x), (b, y)])\\
=&\Delta_E(a, x) \cdot(b, y)+(a, x) \cdot \Delta_E(b, y)+(a, x) \bullet \Delta_E(b, y)\\
&-\Delta_E(b, y) \cdot(a, x)-(b, y)\cdot \Delta_E(a, x)- (b, y)\bullet \Delta_E(a, x).
\end{aligned}
$$
By direct computations, the left hand side is equal to
 \begin{eqnarray*}
&&\Delta_E([(a, x), (b, y)])\\
&=&\Delta_E(a b+x \trr b+a \trl y,  x y)-\Delta_E(ba+y \trr a+b \trl x,  y x)\\
&=&\Delta_A(a b)+\phi(a b)+\psi(a b)+\Delta_A(x \trr b)+\phi(x \trr b)+\psi(x \trr b)+\Delta_A(a \trl y)\\
&&+\phi(a \trl y)+\psi(a \trl y)+\Delta_{H}(x y)-\Delta_A(b a)-\phi(b a)-\psi(b a)-\Delta_A(y \trr a)\\
&&-\phi(y \trr a)-\psi(y \trr a)-\Delta_A(b \trl x)-\phi(b \trl x)-\psi(b \trl x)-\Delta_{H}(y x),
\end{eqnarray*}
and the right hand side is equal to
\begin{eqnarray*}
&&\Delta_E(a, x) \cdot(b, y)+(a, x) \cdot \Delta_E(b, y)+(a, x) \bullet \Delta_E(b, y)\\
&&-\Delta_E(b, y) \cdot(a, x)-(b, y)\cdot \Delta_E(a, x)- (b, y)\bullet \Delta_E(a, x)\\
&=&\left(a_{1} \otimes a_{2}+a_{(-1)} \otimes a_{(0)}+a_{(0)} \otimes a_{(1)}+x_{1} \otimes x_{2}\right) \cdot(b, y)\\
&&+(a, x) \cdot\left(b_{1} \otimes b_{2}+b_{(-1)} \otimes b_{(0)}+b_{(0)} \otimes b_{(1)}+y_{1} \otimes y_{2}\right)\\
&&+(a, x) \bullet\left(b_{1} \otimes b_{2}+b_{(-1)} \otimes b_{(0)}+b_{(0)} \otimes b_{(1)}+y_{1} \otimes y_{2}\right)\\
&&-\left(b_{1} \otimes b_{2}+b_{(-1)} \otimes b_{(0)}+b_{(0)} \otimes b_{(1)}+y_{1} \otimes y_{2}\right)\cdot(a, x)\\
&&-(b, y)\cdot\left(a_{1} \otimes a_{2}+a_{(-1)} \otimes a_{(0)}+a_{(0)} \otimes a_{(1)}+x_{1} \otimes x_{2}\right)\\
&&-(b, y)\bullet\left(a_{1} \otimes a_{2}+a_{(-1)} \otimes a_{(0)}+a_{(0)} \otimes a_{(1)}+x_{1} \otimes x_{2}\right)\\
&=&a_{1} \otimes\left(a_{2} b+a_{2} \trl y\right)+a_{(-1)} \otimes\left(a_{(0)} b+a_{(0)} \trl y\right)+a_{(0)} \otimes\left(a_{(1)} \trr b+a_{(1)} y\right) \\
&&+x_{1} \otimes\left(x_{2} \trr b+x_{2} y\right) +\left(a b_{1}+x \trr b_{1}\right) \otimes b_{2}+(a \trl b_{(-1)}+x b_{(-1)} ) \otimes b_{(0)}\\
&&+\left(a b_{(0)}+x\trr b_{(0)}\right) \otimes b_{(1)}+\left(a \trl y_{1}+x y_{1}\right) \otimes y_{2} + b_{1}\otimes\left(a b_{2}+x \trr b_{2}\right)\\
&&+ b_{(-1)} \otimes (ab_{(0)}+x\trr b_{(0)})+b_{(0)}\ot (xb_{(1)}+a\trl b_{(1)} )+y_{1} \otimes \left(a \trl y_{2}+x y_{2}\right)\\
&&-b_{1} \otimes\left(b_{2} a+b_{2} \trl x\right)-b_{(-1)} \otimes\left(b_{(0)} a+b_{(0)} \trl x\right)-b_{(0)} \otimes\left(b_{(1)} \trr a+b_{(1)} x\right)\\
&&-y_{1} \otimes\left(y_{2} \trr a+y_{2} x\right)-\left(b a_{1}+y \trr a_{1}\right) \otimes a_{2}-(b \trl a_{(-1)}+y a_{(-1)}) \otimes a_{(0)}\\
&&-\left(b a_{(0)}+y\trr a_{(0)}\right) \otimes a_{(1)}-\left(b \trl x_{1} +y x_{1}\right) \otimes x_{2}- a_{1}\otimes\left(b a_{2}+y \trr a_{2}\right)\\
&&- a_{(-1)} \otimes (ba_{(0)}+y\trr a_{(0)})-a_{(0)}\ot (ya_{(1)}+b\trl a_{(1)})  -x_{1} \otimes \left(b \trl x_{2}+y x_{2}\right).
\end{eqnarray*}
Then the two sides are equal to each other if and only if

(1)$\Delta([a, b])=a b\li\ot b\lii+b\li\ot [a,b\lii]-b a\li\ot a\lii -a\li \ot [b,a\lii]$

 $ ~~~+a_{(0)} \otimes\left(a_{(1)} \trr b\right)+(a\trl b_{(-1)}) \otimes b_{(0)}+b_{(0)} \otimes (a\trl b_{(1)})$

$ ~~~-b_{(0)} \otimes\left(b_{(1)} \trr a\right)-(b\trl a_{(-1)}) \otimes a_{(0)}-a_{(0)} \otimes (b\trl a_{(1)})$,

(2) $\phi(x \trr b)-\phi(b \trl x)=x_{1} \otimes\left(x_{2}\trr b\right)+x b_{(-1)}\otimes b_{(0)}$\\
$~~~~~~~~~+b_{(-1)}\ot (x\trr b_{(0)})-b_{(-1)}\ot (b_{(0)}\trl x)-x_{1} \otimes(b\trl x_{2})$,

(3) $\psi(x \trr b)-\psi(b\trl x)=\left(x \trr  b_{(0)}\right) \otimes b_{(1)}+b_{(0)}\otimes [x,b_{(1)}]-(b\trl x_{1})\ot x_{2}$,

(4) $\Delta_{A}(x\trr b)-\Delta_{A}(b\trl x)=\left(x \trr b_{1}\right) \otimes b_{2}- b_{1}\ot  (b_{2}\trl x)+ b_{1}\ot (x\trr  b_{2})$,

(5) $\phi([a ,b]) =a_{(-1)} \otimes\ [a_{(0)}, b]+b_{(-1)}\ot [a,b_{(0)}]$,

 (6)$\psi([a, b]) =a b_{(0)} \otimes b_{(1)}-ba_{(0)}\ot a_{(1)}$.

We also need to show the second compatibility condition:
 $$(\id -\tau)(\Delta_{E}((a, x)(b, y)))=(\id-\tau)(\Delta_E(a, x) \cdot(b, y)+(a, x) \cdot \Delta_E(b, y)+(a, x) \bullet \Delta_E(b, y)).$$
The left hand side is equal to
$$\begin{aligned}
&(\id -\tau)(\Delta_{E}((a, x)(b, y)))\\
=&(\id -\tau)\Delta_E(a b+x \trr b+a \trl y,  x y)\\
=&\Delta_A(a b)+\phi(a b)+\psi(a b)+\Delta_A(x \trr b)+\phi(x \trr b)+\psi(x \trr b)+\Delta_A(a \trl y)\\
&+\phi(a \trl y)+\psi(a \trl y)+\Delta_{H}(x y)-\tau\Big( \Delta_A(a b)+\phi(a b)+\psi(a b)+\Delta_A(x \trr b)\\
&+\phi(x \trr b)+\psi(x \trr b)+\Delta_A(a \trl y)+\phi(a \trl y)+\psi(a \trl y)+\Delta_{H}(x y)  \Big),
\end{aligned}
$$
and the right hand side is equal to
\begin{eqnarray*}
&&(\id -\tau)\Delta_E(a, x) \cdot(b, y)+(a, x) \cdot \Delta_E(b, y)+(a, x) \bullet \Delta_E(b, y)\\
&=&(\id -\tau)\Big(\left(a_{1} \otimes a_{2}+a_{(-1)} \otimes a_{(0)}+a_{(0)} \otimes a_{(1)}+x_{1} \otimes x_{2}\right) \cdot(b, y)\\
&&+(a, x) \cdot\left(b_{1} \otimes b_{2}+b_{(-1)} \otimes b_{(0)}+b_{(0)} \otimes b_{(1)}+y_{1} \otimes y_{2}\right)\\
&&+(a, x) \bullet\left(b_{1} \otimes b_{2}+b_{(-1)} \otimes b_{(0)}+b_{(0)} \otimes b_{(1)}+y_{1} \otimes y_{2}\right)\Big)\\
&=&(\id-\tau)\Big(a_{1} \otimes\left(a_{2} b+a_{2} \trl y\right)+a_{(-1)} \otimes\left(a_{(0)} b+a_{(0)} \trl y\right)+a_{(0)} \otimes\left(a_{(1)} \trr b+a_{(1)} y\right) \\
&&+x_{1} \otimes\left(x_{2} \trr b+x_{2} y\right) +\left(a b_{1}+x \trr b_{1}\right) \otimes b_{2}+(a \trl b_{(-1)}+x b_{(-1)} ) \otimes b_{(0)}\\
&&+\left(a b_{(0)}+x\trr b_{(0)}\right) \otimes b_{(1)}+\left(a \trl y_{1}+x y_{1}\right) \otimes y_{2} + b_{1}\otimes\left(a b_{2}+x \trr b_{2}\right)\\
&&+ b_{(-1)} \otimes (ab_{(0)}+x\trr b_{(0)})+b_{(0)}\ot (xb_{(1)}+a\trl b_{(1)} )+y_{1} \otimes \left(a \trl y_{2}+x y_{2}\right)\Big).
\end{eqnarray*}
Thus the two sides are equal to each other if and only if satisfying the following conditions

 (7) $\Delta_{A}(ab)-\tau\Delta_{A}(ab) $\\
$~~~~~~~~~~~=(\id-\tau)\Big(a_{1} \otimes a_{2} b+a b_{1} \otimes b_{2}+ b_{1} \otimes a b_{2}+a_{(0)} \otimes\left(a_{(1)} \trr b\right)+(a\trl b_{(-1)}) \otimes b_{(0)}+b_{(0)} \otimes (a\trl b_{(1)})\Big)$,

(8) $\phi(x \trr b)-\tau\psi(x\trr b)$\\
$~~~~~~~~~=x_{1} \otimes\left(x_{2}\trr b\right)+ x b_{(-1)}  \otimes b_{(0)}+b_{(-1)}\ot (x\trr b_{(0)})-b_{(1)}\ot (x\trr b_{(0)})-xb_{(1)} \ot b_{(0)}$,

(9) $\phi(a \trl y)-\tau\psi(a \trl y ) =a_{(-1)} \otimes\left(a_{(0)} \trl y\right)+y_{1}\ot (a\trl y_{2})-a_{(1)}y\ot a_{(0)}- y_{2}\ot (a \trl y_{1})$,

(10) $\phi(a b)-\tau\psi(ab)= b_{(-1)}\ot  a b_{(0)}  +a_{(-1)} \otimes\   a_{(0)} b-b_{(1)}\ot a b_{(0)} $,

(11) $\Delta_{A}(x\trr b)-\tau\Delta_{A}(x\trr b)=(\id-\tau)\Big( \left(x \trr b_{1}\right) \otimes b_{2} + b_{1}\ot (x\trr  b_{2})\Big)$,

(12) $\Delta_{A}(a \trl y)-\tau\Delta_{A}(a \trl y)=a_{1} \otimes\left(a_{2} \trl y\right) -\left(a_{2} \trl y\right)\ot a_{1}$.

\noindent From (4)--(6) and (10)--(12) we have that $A$ is a left-symmetric algebra and  left-symmetric coalgebra  in ${}^{H}_{H}\mathcal{M}{}^{H}_{H}$,
from (2)--(3) and (8)--(9) we get that $A$ is a  left-symmetric Hopf bimodule over $H$,  and (1) together with (7) are the conditions for $A$ to be a braided  left-symmetric bialgebra.

The proof is completed.
\end{proof}
\subsection{From quasitriangular left-symmetric bialgebra to braided left-symmetric bialgebra}
Let $(A, \cdot)$ be a left-symmetric algebra and
$\displaystyle  {r}=\sum_i{u_i\otimes v_i}\in A\otimes A$. Set
\begin{equation}
 {r}_{12}=\sum_iu_i\otimes v_i\otimes 1, \quad
 {r}_{13}=\sum_{i}u_i\otimes 1\otimes v_i, \quad r_{23}=\sum_i 1\otimes u_i\otimes v_i,
\end{equation}
In this section, we consider a special class of left-symmetric bialgebras.
That is, the left-symmetric bialgebra $(A, \Delta_r)$ on a left-symmetric algebra $(A, \cdot)$,  with the linear map $\Delta_r$ defined by
\begin{eqnarray}\label{2341}
\Delta_r(a)&=&\sum_i au_i\otimes v_i+u_{i}\otimes  [a, v_{i}].
\end{eqnarray}

\begin{theorem} \cite{Bai08}
Let $A$ be a left-symmetric algebra and $r\in A\ot A$. Suppose $r$ is symmetric, that is $r=\tau (r)$. Then $\Delta_{r}$ defined by \eqref{2341} induces a left-symmetric algebra on $A^{*}$ such that $A$ is a left-symmetric bialgebra if it satisfying the  following $S$-equation:
\begin{eqnarray}
[[r, r]]=r_{12}r_{23}-r_{12}r_{13}+[r_{13}, r_{23}]=0.
\end{eqnarray}
\end{theorem}

This kind of left-symmetric bialgebra is called a quasitriangular  left-symmetric bialgebra.

\begin{theorem}
 Let $(A, \cdot, \Delta_r)$ be a quasitriangular left-symmetric bialgebra and $M$ an $A$-bimodule. Then $M$ becomes a  left-symmetric Hopf bimodule  over $A$ with maps $\phi:  M \rightarrow A \otimes M$ and $\psi: M \rightarrow M \otimes A$ given by
\begin{equation}
\phi(m):=\sum_{i}u_{i}\ot m\trl v_{i}-u_{i} \otimes  v_{i} \trr m,\quad \psi(m):=\sum_{i}m\trl u_{i} \otimes v_{i}
\end{equation}
\end{theorem}

\begin{proof}
We firstly prove that $M$ is a  bicomodule by:
\begin{eqnarray*}
 &&\left(\Delta_{r} \otimes \id \right)\phi(m)-\left(\id  \otimes \phi\right) \phi(m)=\tau_{12}(\left(\Delta_{r} \otimes \id  \right)\phi(m)-\left(\id   \otimes \phi\right) \phi(m)),\\
 &&\left( \psi\otimes \id   \right) \psi(m)-\left(\id \ot\Delta_{r} \right)\psi(m)=\tau_{12}((\phi\ot id )\psi(m)-(\id  \ot \psi)\phi(m)).
\end{eqnarray*}
For the firt equation ,we have the left hand side equal to
\begin{eqnarray*}
&&\left(\Delta_{r} \otimes \id\right)\phi(m)-\left(\id \otimes \phi\right) \phi(m)\\
&=&(\Delta_{r} \otimes \id)(u_{i}\ot m\trl v_{i}-u_{i} \otimes  v_{i} \trr m)-(\id \otimes \phi)(u_{i}\ot m\trl v_{i}-u_{i} \otimes  v_{i} \trr m)\\
&=&-u_{i}u_{j}\ot v_{j}\ot v_{i}\trr m-u_{j}\ot [u_{i}, v_{j}]\ot v_{i}\trr m+u_{i}u_{j}\ot v_{j}\ot m \trl v_{i}\\
&&+u_{j}\ot [u_{i}, v_{j}]\ot m\trl v_{i}-u_{i}\ot u_{j}\ot v_{j}\trr(v_{i}\trr m)+u_{i}\ot u_{j}\ot(v_{i}\trr m)\trl v_{j}\\
&&+u_{i}\ot u_{j}\ot v_{j}\trr (m\trl v_{i})-u_{i}\ot u_{j}\ot (m \trl v_{i})\trl v_{j}\\
&=&-u_{i}u_{j}\ot v_{j}\ot v_{i}\trr m-u_{j}\ot u_{i}v_{j}\ot v_{i}\trr m+u_{i}u_{j}\ot v_{i}\ot v_{j}\trr m\\
&&-u_{i}\ot u_{j}\ot [v_{i}, v_{j}]\trr m+u_{i}u_{j}\ot v_{j}\ot m \trl v_{i}+u_{j}\ot u_{i}v_{j}\ot m\trl v_{i}\\
&&-u_{i}u_{j}\ot v_{i}\ot (m\trl v_{j})+u_{i}\ot u_{j}\ot m\trl [v_{i}, v_{j}]-u_{i}\ot u_{j}\ot v_{j}\trr(v_{i}\trr m)\\
&&+u_{i}\ot u_{j}\ot(v_{i}\trr m)\trl v_{j}+u_{i}\ot u_{j}\ot v_{j}\trr (m\trl v_{i})-u_{i}\it u_{j}\ot (m \trl v_{i})\trl v_{j},
\end{eqnarray*}
and the right hand side equal to
\begin{eqnarray*}
&&\tau_{12}\Big((\Delta_{r} \otimes \id)\phi(v)- (\id \otimes \phi)\phi(v)\Big)\\
&=&\tau_{12}\Big( -u_{i}u_{j}\ot v_{j}\ot v_{i}\trr m-u_{j}\ot [u_{i}, v_{j}]\ot v_{i}\trr m+u_{i}u_{j}\ot v_{j}\ot m \trl v_{i}\\
&&+u_{j}\ot [u_{i}, v_{j}]\ot m\trl v_{i}-u_{i}\ot u_{j}\ot v_{j}\trr(v_{i}\trr m)+u_{i}\ot u_{j}\ot(v_{i}\trr m)\trl v_{j}\\
&&+u_{i}\ot u_{j}\ot v_{j}\trr (m\trl v_{i})-u_{i}\it u_{j}\ot (m \trl v_{i})\trl v_{j} \Big)\\
&=&-v_{j}\ot u_{i}u_{j}\ot v_{i}\trr m-[u_{i}, v_{j}]\ot u_{j}\ot v_{i}\trr m+v_{j}\ot u_{i}u_{j}\ot m \trl v_{i}\\
&&-[u_{i}, v_{j}]\ot u_{j}\ot m\trl v_{i}-u_{i}\ot u_{j}\ot v_{i}\trr(v_{j}\trr m)+u_{i}\ot u_{j}\ot(v_{j}\trr m)\trl v_{i}\\
&&+u_{i}\ot u_{j}\ot v_{i}\trr (m\trl v_{j})-u_{i}\it u_{j}\ot (m \trl v_{j})\trl v_{i}.
\end{eqnarray*}
Thus the two sides are equal to each other if and only if  $r$ is symmetric, $[[r, r]]=0$ and $M$ is an $A$-bimodule.

For the second equation, we have the left hand side equal to the right hand side:
\begin{eqnarray*}
&&\left( \psi\otimes \id   \right) \psi(m)-\left(\id \ot\Delta_{r} \right)\psi(m)\\
&=&(\psi\otimes \id )(m\trl u_{i}\ot v_{i})-(\id \ot\Delta_{r} )(m\trl u_{i}\ot v_{i})\\
&=&(m\trl u_{i})\trl u_{j}\ot v_{j}\ot v_{i}-m\trl u_{i}\ot v_{i}u_{j}\ot v_{j}-m\trl u_{i}\ot u_{j}\ot [v_{i}, v_{j}]\\
&=&(m\trl u_{i})\trl u_{j}\ot v_{j}\ot v_{i}-m\trl (u_{i}u_{j})\ot v_{i}\ot v_{j}\\
&=&(m\trl u_{i})\trl u_{j}\ot v_{j}\ot v_{i}-(m\trl u_{i})\trl u_{j}\ot v_{i}\ot v_{j}+(u_{i}\trr m)\trl u_{j}\ot v_{i}\ot v_{j}\\
&&-u_{i}\trr (m\trl u_{j})\ot v_{i}\ot v_{j}\\
&=&(m\trl u_{i})\trl v_{j}\ot u_{j}\ot v_{i}-v_{j}\trr (m\trl u_{i})\ot u_{j}\ot v_{i}-(m\trl v_{i})\trl u_{j}\ot u_{i}\ot v_{j}\\
&&+(v_{i}\trr m)\trl u_{j}\ot u_{i}\ot v_{j}\\
&=&\tau_{12}(\phi(m\trl u_{i})\ot v_{i}-u_{i}\ot \psi(m\trl v_{i})+u_{i}\ot \psi(v_{i}\trr m))\\
&=&\tau_{12}((\phi\ot \id)\psi(m)-(\id\ot \psi)\phi(m))
\end{eqnarray*}

Next, we prove that $M$ is a left-symmetric Hopf bimodule over $A$.
For $(HM1)$, we have
\begin{eqnarray*}
&&v_{(-1)} \otimes\left(v_{(0)} \trl x\right)+x_{1}\ot (v\trl x_{2})-x_{1}\ot (x_{2}\trr v)-xv_{(-1)}\ot v_{(0)}-v_{(-1)}\ot (x \trr v_{(0)})\\
&=&u_{i}\ot (m\trl v_{i})\trl x+xu_{i}\ot m\trl v_{i}+u_{i}\ot m\trl (xv_{i})-u_{i}\ot m\trl (v_{i}x)\\
&&-u_{i}\ot(v_{i}\trr m)\trl x-xu_{i}\ot v_{i}\trr m-u_{i}\ot (xv_{i})\trr m+u_{i}\ot (v_{i}x)\trr m\\
&&-xu_{i}\ot m\trl v_{i}+xu_{i}\ot v_{i}\trr m-u_{i}\ot x\trr(m\trl v_{i})+u_{i}\ot x\trr(v_{i}\trr m)\\
&=&u_{i}\ot (m\trl v_{i})\trl x+u_{i}\ot m\trl (xv_{i})-u_{i}\ot m\trl (v_{i}x)-u_{i}\ot(v_{i}\trr m)\trl x\\
&&-u_{i}\ot (xv_{i})\trr m+u_{i}\ot (v_{i}x)\trr m-u_{i}\ot x\trr(m\trl v_{i})+u_{i}\ot x\trr(v_{i}\trr m)\\
&=&u_{i}\ot (m\trl x)\trl v_{i}-u_{i}\ot v_{i}\trr (m\trl x)-u_{i}\ot (x\trr m)\trl v_{i}\\
&&+u_{i}\ot v_{i}\trr (x\trr m)\\
&=&\phi(m\trl x)-\phi(x\trr m).
\end{eqnarray*}
For $(HM2)$, we have
\begin{eqnarray*}
&&\left(x \trr  m_{(0)}\right) \otimes m_{(1)}+m_{(0)}\otimes [x, m_{(1)}] -(m\trl x_{1})\ot x_{2}\\
&=&x\trr(m\trl u_{i})\ot v_{i}+m\trr u_{i}\ot [x, v_{i}]-m\trr (xu_{i})\ot v_{i}-m\trr u_{i}\ot [x, v_{i}]\\
&=&x\trr(m\trl u_{i})\ot v_{i} -m\trr (xu_{i})\ot v_{i} \\
&=&(x\trr m)\trl u_{i}\ot v_{i}-(m\trl x)\trl u_{i}\ot v_{i}\\
&=&\phi(x\trr m)-\psi(m\trl x).
\end{eqnarray*}
For For $(HM3)$, we have
\begin{eqnarray*}
&&x_{1} \otimes\left(x_{2}\trr m\right)+x m_{(-1)}\otimes m_{(0)}+m_{(-1)}\ot( x\trr m_{(0)})-m_{(1)}\ot (x\trr m_{(0)})-xm_{(1)} \ot m_{(0)}\\
&=&xu_{i}\ot v_{i}\trr m+u_{i}\ot (xv_{i})\trr m-u_{i}\ot (v_{i}x)\trr m+xu_{i}\ot m\trl v_{i}-xu_{i}\ot v_{i}\trr m\\
&&+u_{i}\ot x\trr (m\trl v_{i})-u_{i}\ot x\trr(v_{i}\trr m)-v_{i}\ot x\trr(m\trl u_{i})-xv_{i}\ot m\trl u_{i}\\
&=&u_{i}\ot (xv_{i})\trr m-u_{i}\ot (v_{i}x)\trr m+u_{i}\ot x\trr (m\trl v_{i})-u_{i}\ot x\trr(v_{i}\trr m)-v_{i}\ot x\trr(m\trl u_{i})\\
&=&u_{i}\ot (x\trr m)\trl v_{i}-u_{i}\ot v_{i}\trr(x\trr m)-v_{i}\ot (x\trr m)\trl u_{i}\\
&=&\phi(x\trr m)-\tau\psi(x\trr m).
\end{eqnarray*}
For For $(HM4)$, we have
\begin{eqnarray*}
&&m_{(-1)} \otimes\left(m_{(0)} \trl x\right)+x_{1}\ot (m\trl x_{2})-m_{(1)}x\ot m_{(0)}- x_{2}\ot (m \trl  x_{1})\\
&=&u_{i}\ot(m\trl v_{i})\trl x-u_{i}\ot(v_{i}\trr m)\trl x+xu_{i}\ot m\trl v_{i}+u_{i}\ot m\trl(xv_{i})-u_{i}\ot m\trl(v_{i}x)\\
&&-v_{i}x\ot m\trl u_{i}-v_{i}\ot m\trl(xu_{i})-xv_{i}\ot m\trl u_{i}+v_{i}x\ot m\trl u_{i}\\
&=&u_{i}\ot(m\trl v_{i})\trl x-u_{i}\ot(v_{i}\trr m)\trl x-u_{i}\ot m\trl(v_{i}x)\\
&=&-u_{i}\ot v_{i}\trr(m\trl x)\\
&=&u_{i}\ot(m\trl x)\trl v_{i} -u_{i}\ot v_{i}\trr(m\trl x)-v_{i}\ot(m\trl x)\trl u_{i}\\
&=&\phi(m\trl x)-\tau\psi(m\trl x).
\end{eqnarray*}
This completed the proof.
\end{proof}

\begin{theorem}
 Let $(A, \cdot, \Delta_r)$ be a quasitriangular  left-symmetric bialgebra. Then $A$ becomes a braided  left-symmetric bialgebra over itself with $M=A$ and $\phi: M \rightarrow A \otimes M$ and $\psi: M \rightarrow M \otimes A$ are given by
 \begin{equation}
\phi(a):=\sum_{i} u_{i} \otimes [a, v_{i}], \quad \psi(a):= \sum_{i} au_{i} \otimes v_{i},
\end{equation}
\end{theorem}

\begin{proof}
All we need to do is to verify the braided compatibility conditions $(BB1)$ and $(BB2)$.
For $(BB1)$, we have the right hand side is equal to the left hand side by
\begin{eqnarray*}
&& a b\li\ot b\lii+b\li\ot [a, b\lii]-b a\li\ot a\lii -a\li \ot [b, a\lii]+a_{(0)} \otimes\left(a_{(1)} \trr b\right)\\
&&+(a\trl b_{(-1)}) \otimes b_{(0)}+b_{(0)} \otimes (a\trl b_{(1)})-b_{(0)} \otimes\left(b_{(1)} \trr a\right)-(b\trl a_{(-1)}) \otimes a_{(0)}-a_{(0)} \otimes (b\trl a_{(1)})\\
&=&a(bu_{i})\ot v_{i}+au_{i}\ot [b, v_{i}]+bu_{i}\ot [a, v_{i}]+u_{i}\ot [a,  [b,v_{i}]]-b(au_{i})\ot v_{i}\\
&&-bu_{i}\ot [a, v_{i}]-au_{i}\ot [b, v_{i}]-u_{i}\ot [b, [a, v_{i}]]+au_{i}\ot v_{i}b+au_{i}\ot [b, v_{i}]\\
&&+bu_{i}\ot av_{i}-bu_{i}\ot v_{i}a-bu_{i}\ot [a, v_{i}]-au_{i}\ot bv_{i}\\
&=&a(bu_{i})\ot v_{i}+u_{i}\ot [a, [b, v_{i}]]-b(au_{i})\ot v_{i}-u_{i}\ot [b, [a,  v_{i}]]\\
&=&a(bu_{i})\ot v_{i}-b(au_{i})\ot v_{i}+u_{i}\ot a(bv_{i})+u_{i}\ot (v_{i}b)a-u_{i}\ot b(av_{i})-u_{i}\ot (v_{i}a)b\\
&=&(ab)u_{i}\ot v_{i}-(ba)u_{i}\ot v_{i}+u_{i}\ot (ab)v_{i}-u_{i}\ot (ba)v_{i}-u_{i}\ot v_{i}(ab)+u_{i}\ot v_{i}(ba)\\
&=&[a, b]u_{i}\ot v_{i}+u_{i}\ot[[a, b], v_{i}]\\
&=&\Delta_{r}([a, b]).
\end{eqnarray*}

For  $(BB2)$,  by similar computations,  we obtain
\begin{eqnarray*}
&&(\id-\tau)\Big(a_{1} \otimes a_{2} b+a b_{1} \otimes b_{2}+ b_{1} \otimes a b_{2}\\
&&+a_{(0)} \otimes\left(a_{(1)} \trr b\right)+(a\trl b_{(-1)}) \otimes b_{(0)}+b_{(0)} \otimes (a\trl b_{(1)})\Big)\\
&=&(\id-\tau)\Big(au_{i}\ot v_{i}b+u_{i}\ot (av_{i})b- u_{i}\ot (v_{i}a)b+a(bu_{i})\ot v_{i}+au_{i}\ot bv_{i} \\
&&-au_{i}\ot v_{i}b+bu_{i}\ot av_{i}+u_{i}\ot a(bv_{i})-u_{i}\ot a(v_{i}b) +au_{i}\ot v_{i}b+au_{i}\ot [b, v_{i}]+bu_{i}\ot av_{i}\Big)\\
&=&(\id-\tau)\Big(u_{i}\ot (av_{i})b- u_{i}\ot (v_{i}a)b+a(bu_{i})\ot v_{i}+u_{i}\ot a(bv_{i})-u_{i}\ot a(v_{i}b)\Big)\\
&=&(\id-\tau)(-u_{i}\ot v_{i}(ab))\\
&=&(\id-\tau)\Big((ab)u_{i}+u_{i}\ot[ab, v_{i}]\Big)\\
&=&(\id-\tau)\Delta_{r}(ab).
\end{eqnarray*}
Thus $(BB1)$ and $(BB2)$ holds. This completed the proof.
\end{proof}

\subsection{From braided left-symmetric bialgebra to braided Lie bialgebra}
We know that  one can get a Lie bialgebra from a left-symmetric bialgebra by redefine their product and coproduct in \cite{Bai08}. The proof process can be rewritten as follows.
\begin{theorem}(\cite{Bai08})
Let $(A, \cdot, \Delta)$ be a left-symmetric bialgebra. Then $(A, [\cdot, \cdot], \delta)$ is a Lie bialgebra with bracket $[a, b]=ab-ba$ and cobracket $\delta(a)=(\Delta-\tau\Delta)(a)$ for any $a\in A$ if and only if
\begin{eqnarray}
a\li b\ot a\lii-a\lii\ot a\li b=b\li a\ot b\lii-b\lii\ot b\li a.
\end{eqnarray}
\end{theorem}
Next,  we will show that we can get a braided Lie bialgebra from a braided left-symmetric bialgebra, before which we recall the definition of braided Lie bialgebra.

\newpage
\begin{definition}(\cite{So96,Ma00})
Let $H$ be  a Lie bialgebra.
 If $V$ is  a left $H$-module and left $H$-comodule, satisfying the following condtion:
\begin{eqnarray}\label{lieydm}
\phi_L (x\trr_L v)= [x, v\boi]\ot v \boo + v\boi \ot x\trr_L v\boo + x_{[1]} \ot x_{[2]}\trr_L v,
\end{eqnarray}
then $V$ is called a left Yetter-Drinfeld module over $H$.
\end{definition}
We denote  the  category of Yetter-Drinfeld modules over $H$ by
${}^{H}_{H}\mathcal{M}$.

\begin{definition}(\cite{So96,Ma00})
Let $H$ be a  Lie bialgebra,  $A$  a left Yetter-Drinfeld module over $H$. We call $A$ a \emph{braided Lie
bialgebra} in ${}^{H}_{H}\mathcal{M}^{ }_{ }$ if the following condition is
satisfied:
\begin{eqnarray}\label{braidedliebialgebra}
\delta ( [a, b ] )&=& [ a, b_{[1]} ] \otimes  b_{[2]}+b_{[1]}  \otimes [ a, b_{[2]}]+ [a_{[1]} ,  b] \otimes  a_{[2]}+ a_{[1]} \otimes [a_{[2]},  b]\\
 \notag&&+a_{[-1]} \trr_L  b \otimes  a_{[0]}+ b_{[0]}\otimes  b_{[-1]} \trr_L  a-  b_{[-1]} \trr   a \otimes  b_{[0]} - a_{[0]}\otimes  a_{[-1]}\trr_L  b,
\end{eqnarray}
where we denote the Lie  cobracket by $\delta(a)=a\bi\ot a\bii$ and left  comodule by $\phi_L(a)=a\boi\ot a\boo$.
\end{definition}

\begin{theorem}
Let $H$ be a left-symmetric bialgebra. If $(A,  \cdot,  \Delta)$ is a braided left-symmetric bialgebra in ${}^{H}_{H}\mathcal{M}{}^{H}_{H}$. Define the bracket by $[a, b]=ab-ba$ and cobracket by $\delta(a)=(\Delta-\tau\Delta)(a)$.
Then   $(A, [\cdot, \cdot], \delta)$ is a braided Lie bialgebra if and only if
\begin{eqnarray}\label{braided01}
x_2 \otimes x_1 \trr a=a_{(-1)} \otimes \otimes a_{(0)}-a_{(1)} \otimes a_{(0)} \trl x,\\
\notag a\li b\ot a\lii-a\lii\ot a\li b+(a\loi\trr b)\ot a\loo-a\loo\ot (a\loi \trr b)\\
=b\li a\ot b\lii-b\lii\ot b\li a+(b\loi\trr b)\ot b\loo-b\loo\ot (b\loi\trr b).\label{braided02}
\end{eqnarray}
\end{theorem}

\begin{proof}

In order to prove that $A$ is a braided Lie bialgebra over $H$,  we  define the left $H$-module and  left $H$-comodule by:
\begin{align*}
&\trr_{L}=\trr-\trl: H\ot A \to A, \quad\phi_{L}=\phi-\tau\psi:A\to H\ot A,
\end{align*}
that is
\begin{align*}
&x\trr_{L} a=x\trr a- a\trl x, ~~~~\phi_{L}(a)=a_{[-1]}\ot a_{[0]}=a_{(-1)}\ot a_{(0)}-a_{(1)}\ot a_{(0)},
\end{align*}
 for any $a\in A$.

First we prove that $A$ is a Yetter-Drinfeld modules over $H$.
We compute as follows:
\begin{eqnarray*}
&&[x, {a}\boi]\ot {a}\boo+{a}\boi\ot (x\trr_{L} {a}\boo)+x\bi\ot (x\bii\trr_{L} {a}) \\
&=&[x, {a}\loi]\ot {a}\loo-[x, {a}_{(1)}]\ot {a}\loo+{a}\loi\ot (x\trr_{L} {a}\loo)-{a}_{(1)}\ot (x\trr_{L} {a}\loo)\\
&&+x\li\ot (x\lii\trr_{L} {a})-x\lii\ot (x\li\trr_{L} {a})\\
&=&x {a}\loi\ot {a}\loo-{a}\loi x\ot {a}\loo-x {a}_{(1)}\ot {a}\loo+{a}_{(1)} x\ot {a}\loo\\
&&+{a}\loi\ot (x\trr {a}\loo)-{a}\loi\ot (  {a}\loo\trl x)-{a}_{(1)}\ot (x\trr {a}\loo)+{a}_{(1)}\ot ( {a}\loo\trl x)\\
&&+x\li\ot (x\lii\trr {a})-x\li\ot ({a} \trl x\lii)-x\lii\ot (x\li\trr {a})+x\lii\ot ({a}\trl x\li )\\
&=&x_{1} \otimes\left(x_{2}\trr {a}\right)+x {a}_{(-1)}\otimes {a}_{(0)}+{a}_{(-1)}\ot( x\trr {a}_{(0)})-{a}_{(1)}\ot (x\trr {a}_{(0)})-x{a}_{(1)} \ot {a}_{(0)}\\
&&-{a}_{(-1)} \otimes\left({a}_{(0)} \trl x\right)-x_{1}\ot ({a}\trl x_{2})+{a}_{(1)}x\ot {a}_{(0)}+ x_{2}\ot ({a} \trl  x_{1})\\
&&-{a}\loi x\ot {a}\loo +{a}_{(1)}\ot ( {a}\loo\trl x)-x\lii\ot (x\li\trr {a})\\
&=&\phi(x \trr {a})-\tau\psi(x\trr {a})-\phi({a} \trl x)+\tau\psi({a} \trl x )\\
&=&\phi_{L}(x \trr {a})-\phi_{L}({a} \trl x)\\
&=&\phi_{L}(x \trr_{L} {a}),
\end{eqnarray*}
where we use the conditions  \eqref{braided01}, (HM3) and  (HM4) in the fourth equality.

Next we check that the condition \eqref{braidedliebialgebra} holds.
By direct calculation, we have
\begin{eqnarray*}
&&[ a, b_{(1)} ] \otimes  b_{(2)}+b_{(1)}\otimes [ a, b_{(2)}]+ [a_{(1)},  b] \otimes  a_{(2)}+ a_{(1)} \otimes [a_{(2)},  b]\\
&&+a_{[-1]} \trr_{L}  b \otimes  a_{[0]}+ b_{[0]}\otimes  b_{[-1]} \trr_{L}  a-  b_{[-]} \trr_{L}   a \otimes  b_{[0]} - a_{[0]}\otimes  a_{[-1]}\trr_{L}  b\\
&=&ab_{(1)}\ot b_{(2)}-b_{(1)}a\ot b_{(2)}+b_{(1)}\ot ab_{(2)}-b_{(1)}\ot b_{(2)}a+a_{(1)}b\ot a_{(2)}- b a_{(1)}\ot a_{(2)}\\
&&+a_{(1)} \otimes a_{(2)}b-a_{(1)} \otimes ba_{(2)}+a_{(-1)} \trr_{L}  b \otimes  a_{(0)}-a_{(1)} \trr_{L}  b \otimes  a_{(0)}
+ b_{(0)}\otimes  b_{(-1)} \trr_{L}  a\\
&&-b_{(0)}\otimes  b_{(1)} \trr_{L}  a-  b_{(-1)} \trr_{L}   a \otimes  b_{(0)}+ b_{(1)} \trr_{L}   a \otimes  b_{(0)} - a_{(0)}\otimes  a_{(-1)}\trr_{L}  b+a_{(0)}\otimes  a_{(1)}\trr_{L}  b\\
&=&ab\li\ot b\lii-ab\lii\ot b\li-b\li a\ot b\lii+b\lii a\ot b\li+b_{1}\ot ab_{2}-b_{2}\ot ab_{1}-b_{1}\ot b_{2}a\\
&&+b_{2}\ot b_{1}a+a_{1}b\ot a_{2}-a_{2}b\ot a_{1}- b a_{1}\ot a_{2}+b a_{2}\ot a_{1}+a_{1} \otimes a_{2}b-a_{2} \otimes a_{1}b\\
&&-a_{1} \otimes ba_{2}+a_{2} \otimes ba_{1}+a_{(-1)} \trr  b \otimes  a_{(0)}-b\trl a_{(-1)} \otimes  a_{(0)}-a_{(1)}\trr b \otimes  a_{(0)}\\
&&+b\trl a_{(1)} \otimes  a_{(0)}
+ b_{(0)}\otimes  b_{(-1)} \trr a-b_{(0)}\otimes a\trl b_{(-1)}-b_{(0)}\otimes  b_{(1)} \trr  a +b_{(0)}\otimes a\trl b_{(1)}\\
&&-b_{(-1)} \trr  a \otimes  b_{(0)}+a\trl b_{(-1)} \otimes  b_{(0)}+ b_{(1)} \trr a \otimes  b_{(0)}-a\trl b_{(1)}\otimes  b_{(0)}\\
&&- a_{(0)}\otimes  a_{(-1)}\trr b+a_{(0)}\otimes  b\trl a_{(-1)}+a_{(0)}\otimes  a_{(1)}\trr   b-a_{(0)}\otimes b\trl a_{(1)}\\
&=&\Delta(ab)-\Delta(ba)-\tau\Delta(ab)+\tau\Delta(ba)+a\li b\ot a\lii-a\lii\ot a\li b-b\li a\ot b\lii+b\lii\ot b\li a\\
&&+(a\loi\trr b)\ot a\loo-a\loo\ot (a\loi \trr b)-(b\loi\trr b)\ot b\loo+b\loo\ot (b\loi\trr b)\\
&=&\Delta(ab)-\Delta(ba)-\tau\Delta(ab)+\tau\Delta(ba)\\
&=&\delta(ab-ba)=\delta ( [a, b ] ),
\end{eqnarray*}
where we use condition \eqref{braided02} in the fourth equality. The proof is completed.
\end{proof}

\section{Unified product of  left-symmetric bialgebras}
\subsection{Matched pair of braided   left-symmetric bialgebras}

In this section, we construct   the double cross biproduct from a matched pair of braided  left-symmetric bialgebras.

Let $A, H$ be both left-symmetric  algebras and  left-symmetric coalgebras.   For any $a, b\in A$, $x, y\in H$,  we denote maps
\begin{align*}
&\ppr: H \otimes A \to A,\quad \ppl: A\otimes H\to A,\\
&\trr: A\otimes H \to H,\quad \trl: H\otimes A \to H,\\
&\phi: A \to H \otimes A,\quad  \psi: A \to A\otimes H,\\
&\rho: H  \to A\otimes H,\quad  \gamma: H \to H \otimes A,
\end{align*}
by
\begin{eqnarray*}
&& \ppr (x \otimes a) = x \ppr a, \quad \ppl(a\otimes x) = a \ppl x, \\
&& \trr (a \otimes x) = a \trr x , \quad \trl(x \otimes a) = x \triangleleft a, \\
&& \phi (a)=\sum a\lmoi\ot a\loo, \quad \psi (a) = \sum a\loo \ot a\lmi,\\
&& \rho (x)=\sum x\boi\ot x\boo, \quad \gamma (x) = \sum x\boo \ot x\bi.
\end{eqnarray*}

\begin{definition}\cite{Bai08}
A \emph{matched pair} of left-symmetric  algebras is a system $(A, \, {H},\, \trl, \, \trr, \, \ppl, \, \ppr)$ consisting of two left-symmetric  algebras $A$ and ${H}$ and four bilinear maps $\triangleleft: {H}\otimes A\to {H}$, $\trr: {A} \otimes H
\to H$, $\ppl: A \otimes {H} \to A$, $\ppr: H\otimes {A} \to {A}$ such that $({H}, \, \trr, \, \trl)$ is an $A$-bimodule,  $(A, \, \ppr, \, \ppl)$ is an ${H}$-bimodule and satisfying the following compatibility conditions:
\begin{enumerate}
\item[(AM1)] $x \ppr(a b) =(x\ppr a-a\ppl x) b+(x\trl a-a\trr x) \ppr b+a(x\ppr b)+a\ppl(x\trl b)$,
\item[(AM2)] $[a,b]\ppl x =a(b\ppl x)+a\ppl (b \trr x)-b(a\ppl x)-b\ppl(a\trr x)$.
\item[(AM3)] $ a \trr (x y) =(a \trr x-x\trl a) y+(a\ppl x-x\ppr a) \trr y+x\trl (a\ppl y)+x(a\trr y)$,
\item[(AM4)] $[x, y] \trl a=x \trl(y \ppr a)-y \trl(x \ppr a)+x(y \trl a)-y(x \trl a)$,
\end{enumerate}
\end{definition}

\begin{lemma}\cite{Bai08}
Let  $(A, \, {H},\, \trl, \, \trr, \, \ppl, \, \ppr)$ be a matched pair of  left-symmetric   algebras.
Then $A \, \bowtie {H}:= A \oplus  {H}$, as a vector space, with the product defined for any $a, b\in A$ and $x, y\in {H}$ by
\begin{equation}
(a, x) (b, y) := \big(ab+ a \ppl y + x\ppr b,\, \, a\trr y + x\trl b + xy \big),
\end{equation}
is a   left-symmetric  algebra called the \emph{bicrossed product} associated to the matched pair of   left-symmetric  algebras $A$ and ${H}$.
\end{lemma}

Now we introduce the notion of matched pairs of  left-symmetric  coalgebras, which is the dual version of  matched pairs of  left-symmetric algebras.

\begin{theorem}
Let $A, H$ be both left-symmetric coalgebras, and there be four bilinear maps
$\phi: {A}\to H\otimes A$, $\psi: {A}\to A \otimes H$, $\rho: H\to A\otimes {H}$, $\gamma: H \to {H} \ot {A}$.
 We define $E=A\lrcoprod H$ as the vector space $A\oplus H$ with   coproduct
$$\Delta_{E}(a)=(\Delta_{A}+\phi+\psi)(a),\quad i.e. ~~\Delta_{E}(a)=\sum a\li \ot a\lii+\sum a\loi \ot a\loo+\sum a\mo\ot a\mi;$$
$$ \Delta_{E}(x)=(\Delta_{H}+\rho+\gamma)(x),\quad i.e. ~~\Delta_{E}(x)=\sum x\li \ot x\lii+\sum  x\boi \ot x\boo+\sum x\boo \ot x\bi.$$
Then  $A\lrcoprod H$ is a   left-symmetric  coalgebra if and only if $(A, \, {H}, \, \phi, \, \psi, \, \rho, \, \gamma)$ satisfies the following compatibility conditions for any $a\in A$, $x\in {H}$:
\begin{enumerate}
\item[(MC1)] $\phi\left(a_{1}\right) \otimes a_{2}+\gamma\left(a_{(-1)}\right) \otimes a_{(0)}-a_{(-1)} \otimes \Delta_{A}\left(a_{(0)}\right)\\
    =\tau_{12}\left(\psi\left(a_{1}\right) \otimes a_{2}+\rho\left(a_{(-1)}\right) \otimes a_{(0)}-a_{1} \otimes \phi(a_{2})-a_{(0)} \otimes \gamma(a_{(1)})\right)$,
\item[(MC2)] $\Delta_{A}\left(a_{(0)}\right) \otimes a_{(1)} -a_{1} \otimes \psi\left(a_{2}\right)-a_{(0)} \otimes \rho\left(a_{(1)}\right)\\
    =\tau_{12}\left(\Delta_{A}\left(a_{(0)}\right) \otimes a_{(1)} -a_{1} \otimes \psi\left(a_{2}\right)-a_{(0)} \otimes \rho\left(a_{(1)}\right)\right)$,
\item[(MC3)] $\rho\left(x_{1}\right) \otimes x_{2}+\psi\left(x_{[-1]}\right) \otimes x_{[0]}-x_{[-1]} \otimes \Delta_{H}\left( x_{[0]}\right)  \\
    =\tau_{12}\left(\gamma\left(x_{1}\right) \otimes x_{2}+\phi\left(x_{[-1]}\right) \otimes x_{[0]}-x_{[0]} \otimes \psi\left(x_{[1]}\right)-x_{1} \otimes \rho(x_{2})\right)$,
\item[(MC4)] $\Delta_{H}(x\boo)\ot x\bi-x\boo\ot \phi(x\bi)-x_1\ot \gamma(x_2)\\
=\tau_{12}\left(\Delta_{H}(x\boo)\ot x\bi-x\boo\ot \phi(x\bi)-x_1\ot \gamma(x_2) \right)$,
\end{enumerate}
We call $(A, H, \phi, \psi, \rho, \gamma)$ satisfying these conditions the \emph{matched pair} of  left-symmetric  coalgebras and $A\lrcoprod H$ is called the \emph{bicrossed coproduct} associated to the matched pair of   left-symmetric  coalgebras $A$ and $H$.
\end{theorem}

\begin{proof} The proof is by direct computations. We need to prove that
 $(\Delta_{E} \otimes \id) \Delta_{E}(a, x)-(\id \otimes \Delta_{E}) \Delta_{E}(a, x)=\tau_{12}\left((\Delta_{E} \otimes \id) \Delta_{E}(a, x)-(\id \otimes \Delta_{E}) \Delta_{E}(a, x) \right)$. The left hand side is equal to
\begin{eqnarray*}
& &(\Delta_{E} \otimes \id) \Delta_{E}(a, x)-(\id \otimes \Delta_{E}) \Delta_{E}(a, x)\\
&=&(\Delta_{E}\otimes \id)\left(a_{1} \otimes a_{2}+a_{(-1)} \otimes a_{(0)}+a_{(0)} \otimes a_{(1)}+x_{1} \otimes x_{2}+x_{[-1]} \otimes x_{[0]}+x_{[0]} \otimes x_{[1]}\right)\\
&&-(\id \otimes \Delta_{E})\left(a_{1} \otimes a_{2}+a_{(-1)} \otimes a_{(0)}+a_{(0)} \otimes a_{(1)}+x_{1} \otimes x_{2}+x_{[-1]} \otimes x_{[0]}+x_{[0]} \otimes x_{[1]}\right)\\
&=&\Delta_{A}\left(a_{1}\right) \otimes a_{2}+\phi\left(a_{1}\right) \otimes a_{2}+\psi\left(a_{1}\right) \otimes a_{2}+\Delta_{H}\left(a_{(-1)}\right) \otimes a_{(0)}+\rho\left(a_{(-1)}\right) \otimes a_{(0)}\\
&&+\gamma\left(a_{(-1)}\right) \otimes a_{(0)}+\Delta_{A}\left(a_{(0)}\right) \otimes a_{(1)}+\phi\left(a_{(0)}\right) \otimes a_{(1)}+\psi\left(a_{(0)}\right) \otimes a_{(1)}+\Delta_{H}\left(x_{1}\right) \otimes x_{2}\\
&&+\rho\left(x_{1}\right) \otimes x_{2}+\gamma\left(x_{1}\right) \otimes x_{2}+\Delta_{A}\left(x_{[-1])}\right) \otimes x_{[0]}+\phi\left(x_{[-1]}\right) \otimes x_{[0]}+\psi\left(x_{[-1]}\right) \otimes x_{[0]}\\
&& +\Delta_{H}\left(x_{[0]}\right) \otimes x_{[1]}+\rho\left(x_{[0]}\right) \otimes x_{[1]}+\gamma\left(x_{[0]}\right) \otimes x_{[1]}\\
&& -a_{1} \otimes \Delta_{A}\left(a_{2}\right)-a_{1} \otimes \phi\left(a_{2}\right)-a_{1} \otimes \psi\left(a_{2}\right)
 -a_{(1)} \otimes \Delta_{A}\left(a_{(0)}\right)-a_{(-1)} \otimes \phi\left(a_{(0)}\right)\\
 &&-a_{(-1)} \otimes \psi\left(a_{(0)}\right) -a_{(0)} \otimes \Delta_{H}\left(a_{(1)}\right)-a_{(0)} \otimes \rho\left(a_{(1)}\right)-a_{(0)} \otimes \gamma\left(a_{(1)}\right)-x_{1} \otimes \Delta_{H}\left(x_{2}\right)\\
 &&-x_{1} \otimes \rho\left(x_{2}\right)-x_{1} \otimes \gamma\left(x_{2}\right) -x_{[-1]} \otimes \Delta_{H}\left(x_{[0]}\right)-x_{[-1]} \otimes \rho\left(x_{[0]}\right)-x_{[-1]} \otimes \gamma\left(x_{[0]}\right)\\
 &&-x_{[0]} \otimes \Delta_{A}\left(x_{[1]}\right)-x_{[0]} \otimes \phi\left(x_{[1]}\right)-x_{[0]} \otimes \psi\left(x_{[1]}\right).
\end{eqnarray*}
and the right hand side can be computed similarly. Thus the two sides are equal to each other  if and only if $(A, \, {H}, \, \phi, \, \psi, \, \rho, \, \gamma)$ is a matched pair of  left-symmetric   coalgebras.
\end{proof}

In the following of this section, we construct  left-symmetric bialgebra from the double cross biproduct of a pair of braided left-symmetric bialgebras.
First we generalize the concept of Hopf bimodule to the case of $A$ is not necessarily a left-symmetric bialgebra.
But by abuse of notation, we also call it Hopf bimodule.



\begin{definition}
Let $A$ be simultaneously a left-symmetric  algebra and a  left-symmetric coalgebra.
A left-symmetric Hopf bimodule over $A$ is a space $V$ endowed with maps
\begin{align*}
&\ppr: A \otimes V \to V,\quad \ppl: V\otimes A\to V,\\
&\rho: V  \to A\otimes V,\quad  \gamma: V \to V\otimes A,
\end{align*}
such that $V$ is simultaneously a bimodule, a  bicomodule over $A$ and satisfying
 the following compatibility conditions
 \begin{enumerate}
\item[(HM1')] $\gamma(a \ppr v)-\gamma(v\ppl a)=\left(a \ppr v_{[0]}\right) \otimes v_{[1]}+v_{[0]}\otimes [a, v_{[1]}] -(v\ppl a_{1})\ot a_{2},$
\item[(HM2')]$\rho(v \ppl a)-\rho(a\ppr v)$\\
$=v_{[-1]} \otimes\left(v_{[0]} \ppl a\right)+a_{1}\ot (v\ppl a_{2})-a_{1}\ot (a_{2}\ppr v)-av_{[-1]}\ot v_{[0]}-v_{[-1]}\ot (a \ppr v_{[0]})$,
\item[(HM3')] $\rho(a \ppr v)-\tau\gamma(a\ppr v)$\\
$=a_{1} \otimes\left(a_{2}\ppr v\right)+a v_{[-1]}\otimes v_{[0]}+v_{[-1]}\ot( a\ppr v_{[0]})-v_{[1]}\ot( a\ppr v_{[0]})-av_{[1]} \ot v_{[0]}$,
\item[(HM4')] $\rho(v \ppl a)-\tau\gamma(v \ppl a ) =v_{[-1]} \otimes\left(v_{[0]} \ppl a\right)+a_{1}\ot (v\ppl a_{2})-v_{[1]}a\ot v_{[0]}- a_{2}\ot (v \ppl a_{1})$,
\end{enumerate}

then $V$ is called a left-symmetric Hopf bimodule over $A$.
\end{definition}
We denote  the  category of  left-symmetric Hopf bimodules over $A$ by ${}^{A}_{A}\mathcal{M}{}^{A}_{A}$.
\begin{definition}
Let $A$ be a left-symmetric algebra and left-symmetric  coalgebra,  $H$  a left-symmetric Hopf bimodule over $A$. If $H$ is a left-symmetric algebra and a left-symmetric coalgebra in ${}^{A}_{A}\mathcal{M}^{A}_{A}$, then we call $H$ a \emph{braided left-symmetric bialgebra} over $A$  if the following conditions are satisfied:

\begin{enumerate}
\item[(BB3)] $\Delta_{H}([x,y]) $\\
$=x_{1} \otimes [x_{2} ,y]+x y_{1} \otimes y_{2}-y_{1} \otimes [y_{2}, x]-y x_{1} \otimes x_{2}+x_{[0]} \otimes\left(x_{[1]} \trr y\right)$\\
 $+\left(x \trl y_{[-1]}\right) \otimes y_{[0]}+ y_{[0]} \otimes \left(x \trl y_{[1]}\right)-y_{[0]} \otimes\left(y_{[1]} \trr x\right)-\left(y \trl x_{[-1]}\right) \otimes x_{[0]}- x_{[0]} \otimes \left(y \trl x_{[1]}\right)$,
\end{enumerate}
\begin{enumerate}
\item[(BB4)] $(\id-\tau)\left(\Delta_{H}(xy) \right)$\\
$=(\id-\tau)\Big(x_{1} \otimes x_{2} y+x y_{1} \otimes y_{2}+ y_{1} \otimes x y_{2}+x_{[0]} \otimes\left(x_{[1]} \trr y\right)+\left(x \trl y_{[-1]}\right) \otimes y_{[0]}+ y_{[0]} \otimes \left(x \trl y_{[1]}\right)\Big)$.
\end{enumerate}
\end{definition}

\begin{definition}\label{def:dmp}
Let $A, H$ be both  left-symmetric algebras and  left-symmetric coalgebras. If  the following conditions hold:
\begin{enumerate}
\item[(DM1)]  $\phi([a, b])  =a_{(-1)} \otimes[a_{(0)}, b]+b_{(-1)}\ot [a, b_{(0)}]+(a \trr b_{(-1)}) \otimes b_{(0)}-(b \trr a_{(-1)}) \otimes a_{(0)} $,

\item[(DM2)] $\psi([a, b])  $\\
$=a b_{(0)} \otimes b_{(1)}+a_{(0)} \otimes\left(a_{(1)} \trl b\right)  +b_{(0)}\ot (a\trr b_{(1)} )$\\
$-b a_{(0)} \otimes a_{(1)}-b_{(0)} \otimes\left(b_{(1)} \trl a\right) -a_{(0)}\ot (b\trr a_{(1)}) $,

\item[(DM3)] $\rho([x, y]) =x_{[-1]} \otimes [x_{[0]}, y]+y_{[-1]}\ot [x, y_{[0]}]+\left(x \ppr y_{[-1]}\right) \otimes y_{[0]}-\left(y \ppr x_{[-1]}\right) \otimes x_{[0]} $,

\item[(DM4)] $\gamma([x, y])  $\\
$=x_{[0]}\otimes (x_{[1]}\ppl y)+xy_{[0]}\otimes y_{[1]} +y\poo\ot (x\ppr y\bi)$\\
$-y_{[0]}\otimes (y_{[1]}\ppl x)-yx_{[0]}\otimes x_{[1]} -x\poo\ot (y\ppr x\bi)$,

\item[(DM5)] $\Delta_{A}(x \ppr b)-\Delta_{A}(b \ppl x) $ \\
$=x_{[-1]} \otimes\left(x_{[0]} \ppr b\right)+\left(x \ppr b_{1}\right) \otimes b_{2} +b\li \ot (x\ppr b\lii )$\\
$-b_{1} \otimes\left(b_{2} \ppl x\right)-\left(b\ppl x_{[0]}\right) \otimes x_{[1]} -x\boi \ot (b\ppl x\boo) $,

\item[(DM6)] $\Delta_{H}(x \trl b)-\Delta_{H}(b \trr x) $\\
$=x_{1} \otimes\left(x_{2}  \trl b\right)+\left(x\trl b_{(0)}\right) \otimes b_{(1)} +b_{(-1)}\ot (x\trl b\loo)$\\
$-b_{(-1)} \otimes\left(b_{(0)}\trr x\right)-\left(b \trr x_{1}\right) \otimes x_{2}  -x\li\ot (b\trr x\lii) $,

\item[(DM7)]
 $\phi(x \ppr b)-\phi(b \ppl x)+\gamma(x\trl b)-\gamma(b\trr x)$\\
 $=x_{1} \otimes\left(x_{2} \ppr b\right)+x b_{(-1)} \otimes b_{(0)}+\left(x\trl b_{1}\right) \otimes b_{2}+b_{(-1)}\ot (x\ppr b\loo) $\\
$-b_{(-1)} \otimes\left(b_{(0)}\ppl x\right) -\left(b\trr x_{[0]}\right)\otimes x_{[1]}-x\li  \ot (b\ppl x\lii)- x\poo\ot [b,x\bi]$,

\item[(DM8)]
$\psi(x \ppr b)-\psi(b \ppl x)+\rho(x\trl b)-\rho(b\trr x)$\\
$=(x \ppr b_{(0)}) \otimes b_{(1)} +x_{[-1]} \otimes\left(x_{[0]} \trl b\right)+b\li\ot (x\trl b\lii)+b\loo\ot [x,b\lmi]$\\
$ -\left(b\ppl x_{1}\right) \otimes x_{2}-b_{1} \otimes\left(b_{2} \trr x\right)-b x_{[-1]} \otimes x_{[0]}-x\boi \ot (b\trr x\poo) $,

\item[(DM9)]  $\phi(ab) -\tau\psi(a b) $\\
$=a_{(-1)} \otimes a_{(0)} b +(a \trr b_{(-1)}) \otimes b_{(0)} +b_{(-1)}\ot a b_{(0)} $\\
$-\tau\Big(a b_{(0)} \otimes b_{(1)}+ a_{(0)} \otimes\left(a_{(1)} \trl b\right)+ b_{(0)}\ot (a\trr b_{(1)})\Big) $,

\item[(DM10)] $\rho(x y)  -\tau\gamma(x y)  $\\
$=x_{[-1]} \otimes x_{[0]} y +\left(x \ppr y_{[-1]}\right) \otimes y_{[0]}+y_{[-1]}\ot xy_{[0]} $\\
$-\tau\Big(x_{[0]}\otimes (x_{[1]}\ppl y)+xy_{[0]}\otimes y_{[1]} +y\poo\ot (x\ppr y\bi)\Big)$,

\item[(DM11)]
 $\phi(x \ppr b) +\gamma(x\trl b) -\tau\psi(x \ppr b) -\tau\rho(x\trl b) $\\
 $=x_{1} \otimes\left(x_{2} \ppr b\right)+x b_{(-1)} \otimes b_{(0)}+x_{[0]} \otimes x_{[1]} b+\left(x\trl b_{1}\right) \otimes b_{2}+b_{(-1)}\ot (x\ppr b\loo) $\\
$-\tau\Big((x \ppr b_{(0)}) \otimes b_{(1)} +x_{[-1]} \otimes\left(x_{[0]} \trl b\right) +b\li\ot (x\trl b\lii)+b\loo\ot xb\lmi\Big)$,

\item[(DM12)]$\psi(a\ppl y) +\rho(a \trr y)-\tau \phi(a\ppl y) -\tau\gamma(a \trr y)$\\
$=a_{(0)} \otimes a_{(1)} y+\left(a\ppl y_{1}\right) \otimes y_{2}+a_{1} \otimes\left(a_{2} \trr y\right)+a y_{[-1]} \otimes y_{[0]}+y\boi \ot (a\trr y\poo) $\\
$-\tau\Big(a_{(-1)} \otimes\left(a_{(0)}\ppl y\right) +\left(a\trr y_{[0]}\right)\otimes y_{[1]} +y\li \ot( a\ppl y\lii)+ y\poo\ot ay\bi\Big)$,

\item[(DM13)] $(\id-\tau)(\Delta_{A}(x \ppr b)  )
 =(\id-\tau)\Big(x_{[-1]} \otimes\left(x_{[0]} \ppr b\right)+\left(x \ppr b_{1}\right) \otimes b_{2}  +b\li \ot (x\ppr b\lii) $\Big),

\item[(DM14)] $(\id-\tau)(\Delta_{A}(a\ppl y))
 =(\id-\tau)\Big(a_{1} \otimes\left(a_{2} \ppl y\right)+\left(a\ppl y_{[0]}\right) \otimes y_{[1]} +y\boi \ot (a\ppl y\boo) $\Big),

\item[(DM15)] $(\id-\tau)(\Delta_{H}(a \trr y))
 =(\id-\tau)\Big(a_{(-1)} \otimes\left(a_{(0)}\trr y\right)+\left(a \trr y_{1}\right) \otimes y_{2} +y\li\ot (a\trr y\lii)\Big),$

\item[(DM16)] $(\id-\tau)(\Delta_{H}(x \trl b))
=(\id-\tau)\Big(x_{1} \otimes\left(x_{2}  \trl b\right)+\left(x\trl b_{(0)}\right) \otimes b_{(1)} +b_{(-1)}\ot (x\trl b\loo)\Big),$
\end{enumerate}
\noindent then $(A, H)$ is called a \emph{double matched pair}.
\end{definition}

\begin{theorem}\label{main1}Let $(A, H, \trr, \trl, \ppr, \ppl)$ and $(A, H, \phi, \psi, \rho, \gamma)$ be a matched pair of left-symmetric algebras and coalgebras,
$A$ be  a  braided    left-symmetric bialgebra in
${}^{H}_{H}\mathcal{M}^{H}_{H}$ and $H$ be  a braided    left-symmetric bialgebra in
${}^{A}_{A}\mathcal{M}^{A}_{A}$. If we define the double cross biproduct of $A$ and
$H$, denoted by $A\lrbiprod H$, $A\lrbiprod H=A\bowtie H$ as left-symmetric
algebra, $A\lrbiprod H=A\lrcoprod H$ as  left-symmetric coalgebra, then
  $A\lrbiprod H$ becomes a   left-symmetric bialgebra if and only if  $(A, H)$ forms a double matched pair.
\end{theorem}

\begin{proof} Simply,  we check the first compatibility condition $\Delta([(a, x), (b, y)])=\Delta(a, x) \cdot(b, y)+(a, x) \cdot \Delta(b, y)+(a, x) \bullet \Delta(b, y)-\Delta(b, y ) \cdot(a, x)-(b, y ) \cdot \Delta(a, x)-( b, y) \bullet \Delta(a, x)$.
The left hand side is equal to
\begin{eqnarray*}
&&\Delta([(a, x), (b, y)])\\
&=&\Delta((a, x) (b, y))-\Delta((b, y) (a, x))\\
&=&\Delta(a b+x \ppr b+a\ppl y , x y+x \trl b+a\trr y )\\
&&-\Delta(ba+y \ppr a+b\ppl x , yx+y \trl a+b\trr x )\\
&=&\Delta_A(a b)+\phi(a b)+\psi(a b)+\Delta_A(x \ppr b)+\phi(x \ppr b)+\psi(x \ppr b)\\
&&+\Delta_{A}(a\ppl y)+\phi(a\ppl y)+\psi(a\ppl y)+\Delta_{H}(x y)+\rho(x y)+\gamma(x y)\\
&&+\Delta_{H}(x \trl b)+\rho(x \trl b)+\gamma(x \trl b)+\Delta_{H}(a \trr y)+\rho(a \trr y)+\gamma(a \trr y)\\
&&-\Delta_A(ba)-\phi(ba)-\psi(ba)-\Delta_A(y \ppr a)-\phi(y \ppr a)-\psi(y \ppr a)\\
&&-\Delta_{A}(b\ppl x)-\phi(b\ppl x)-\psi(b\ppl x)-\Delta_{H}(yx)-\rho(yx)-\gamma( y x)\\
&&-\Delta_{H}(y \trl a)-\rho(y \trl a)-\gamma(y \trl a)-\Delta_{H}(b \trr x)-\rho(b \trr x)-\gamma(b \trr x),
\end{eqnarray*}
and the right hand side is equal to
\begin{eqnarray*}
&&\Delta(a, x) \cdot(b, y)+(a, x) \cdot \Delta(b, y)+(a, x) \bullet \Delta(b, y)\\
&&-\Delta(b, y ) \cdot(a, x)-(b, y ) \cdot \Delta(a, x)-( b, y) \bullet \Delta(a, x)\\
&&=a_{1} \otimes a_{2} b+a_{1} \otimes\left(a_{2}\ppl y\right)+a_{1} \otimes\left(a_{2} \trr y\right)
 +a_{(-1)} \otimes a_{(0)} b+a_{(-1)} \otimes\left(a_{(0)}\ppl y\right)\\&&+a_{(-1)} \otimes\left(a_{(0)} \trr y\right)
 +a_{(0)}\otimes\left(a_{(1)} \ppr b\right)+a_{(0)} \otimes\left(a_{(1)} \trl b\right)+a_{(0)} \otimes a_{(1)}y
 +x_{1} \otimes\left(x_{2}\ppr b\right)\\&&+x_{1} \otimes\left(x_{2} \trl b\right)+x_{1} \otimes x_{2} y
 +x\boi\ot(x\boo\ppr b)+x\boi\ot(x\boo\trl b)+x\boi\ot x\boo y\\
 &&+x\boo\ot x\bi b+x\boo\ot (x\bi\ppl y)+x\boo\ot (x\bi\trr y)
 +a b_{1} \otimes b_{2}+\left(x \ppr b_{1}\right) \otimes b_{2}\\&&+\left(x \trl b_{1}\right) \otimes b_{2}
 +(a\ppl b_{(-1)}) \otimes b_{(0)}+(a \trr b_{(-1)}) \otimes b_{(0)}+x b_{(-1)} \otimes b_{(0)}
 +a b_{(0)} \otimes b_{(1)}\\&&+(x \ppr b_{(0)}) \otimes b_{(1)}+(x \trl b_{(0)}) \otimes b_{(1)}
 +\left(a\ppl y_{1}\right) \otimes y_{2}+\left(a \trr y_{1}\right) \otimes y_{2}+x y_{1}\ot y_{2} \\
&&+ay\boi\ot y\boo+(x\ppr y\boi)\ot y\boo+(x\trl y\boi)\ot y\boo
 +(a\ppl y\boo)\to y\bi+(a\trr y\boo)\ot y\bi\\&&+xy\boo\ot y\bi
 + b_{1}\otimes a b_{2}+b_{1} \otimes \left(x \ppr b_{2}\right)+ b_{1}\otimes \left(x \trl b_{2}\right)
 + b_{(0)}\otimes (a\ppl b_{(1)})\\&&+ b_{(0)}\otimes (a \trr b_{(1)})+ b_{(0)}\otimes x b_{(1)}
 +b_{(-1)}\otimes a b_{(0)} + b_{(-1)}\otimes(x \ppr b_{(0)})+b_{(-1)}\otimes (x \trl b_{(0)})\\
 &&+ y_{1}\otimes\left(a\ppl y_{2}\right)+ y_{1}\otimes \left(a \trr y_{2}\right)+y_{1}\ot x y_{2}
 +y\boo\ot ay\bi+ y\boo\ot(x\ppr y\bi)\\&&+ y\boo\ot(x\trl y\bi)
 + y\boi\ot(a\ppl y\boo)+y\boi\ot (a\trr y\boo)+y\boi\ot xy\boo \\
 &&-b_{1} \otimes b_{2} a-b_{1} \otimes\left(b_{2}\ppl x\right)-b_{1} \otimes\left(b_{2} \trr x\right)
 -b_{(-1)} \otimes b_{(0)} a-b_{(-1)} \otimes\left(b_{(0)}\ppl x\right)\\&&-b_{(-1)} \otimes\left(b_{(0)} \trr x\right)
 -b_{(0)}\otimes\left(b_{(1)} \ppr a\right)-b_{(0)} \otimes\left(b_{(1)} \trl a\right)-b_{(0)} \otimes b_{(1)}x
 -y_{1} \otimes\left(y_{2}\ppr a\right)\\&&-y_{1} \otimes\left(y_{2} \trl a\right)-y_{1} \otimes y_{2} x
 -y\boi\ot(y\boo\ppr a)-y\boi\ot(y\boo\trl a)-y\boi\ot y\boo x\\
 &&-y\boo\ot y\bi a-y\boo\ot (y\bi\ppl x)-y\boo\ot (y\bi\trr x)
  -b a_{1} \otimes a_{2}-\left(y \ppr a_{1}\right) \otimes a_{2}\\&&-\left(y \trl a_{1}\right) \otimes a_{2}
-(b\ppl a_{(-1)}) \otimes a_{(0)}-(b \trr a_{(-1)}) \otimes a_{(0)}-y a_{(-1)} \otimes a_{(0)}
 -b a_{(0)} \otimes a_{(1)}\\&&-(y \ppr a_{(0)}) \otimes a_{(1)}-(y \trl a_{(0)}) \otimes a_{(1)}
-\left(b\ppl x_{1}\right) \otimes x_{2}-\left(b \trr x_{1}\right) \otimes x_{2}-y x_{1}\ot x_{2} \\
&&-bx\boi\ot x\boo-(y\ppr x\boi)\ot x\boo-(y\trl x\boi)\ot x\boo
 -(b\ppl x\boo)\ot x\bi-(b\trr x\boo)\ot x\bi\\&&-yx\boo\ot x\bi
 - a_{1}\otimes b a_{2}-a_{1} \otimes \left(y \ppr a_{2}\right)- a_{1}\otimes \left(y \trl a_{2}\right)
 - a_{(0)}\otimes (b\ppl a_{(1)})\\&&- a_{(0)}\otimes (b \trr a_{(1)})- a_{(0)}\otimes y a_{(1)}
 -a_{(-1)}\otimes b a_{(0)} - a_{(-1)}\otimes(y \ppr a_{(0)})-a_{(-1)}\otimes (y \trl a_{(0)})\\
 &&- x_{1}\otimes\left(b\ppl x_{2}\right)- x_{1}\otimes \left(b \trr x_{2}\right)-x_{1}\ot y x_{2}
 -x\boo\ot bx\bi- x\boo\ot(y\ppr x\bi)\\&&- x\boo\ot(y\trl x\bi)
 - x\boi\ot(b\ppl x\boo)-x\boi\ot (b\trr x\boo)-x\boi\ot yx\boo.
\end{eqnarray*}
Compare  both the two sides, we will find  the double matched pair conditions (CDM1)--(CDM8) in Definition \ref{def:dmp}.

Now we continue to check the second compatibility condition $(\id-\tau)(\Delta( (a, x), (b, y) ))=(\id-\tau)(\Delta(a, x) \cdot(b, y)+(a, x) \cdot \Delta(b, y)+(a, x) \bullet \Delta(b, y))$,  the left hand side is equal to
\begin{eqnarray*}
&&(\id-\tau)(\Delta( (a, x), (b, y) ))\\
&=&(\id-\tau)\Delta(a b+x \ppr b+a\ppl y, x y+x \trl b+a\trr y)\\
&=& \Delta_A(a b)+\phi(a b)+\psi(a b)+\Delta_A(x \ppr b)+\phi(x \ppr b)+\psi(x \ppr b)\\
&&+\Delta_{A}(a\ppl y)+\phi(a\ppl y)+\psi(a\ppl y)+\Delta_{H}(x y)+\rho(x y)+\gamma(x y)\\
&&+\Delta_{H}(x \trl b)+\rho(x \trl b)+\gamma(x \trl b)+\Delta_{H}(a \trr y)+\rho(a \trr y)+\gamma(a \trr y)\\
&&-\tau\Delta_A(a b)-\tau\phi(a b)-\tau\psi(a b)-\tau\Delta_A(x \ppr b)-\tau\phi(x \ppr b)-\tau\psi(x \ppr b)\\
&&-\tau\Delta_{A}(a\ppl y)-\tau\phi(a\ppl y)-\tau\psi(a\ppl y)-\tau\Delta_{H}(x y)-\tau\rho(x y)-\tau\gamma(x y)\\
&&-\tau\Delta_{H}(x \trl b)-\tau\rho(x \trl b)-\tau\gamma(x \trl b)-\tau\Delta_{H}(a \trr y)-\tau\rho(a \trr y)-\tau\gamma(a \trr y),
\end{eqnarray*}
and the right hand side is equal to
\begin{eqnarray*}
&&(\id-\tau)\Delta(a, x) \cdot(b, y)+(a, x) \cdot \Delta(b, y)+(a, x) \bullet \Delta(b, y)\\
&=&(\id-\tau)\Big(a_{1} \otimes a_{2} b+a_{1} \otimes\left(a_{2}\ppl y\right)+a_{1} \otimes\left(a_{2} \trr y\right)+a_{(-1)} \otimes a_{(0)} b+a_{(-1)} \otimes\left(a_{(0)}\ppl y\right)\\&&+a_{(-1)} \otimes\left(a_{(0)} \trr y\right)
 +a_{(0)}\otimes\left(a_{(1)} \ppr b\right)+a_{(0)} \otimes\left(a_{(1)} \trl b\right)+a_{(0)} \otimes a_{(1)}y
 +x_{1} \otimes\left(x_{2}\ppr b\right)\\&&+x_{1} \otimes\left(x_{2} \trl b\right)+x_{1} \otimes x_{2} y
 +x\boi\ot(x\boo\ppr b)+x\boi\ot(x\boo\trl b)+x\boi\ot x\boo y\\
 &&+x\boo\ot x\bi b+x\boo\ot (x\bi\ppl y)+x\boo\ot (x\bi\trr y)
 +a b_{1} \otimes b_{2}+\left(x \ppr b_{1}\right) \otimes b_{2}\\&&+\left(x \trl b_{1}\right) \otimes b_{2}
 +(a\ppl b_{(-1)}) \otimes b_{(0)}+(a \trr b_{(-1)}) \otimes b_{(0)}+x b_{(-1)} \otimes b_{(0)}
 +a b_{(0)} \otimes b_{(1)}\\&&+(x \ppr b_{(0)}) \otimes b_{(1)}+(x \trl b_{(0)}) \otimes b_{(1)}
 +\left(a\ppl y_{1}\right) \otimes y_{2}+\left(a \trr y_{1}\right) \otimes y_{2}+x y_{1}\ot y_{2} \\
&&+ay\boi\ot y\boo+(x\ppr y\boi)\ot y\boo+(x\trl y\boi)\ot y\boo
 +(a\ppl y\boo)\ot y\bi+(a\trr y\boo)\ot y\bi\\&&+xy\boo\ot y\bi
 + b_{1}\otimes a b_{2}+b_{1} \otimes \left(x \ppr b_{2}\right)+ b_{1}\otimes \left(x \trl b_{2}\right)
 + b_{(0)}\otimes (a\ppl b_{(1)})\\&&+ b_{(0)}\otimes (a \trr b_{(1)})+ b_{(0)}\otimes x b_{(1)}
 +b_{(-1)}\otimes a b_{(0)} + b_{(-1)}\otimes(x \ppr b_{(0)})+b_{(-1)}\otimes (x \trl b_{(0)})\\
 &&+ y_{1}\otimes\left(a\ppl y_{2}\right)+ y_{1}\otimes \left(a \trr y_{2}\right)+y_{1}\ot x y_{2}
 +y\boo\ot ay\bi+ y\boo\ot(x\ppr y\bi)\\&&+ y\boo\ot(x\trl y\bi)
 + y\boi\ot(a\ppl y\boo)+y\boi\ot (a\trr y\boo)+y\boi\ot xy\boo\Big).
\end{eqnarray*}
Compare  both the two sides, we will find  the double matched pair conditions (CDM9)--(CDM16) in Definition \ref{def:dmp}. Thus the proof is completed.
\end{proof}

\subsection{Cocycle bicrossproduct   left-symmetric bialgebras}

In this section, we construct cocycle bicrossproduct   left-symmetric bialgebras, which is a generalization of double cross biproduct.

Let $A, H$ be both  left-symmetric algebras and  left-symmetric coalgebras.   For any $a, b\in A$, $x, y\in H$,  we denote maps
\begin{align*}
&\sigma: H\otimes H \to A,\quad \theta: A\otimes A \to H,\\
&P: A  \to H\otimes H,\quad  Q: H \to A\otimes A,
\end{align*}
by
\begin{eqnarray*}
&& \sigma (x, y)  \in A, \quad \theta(a, b) \in H,\\
&& P(a)=\sum a\ppi\ot  a\pii, \quad Q(x) = \sum x\qi \ot x\qii.
\end{eqnarray*}

A bilinear map $\si: H\ot H\to A$ is called a cocycle on $H$ if
\begin{enumerate}
\item[(CC1)] $\sigma(x y, z)-{\sigma}(x, y z)+\sigma(x, y)\ppl z-x \ppr \sigma(y, z)\\
=\sigma(yx, z)-\sigma(y, xz)+\sigma(y, x)\ppl z-y\ppr\si(x, z).$
\end{enumerate}

A bilinear map $\theta: A\ot A\to H$ is called a cocycle on $A$ if
\begin{enumerate}
\item[(CC2)] $\theta(a b, c)-\theta(a, b c)+\theta(a, b) \triangleleft c-a\trr \theta(b, c)=\theta(ba, c)-\theta(b, a c)+\theta(b, a) \triangleleft c-b\trr \theta(a, c).$
\end{enumerate}

A bilinear map $P: A\to H\ot H$ is called a cycle on $A$ if
\begin{enumerate}
\item[(CC3)]  $\Delta_H(a\ppi)\ot a\pii-a\ppi\ot \Delta_H(a\pii)+P(a\lmoo)\ot a\mi-a\lmoi\ot P(a\lmoo)\\
=\tau_{12}\left(\Delta_H(a\ppi)\ot a\pii-a\ppi\ot \Delta_H(a\pii)+P(a\lmoo)\ot a\mi-a\lmoi\ot P(a\lmoo)\right)$.
\end{enumerate}

A bilinear map $Q: H\to A\ot A$ is called a cycle on $H$ if
\begin{enumerate}
\item[(CC4)]  $\Delta_A(x\qi)\ot x\qii-x\qi\ot \Delta_A(x\qii)+Q(x\boo)\ot x\bi-x\boi\ot Q(x\boo)\\
=\tau_{12}\left( \Delta_A(x\qi)\ot x\qii-x\qi\ot \Delta_A(x\qii)+Q(x\boo)\ot x\bi-x\boi\ot Q(x\boo)\right)$.
\end{enumerate}

In the following definitions, we introduce the concept of cocycle
 left-symmetric algebras and cycle  left-symmetric coalgebras, which are  in fact not really ordinary  left-symmetric algebras and  left-symmetric coalgebras, but generalized ones.

\begin{definition}
(i): Let $\si$ be a cocycle on a vector space  $H$ equipped with a product $H \ot H \to H$, satisfying the following cocycle  identity:
\begin{enumerate}
\item[(CC5)] $(x y) z-x(y z)+\sigma(x, y) \trr z-x \triangleleft \sigma(y, z)=(yx) z-y(x z)+\sigma(y, x) \trr z-y \triangleleft \sigma(x, z)$.
\end{enumerate}
Then  $H$ is called a $\si$-left-symmetric algebra which is denoted by $(H, \sigma)$.

(ii): Let $\theta$ be a cocycle on a vector space $A$  equipped with a product $ A \ot A \to A$, satisfying the  following cocycle  identity:
\begin{enumerate}
\item[(CC6)] $(a b) c-a(b c)+\theta(a,  b) \ppr c-a\ppl \theta(b, c)=(ba) c-b(a c)+\theta(b, a) \ppr c-b\ppl \theta(a, c)$.
\end{enumerate}
Then  $A$ is called a $\theta$-left-symmetric algebra which is denoted by $(A, \theta)$.

(iii) Let $P$ be a cycle on a vector space  $H$ equipped with a coproduct $\Delta: H \to H \ot H$, satisfying the  following cycle  identity:
\begin{enumerate}
\item[(CC7)] $\Delta_H(x\li)\ot x\lii-x_1\ot \Delta_H(x_2)+ P(x\boi) \ot x\boo-x\boo\ot P(x\bi)\\
=\tau_{12} \left(\Delta_H(x\li)\ot x\lii-x_1\ot \Delta_H(x_2)+ P(x\boi) \ot x\boo-x\boo\ot P(x\bi) \right)$.
\end{enumerate}
\noindent Then  $H$ is called a $P$-left-symmetric coalgebra which is denoted by $(H, P)$.

(iv) Let $Q$ be a cycle on a vector space  $A$ equipped with a coproduct $\Delta: A \to A \ot A$, satisfying the following cycle  identity:
\begin{enumerate}
\item[(CC8)] $\Delta_A(a\li)\ot a\lii-a\li\ot \Delta_A(a\lii)+Q(a\lmoi)\ot a\lmoo-a\mo\ot Q(a\mi)\\
=\tau_{12} \left( \Delta_A(a\li)\ot a\lii-a\li\ot \Delta_A(a\lii)+Q(a\lmoi)\ot a\lmoo-a\mo\ot Q(a\mi) \right)$.
\end{enumerate}
\noindent Then  $A$ is called a $Q$-left-symmetric coalgebra which is denoted by $(A, Q)$.
\end{definition}

\begin{theorem}
Let $A$ be a $\theta$-left-symmetric algebra and $H$ a $\sigma$-left-symmetric algebra, where $\si: H\ot H\to A$ is a cocycle on $H$, $\theta: A\ot A\to H$ is a cocycle on $A$.
If we define $E=A_{\sigma}\#_{\theta} H$ as the vector space $A\oplus H$ with the   product
\begin{align}
(a, x)(b, y)=\big(ab+x\ppr b+a\ppl y+\sigma(x, y), \, xy+x\trl b+a\trr y+\theta(a, b)\big).
\end{align}
Then $E=A_{\sigma}\#_{\theta} H$ forms a  left-symmetric algebra,  which we call the cocycle cross product left-symmetric algebra, if and only if  the following conditions are satisfied:
\begin{enumerate}
\item[(CP1)] $[a, b] \trr x +(\theta(a, b) -\theta(b, a)) x=a\trr (b\trr x)-b\trr (a\trr x)+\theta(a, b\ppl x)-\theta(b, a\ppl x)$,
\item[(CP2)] $x \trl (a b)+x\theta(a, b)=a \trr (x\trl b)  +(x \trl a-a\trr x) \trl b+\theta(x \ppr a -a \ppl x, b)+\theta(a,x\ppr b)$,
\item[(CP3)] $ a \trr (x y)+\theta(a, \sigma(x, y))=(a \trr x-x\trl a) y+(a\ppl x-x\ppr a) \trr y+x\trl (a\ppl y)+x(a\trr y)$,
\item[(CP4)] $[x, y] \trl a+\theta(\sigma(x, y)-\sigma(y, x), a) =x \trl(y \ppr a)-y \trl(x \ppr a)+x(y \trl a)-y(x \trl a)$,
\item[(CP5)] $[x, y] \ppr a +(\sigma(x, y)  -\sigma(y, x)) a=x\ppr(y\ppr a)-y\ppr(x\ppr a)+\sigma(x, y\trl a)-\sigma(y, x\trl a)$,
\item[(CP6)] $a\ppl (x y)+a \sigma(x, y)=x\ppr (a\ppl y) +(a\ppl x-x\ppr a)\ppl y+\sigma(a \trr x-x\trl a, y)+\sigma(x, a\trr y)$,
\item[(CP7)] $x \ppr(a b)+\sigma(x, \theta(a, b))=(x\ppr a-a\ppl x) b+(x\trl a-a\trr x) \ppr b+a(x\ppr b)+a\ppl(x\trl b)$,
\item[(CP8)] $[a, b]\ppl x +\sigma(\theta(a, b)-\theta(b, a), x)=a(b\ppl x)+a\ppl (b \trr x)-b(a\ppl x)-b\ppl(a\trr x)$.
\end{enumerate}
And  $(A, H, \sigma, \theta)$ satisfying   above conditions is called  a  \emph{cocycle cross product system }.
\end{theorem}

\begin{proof} We have to check
$$((a, x)(b, y))(c, z)-(a, x)((b, y)(c, z))=((b, y)(a, x))(c, z)-(b, y)((a, x)(c, z)).$$
By direct computations, the right hand side is equal to
\begin{eqnarray*}
&&\big((b, y)(a, x)\big)(c, z)-(b, y)\big((a, x)(c, z)\big)\\
&=&\big(b a+y \ppr a+b\ppl x+\sigma(y, x), y x+y \trl a+b\trr x+\theta(b, a)\big)(c, z)\\
&&-\left(b, y\right)\big(a c+x\ppr c+a\ppl z+\sigma(x, z), x z+x \trl c+a \trr z+\theta(a, c)\big)\\
&=&\Big( (b a) c+(y\ppr a) c+(b\ppl x) c+\sigma(y, x) c+(y x) \ppr c
+(y \trl a) \ppr c+(b \trr x)\ppr c\\&&+\theta(b, a) \ppr c
 +(b a) \ppl z+(y \ppr a)\ppl z+(b\ppl x)\ppl z+\sigma(y ,x)\ppl z
 +\sigma(y x, z)\\&&+\sigma(y \trl a, z)+\sigma(b \trr x, z) +\sigma(\theta(b, a), z),
 \quad (y x) z+(y\trl a) z+(b\trr x) z+\theta(b, a) z\\&&+(y x) \trl c
+(y \trl a) \trl c+(b \trr x) \trl c+\theta(b, a) \trl c
 +(b a) \trr z+(y \ppr a) \trr z+\\&&(b\ppl x)\trr z+\sigma(y, x) \trr z
+\theta(b a, c)+\theta(y \ppr a, c)+\theta(b\ppl x, c)+\theta(\sigma(y, x), c)\Big)\\
&&-\Big(b(a c)+b(x \ppr c)+b(a\ppl z)+b(\sigma(x, z))+y \ppr(a c)
 +y \ppr(x\ppr c)+y \ppr(a\ppl z)\\&&+y \ppr \sigma(x, z)
 +b\ppl (x z)+b\ppl (x \trl c)+b\ppl (a \trr z)+b\ppl \theta(a, c)
 +\sigma(y, x z)+\sigma(y, x \trl c)\\&&+\sigma(y, a \trr z)+\sigma(y, \theta(a, c)),
 \quad y(x z)+y(x \trl c)+y(a \trr z)+y \theta(a, c)+y \trl (a c) \\
&&+y \trl(x \ppr c)+y\trl (a\ppl z)+y \trl \sigma(x, z )
 +b \trr(x z)+b \trr(x\trl c)+b \trr(a \trr z)\\&&+b \trr \theta(a, c)
 +\theta(b, a c)+\theta(b, x \ppr c)+\theta(b, a\ppl z)+\theta(b, \sigma(x, z))\Big),
\end{eqnarray*}
and the left hand side is equal to
\begin{eqnarray*}
&&\big((a, x)(b, y)\big)(c, z)-(a, x)\big((b, y)(c, z)\big)\\
&=&\big(a b+x \ppr b+a\ppl y+\sigma(x, y), x y+x \trl b+a\trr y+\theta(a, b)\big)(c, z)\\
&&-\left(a, x\right)\big(b c+y\ppr c+b\ppl z+\sigma(y, z), y z+y \trl c+b \trr z+\theta(b, c)\big)\\
&=&\Big( (a b) c+(x\ppr b) c+(a\ppl y) c+\sigma(x, y) c+(x y) \ppr c
+(x \trl b) \ppr c+(a \trr y)\ppr c\\&&+\theta(a, b) \ppr c
 +(a b) \ppl z+(x \ppr b)\ppl z+(a\ppl y)\ppl z+\sigma(x,  y)\ppl z
 +\sigma(x y, z)\\&&+\sigma(x \trl b, z)+\sigma(a \trr y, z) +\sigma(\theta(a, b) , z),
 \quad (x y) z+(x\trl b) z+(a\trr y) z+\theta(a, b) z\\&&+(x y) \trl c
 +(x \trl b) \trl c+(a \trr y) \trl c+\theta(a, b) \trl c
 +(a b) \trr z+(x \ppr b) \trr z+(a\ppl y)\trr z\\&&+\sigma(x, y) \trr z
 +\theta(a b, c)+\theta(x \ppr b, c)+\theta(a\ppl y, c)+\theta(\sigma(x, y), c)\Big)\\
&&-\Big(a(b c)+a(y \ppr c)+a(b\ppl z)+a(\sigma(y, z))+x \ppr(b c)
 +x \ppr(y\ppr c)+x \ppr(b\ppl z)\\&&+x \ppr \sigma(y, z)
 +a\ppl (y z)+a\ppl (y \trl c)+a\ppl (b \trr z)+a\ppl \theta(b, c)
 +\sigma(x, y z)\\&&+\sigma(x, y \trl c)+\sigma(x, b \trr z)+\sigma(x, \theta(b, c)),
 \quad x(y z)+x(y \trl c)+x(b \trr z)+x \theta(b, c)\\&&+x \trl (b c)
 +x \trl(y \ppr c)+x\trl (b\ppl z)+x \trl \sigma(y, z)
 +a \trr(y z)+a \trr(y\trl c)\\&&+a \trr(b \trr z)+a \trr \theta(b, c)
 +\theta(a, b c)+\theta(a, y \ppr c)+\theta(a, b\ppl z)+\theta(a, \sigma(y, z))\Big).
\end{eqnarray*}
 Thus the two sides are equal to each other if and only if (CP1)--(CP8) hold.
\end{proof}

\begin{definition}
A  \emph{cycle cross coproduct system }$(A, H, P, Q)$ is a pair of   $P$-left-symmetric coalgebra $A$ and  $Q$-left-symmetric coalgebra $H$,  where $P: A\to H\ot H$ is a cycle on $A$,  $Q: H\to A\ot A$ is a cycle over $H$ such that following conditions are satisfied:
\begin{enumerate}
\item[(CCP1)] $\phi\left(a_{1}\right) \otimes a_{2}+\gamma\left(a_{(-1)}\right) \otimes a_{(0)}-a_{(-1)} \otimes \Delta_{A}\left(a_{(0)}\right)-a\ppi \otimes Q\left(a\pii\right)\\
    =\tau_{12}\left(\psi\left(a_{1}\right) \otimes a_{2}+\rho\left(a_{(-1)}\right) \otimes a_{(0)}-a_{1} \otimes \phi(a_{2})-a_{(0)} \otimes \gamma(a_{(1)})\right)$,
\item[(CCP2)] $\Delta_{A}\left(a_{(0)}\right) \otimes a_{(1)}+Q\left(a\ppi\right) \otimes a\pii-a_{1} \otimes \psi\left(a_{2}\right)-a_{(0)} \otimes \rho\left(a_{(1)}\right)\\
    =\tau_{12}\left(\Delta_{A}\left(a_{(0)}\right) \otimes a_{(1)}+Q\left(a\ppi\right) \otimes a\pii-a_{1} \otimes \psi\left(a_{2}\right)-a_{(0)} \otimes \rho\left(a_{(1)}\right)\right)$,
\item[(CCP3)] $\rho\left(x_{1}\right) \otimes x_{2}+\psi\left(x_{[-1]}\right) \otimes x_{[0]}-x_{[-1]} \otimes \Delta_{H}\left( x_{[0]}\right)-x\qi \otimes P\left(x\qii\right)  \\
    =\tau_{12}\left(\gamma\left(x_{1}\right) \otimes x_{2}+\phi\left(x_{[-1]}\right) \otimes x_{[0]}-x_{[0]} \otimes \psi\left(x_{[1]}\right)-x_{1} \otimes \rho(x_{2})\right)$,
\item[(CCP4)] $\Delta_{H}(x\boo)\ot x\bi+P(x\qi)\ot x\qii-x\boo\ot \phi(x\bi)-x_1\ot \gamma(x_2)\\
=\tau_{12}\left(\Delta_{H}(x\boo)\ot x\bi+P(x\qi)\ot x\qii-x\boo\ot \phi(x\bi)-x_1\ot \gamma(x_2) \right)$,
\item[(CCP5)] $\Delta_{H}\left(a_{(-1)}\right) \otimes a_{(0)}+P\left(a_{1}\right) \otimes a_{2}-a_{(-1)} \otimes \phi\left(a_{(0)}\right)-a\ppi\otimes \gamma\left(a\pii\right)\\
    =\tau_{12}\left( \Delta_{H}\left(a_{(-1)}\right) \otimes a_{(0)}+P\left(a_{1}\right) \otimes a_{2}-a_{(-1)} \otimes \phi\left(a_{(0)}\right)-a\ppi\otimes \gamma\left(a\pii\right) \right)$,
\item[(CCP6)] $a_{(0)}\ot \Delta_{H}\left(a_{(1)}\right)+a_{1}\otimes P\left(a_{2}\right)-\psi\left(a_{(0)}\right) \otimes a_{(1)}-\rho\left(a\ppi\right) \otimes a\pii\\
    =\tau_{12}\left(  a_{(-1)} \otimes \psi\left(a_{(0)}\right)+a\ppi \otimes \rho\left(a\pii\right)-\phi\left(a_{(0)}\right) \otimes a_{(1)}-\gamma\left(a\ppi\right) \otimes a\pii \right)$,
\item[(CCP7)] $x_{[-1]} \otimes \rho\left(x_{[0]}\right)+x\qi \otimes \psi\left(x\qii\right)-\Delta_{A}\left(x_{[-1]}\right) \otimes x_{[0]}-Q\left(x_{1}\right) \otimes x_{2}\\
    =\tau_{12}\left( x_{[-1]} \otimes \rho\left(x_{[0]}\right)+x\qi \otimes \psi\left(x\qii\right)-\Delta_{A}\left(x_{[-1]}\right) \otimes x_{[0]}-Q\left(x_{1}\right) \otimes x_{2} \right)$,
\item[(CCP8)] $x\boo\ot\Delta_A(x\bi)+x_1\ot Q(x_2)-\gamma(x\boo)\ot x\bi-\phi(x\qi)\ot x\qii\\
=\tau_{12}\left(x\boi\ot\gamma(x\boo)+x\qi\ot\phi(x\qii)-\rho(x\boo)\ot x\bi-\psi(x\qi)\ot x\qii\right)$.
\end{enumerate}
\end{definition}

\begin{lemma}\label{lem2} Let $(A, H, P, Q)$ be  a  cycle cross coproduct system. If we define $E=A^{P}\# {}^{Q} H$ as the vector
space $A\oplus H$ with the   coproduct
$$\Delta_{E}(a)=(\Delta_{A}+\phi+\psi+P)(a),\quad \Delta_{E}(x)=(\Delta_{H}+\rho+\gamma+Q)(x), $$
that is
$$\Delta_{E}(a)= a\li \ot a\lii+ a\moi \ot a\mo+a\mo\ot a\mi+a\ppi\ot a\pii,$$
$$\Delta_{E}(x)= x\li \ot x\lii+ x\boi \ot x\boo+x\boo \ot x\bi+x\qi\ot x\qii,$$
then  $A^{P}\# {}^{Q} H$ forms a   left-symmetric coalgebra which we will call it the cycle cross coproduct left-symmetric  coalgebra.
\end{lemma}

\begin{proof} We have to check $ (\id-\tau_{12})\Big((\Delta_{E} \otimes \id) \Delta_{E}(a, x)-(\id\otimes \Delta_{E}  ) \Delta_{E}(a, x)\Big)=0$. By direct computations, we have that
\begin{eqnarray*}
&& (\Delta_{E} \otimes \id) \Delta_{E}(a, x)-(\id\otimes \Delta_{E}  ) \Delta_{E}(a, x)\\
&=&\Delta_{A}\left(a_{1}\right) \otimes a_{2}+\phi\left(a_{1}\right) \otimes a_{2}+\psi\left(a_{1}\right) \otimes a_{2}+P\left(a_{1}\right) \otimes a_{2}\\
&&+\Delta_{H}\left(a_{(-1)}\right) \otimes a_{(0)}+\rho\left(a_{(-1)}\right) \otimes a_{(0)}+\gamma\left(a_{(-1)}\right) \otimes a_{(0)}+Q\left(a_{(-1)}\right) \otimes a_{(0)}\\
&&+\Delta_{A}\left(a_{(0)}\right) \otimes a_{(1)}+\phi\left(a_{(0)}\right) \otimes a_{(1)}+\psi\left(a_{(0)}\right) \otimes a_{(1)}+P\left(a_{(0)}\right) \otimes a_{(1)}\\
&&+\Delta_{H}\left(a\ppi\right) \otimes a\pii+\rho\left(a\ppi\right) \otimes a\pii+\gamma\left(a\ppi\right) \otimes a\pii+Q\left(a\ppi\right) \otimes a\pii\\
&&+\Delta_{H}\left(x_{1}\right) \ot x_{2}+\rho\left(x_{1}\right) \otimes x_{2}+\gamma\left(x_{1}\right) \otimes x_{2}+Q\left(x_{1}\right) \otimes x_{2}\\
&&+\Delta_{A}\left(x\boi\right) \otimes x\boo+ \phi\left(x\boi\right) \otimes x\boo+\psi\left(x\boi\right)\otimes x\boo+P\left(x\boi\right)\otimes x\boo\\
&&+\Delta_{H}\left(x\boo\right) \otimes x\bi+ \rho\left(x\boo\right) \otimes x\bi+\gamma\left(x\boo\right)\otimes x\bi+Q\left(x\boo\right)\otimes x\bi\\
&&+\Delta_{A}\left(x\qi\right) \otimes x\qii+\phi\left(x\qi\right) \otimes x\qii+\psi\left(x\qi\right) \otimes x\qii+P\left(x\qi\right) \otimes x\qii\\
&&-a_{1} \otimes \Delta_{A}\left(a_{2}\right)-a_{1} \otimes \phi\left(a_{2}\right)-a_{1} \otimes \psi\left(a_{2}\right)-a_{1} \otimes P\left(a_{2}\right) \\
&&-a_{(-1)} \otimes \Delta_{A}\left(a_{(0)}\right)-a_{(-1)} \otimes \phi\left(a_{(0)}\right)-a_{(-1)} \otimes \psi\left(a_{(0)}\right)-a_{(-1)} \otimes P\left(a_{(0)}\right)\\
&&-a_{(0)} \otimes \Delta_{H}\left(a_{(1)}\right)-a_{(0)} \otimes \rho\left(a_{(1)}\right)-a_{(0)} \otimes \gamma\left(a_{(1)}\right)-a_{(0)} \otimes Q\left(a_{(1)}\right)\\
&&-a\ppi \otimes \Delta_{H}\left(a\pii\right)-a\ppi \otimes \rho\left(a\pii\right)-a\ppi \otimes \gamma\left(a\pii\right)-a\ppi\otimes Q\left(a\pii\right)\\
&&-x_{1} \otimes \Delta_{H}\left(x_{2}\right)-x_{1} \otimes \rho\left(x_{2}\right)-x_{1} \otimes \gamma\left(x_{2}\right)-x_{1} \otimes Q\left(x_{2}\right)\\
&&-x\boi \otimes \Delta_{H}\left(x\boo\right)-x\boi\otimes \rho\left(x\boo\right)-x\boi \otimes \gamma\left(x\boo\right)-x\boi\otimes Q\left(\boo\right)\\
&&-x\boo \otimes \Delta_{A}\left(x\bi\right)-x\boo \otimes \phi\left(x\bi\right)-x\boo\otimes \psi\left(x\bi\right)-x\boo \otimes P\left(x\bi\right)\\
&&-x\qi \otimes \Delta_{A}\left(x\qii\right)-x\qi \otimes \phi\left(x\qii\right)-x\qi \otimes \psi\left(x\qii\right)-x\qi \otimes P\left(x\qii\right).
\end{eqnarray*}
Thus the  coproduct  is left symmetric if and only if (CCP1)--(CCP8) hold.
\end{proof}

\begin{definition}\label{cocycledmp}
Let $A, H$ be both  left-symmetric algebras and  left-symmetric coalgebras. If  the following conditions hold:
\begin{enumerate}
\item[(CDM1)]  $\phi([a, b])+\gamma(\theta(a, b) - \theta(b, a))$\\
$=a_{(-1)} \otimes[a_{(0)}, b]+(a \trr b_{(-1)}) \otimes b_{(0)}+a\ppi \otimes \left(a\pii \ppr b\right)$\\
$+\theta\left(a, b_{1}\right) \otimes b_{2}+b_{(-1)}\ot [a,b_{(0)}]+b\ppi \ot (a\ppl b\pii) $\\
$-(b \trr a_{(-1)}) \otimes a_{(0)}-b\ppi \otimes \left(b\pii \ppr a\right)-\theta\left(b, a_{1}\right) \otimes a_{2}-a\ppi \ot (b\ppl a\pii) $,

\item[(CDM2)] $\psi([a, b]) +\rho(\theta(a, b) - \theta( b, a))$\\
$=a b_{(0)} \otimes b_{(1)}+a_{(0)} \otimes\left(a_{(1)} \trl b\right)+ \left(a\ppl b\ppi\right) \otimes b\pii+a_{1} \otimes \theta\left(a_{2}, b\right)$\\
$+b_{(0)}\ot (a\trr b_{(1)})+b_{1}\ot \tht(a, b_{2})-b a_{(0)} \otimes a_{(1)}-b_{(0)} \otimes\left(b_{(1)} \trl a\right)$\\
$- \left(b\ppl a\ppi\right) \otimes a\pii-b_{1} \otimes \theta\left(b_{2}, a\right)-a_{(0)}\ot (b\trr a_{(1)})-a_{1}\ot \tht(b, a_{2})$,

\item[(CDM3)] $\rho([x, y]) +\psi(\sigma(x, y) - \sigma(y, x))$\\
$=x_{[-1]} \otimes [x_{[0]}, y] +\left(x \ppr y_{[-1]}\right) \otimes y_{[0]}+x\qi\otimes (x\qii\trr y)$\\
$+\sigma\left(x, y_{1}\right) \otimes y_{2}+y_{[-1]}\ot [x, y_{[0]}]+y\qi\ot (x\trl y\qii)-\left(y \ppr x_{[-1]}\right) \otimes x_{[0]}$\\
$-y\qi\otimes (y\qii\trr x) -\sigma\left(y, x_{1}\right) \otimes x_{2} -x\qi\ot (y\trl x\qii)$,

\item[(CDM4)] $\gamma([x, y]) +\phi(\sigma(x, y) - \sigma(y, x))$\\
$=x_{[0]}\otimes (x_{[1]}\ppl y)+xy_{[0]}\otimes y_{[1]}+x_{1} \otimes \sigma\left(x_{2}, y\right)+ \left(x\trl y\qi\right)\otimes y\qii$\\
$+y\li\ot \si(x, y\lii)+y\poo\ot (x\ppr y\bi)-y_{[0]}\otimes (y_{[1]}\ppl x)-yx_{[0]}\otimes x_{[1]}$\\
$-y_{1} \otimes \sigma\left(y_{2}, x\right)- \left(y\trl x\qi\right)\otimes x\qii-x\li\ot \si(y, x\lii)-x\poo\ot (y\ppr x\bi)$,

\item[(CDM5)] $\Delta_{A}(x \ppr b)-\Delta_{A}(b \ppl x)+Q(x\trl  b)-Q(b \trr x)$ \\
$=x_{[-1]} \otimes\left(x_{[0]} \ppr b\right)+\left(x \ppr b_{1}\right) \otimes b_{2}+x\qi \otimes [x\qii, b]+\sigma(x, b_{(-1)}) \otimes b_{(0)}$\\
$+b\li \ot (x\ppr b\lii)+b\loo\ot \si(x, b\lmi)-b_{1} \otimes\left(b_{2} \ppl x\right)-\left(b\ppl x_{[0]}\right) \otimes x_{[1]}$\\
$-b_{(0)} \otimes \sigma\left(b_{(1)},  x\right)-b x\qi \otimes x\qii-x\boi \ot (b\ppl x\boo)$,

\item[(CDM6)] $\Delta_{H}(a \trr y)-\Delta_{H}(y \trl a)+P(a\ppl y)-P(y\ppr a)$\\
$=a_{(-1)} \otimes\left(a_{(0)}\trr y\right)+\left(a \trr y_{1}\right) \otimes y_{2}+a\ppi \otimes [a\pii, y]+\theta\left(a, y_{[-1]}\right) \otimes y_{[0]}$\\
$+y\li\ot (a\trr y\lii)+y\poo\ot \tht(a, y\bi)-y_{1} \otimes\left(y_{2}  \trl a\right)-\left(y\trl a_{(0)}\right) \otimes a_{(1)}$\\
$-y_{[0]} \otimes \theta\left(y_{[1]}, a\right)-y a\ppi\otimes a\pii -a_{(-1)}\ot (y\trl a\loo)$,

\item[(CDM7)]$\Delta_{H}(\theta(a,b)-\theta(b,a))+P([a, b])$\\
$=a_{(-1)} \otimes\theta(a_{(0)},b)+a\ppi \otimes (a\pii\trl b)+\theta(a,b_{(0)})\otimes b_{(1)}$\\
$+(a\trr b\ppi)\otimes b\pii+b\loi \ot \tht(a,b\lmoo)+b\ppi\ot (a\trr b\pii)$\\
$-b_{(-1)} \otimes\theta(b_{(0)},a)-b\ppi \otimes (b\pii\trl a)-\theta(b,a_{(0)})\otimes a_{(1)}$\\
$-(b\trr a\ppi)\otimes a\pii-a\loi \ot \tht(b,a\lmoo)-a\ppi\ot (b\trr a\pii)$,

\item[(CDM8)]$\Delta_{A}(\sigma(x,y)-\sigma(y, x))+Q([x, y])$\\
$=x_{[-1]}\otimes \sigma(x_{[0]},y)+x\qi\otimes (x\qii\ppl y)+\sigma(x,y_{[0]})\otimes y_{[1]}$\\
$+(x\ppr y\qi)\otimes y\qii+y\poi \ot \si(x,y\poo)+y\qi\ot (x\ppr y\qii)$\\
$-y_{[-1]}\otimes \sigma(y_{[0]},x)-y\qi\otimes (y\qii\ppl x)-\sigma(y,x_{[0]})\otimes x_{[1]}$\\
$-(y\ppr x\qi)\otimes x\qii-x\poi \ot \si(y,x\poo)-x\qi\ot (y\ppr x\qii)$,

\item[(CDM9)]
 $\phi(x \ppr b)-\phi(b \ppl x)+\gamma(x\trl b)-\gamma(b\trr x)$\\
 $=x_{1} \otimes\left(x_{2} \ppr b\right)+x b_{(-1)} \otimes b_{(0)}+x_{[0]} \otimes [x_{[1]}, b]+\left(x\trl b_{1}\right) \otimes b_{2}$\\
$+b_{(-1)}\ot (x\ppr b\loo)+b\ppi\ot \si(x,b\ppi)-b_{(-1)} \otimes\left(b_{(0)}\ppl x\right)$\\
$-\tht(b, x\qi)\ot x\qii-\left(b\trr x_{[0]}\right)\otimes x_{[1]} -b\ppi \ot \si(b\pii, x)-x\li \ot (b\ppl x\lii)$,

\item[(CDM10)]
$\psi(x \ppr b)-\psi(b \ppl x)+\rho(x\trl b)-\rho(b\trr x)$\\
$=(x \ppr b_{(0)}) \otimes b_{(1)}+x\qi\ot \tht(x\qii, b)+x_{[-1]} \otimes\left(x_{[0]} \trl b\right)+\si(x, b\ppi)\ot b\pii$\\
$+b\li\ot (x\trl b\lii)+b\loo\ot [x, b\lmi]-\left(b\ppl x_{1}\right) \otimes x_{2}-b_{1} \otimes\left(b_{2} \trr x\right)$\\
$-b x_{[-1]} \otimes x_{[0]}-x\boi \ot (b\trr x\poo)-x\qi \ot \tht(b, x\qii)$,

\item[(CDM11)]  $\phi(ab)+\gamma(\theta(a, b))-\tau\psi(a b)-\tau\rho(\theta(a, b))$\\
$=a_{(-1)} \otimes\left(a_{(0)} b\right)+(a \trr b_{(-1)}) \otimes b_{(0)}+a\ppi \otimes \left(a\pii \ppr b\right)$\\
$+\theta\left(a, b_{1}\right) \otimes b_{2}+b_{(-1)}\ot ab_{(0)}+b\ppi \ot (a\ppl b\pii) $\\
$-\tau\Big((a b_{(0)}) \otimes b_{(1)}+ a_{(0)} \otimes\left(a_{(1)} \trl b\right)+ \left(a\ppl b\ppi\right) \otimes b\pii$\\
$+ a_{1} \otimes \theta\left(a_{2}, b\right)+ b_{(0)}\ot (a\trr b_{(1)})+ b_{1}\ot \tht(a, b_{2})\Big)$,

\item[(CDM12)] $\rho(x y) +\psi(\sigma(x, y))-\tau\gamma(x y)-\tau\phi(\sigma(x, y))$\\
$=x_{[-1]} \otimes x_{[0]} y +\left(x \ppr y_{[-1]}\right) \otimes y_{[0]}+x\qi\otimes (x\qii\trr y)$\\
$+\sigma\left(x, y_{1}\right) \otimes y_{2}+y_{[-1]}\ot xy_{[0]}+y\qi\ot (x\trl y\qii)$\\
$-\tau\Big(x_{[0]}\otimes (x_{[1]}\ppl y)+xy_{[0]}\otimes y_{[1]}+x_{1} \otimes \sigma\left(x_{2}, y\right)$\\
$+ \left(x\trl y\qi\right)\otimes y\qii+y\li\ot \si(x, y\lii)+y\poo\ot (x\ppr y\bi)\Big)$,

\item[(CDM13)] $(\id-\tau)(\Delta_{A}(x \ppr b) +Q(x\trl  b) )$ \\
$=(\id-\tau)\Big(x_{[-1]} \otimes\left(x_{[0]} \ppr b\right)+\left(x \ppr b_{1}\right) \otimes b_{2}+x\qi \otimes x\qii b$\\
$+\sigma(x, b_{(-1)}) \otimes b_{(0)}+b\li \ot (x\ppr b\lii)+b\loo\ot \si(x, b\lmi)$\Big),

\item[(CDM14)] $(\id-\tau)(\Delta_{A}(a\ppl y) +Q(a\trr y) )$\\
$=(\id-\tau)\Big(a_{1} \otimes\left(a_{2} \ppl y\right)+\left(a\ppl y_{[0]}\right) \otimes y_{[1]}+a_{(0)} \otimes \sigma\left(a_{(1)}, y\right)$\\
$+a y\qi \otimes y\qii+y\boi \ot (a\ppl y\boo)+y\qi\ot ay\qii$\Big),

\item[(CDM15)] $(\id-\tau)(\Delta_{H}(a \trr y) +P(a\ppl y) )$\\
$=(\id-\tau)\Big(a_{(-1)} \otimes\left(a_{(0)}\trr y\right)+\left(a \trr y_{1}\right) \otimes y_{2}+a\ppi \otimes a\pii y$\\
$+\theta\left(a, y_{[-1]}\right) \otimes y_{[0]}+y\li\ot (a\trr y\lii)+y\poo\ot \tht(a, y\bi)$\Big),

\item[(CDM16)] $(\id-\tau)(\Delta_{H}(x \trl b) +P(x\ppr b) )$\\
$=(\id-\tau)\Big(x_{1} \otimes\left(x_{2}  \trl b\right)+\left(x\trl b_{(0)}\right) \otimes b_{(1)}+x_{[0]} \otimes \theta\left(x_{[1]}, b\right)$\\
$+x b\ppi\otimes b\pii+b\ppi\ot xb\pii+b_{(-1)}\ot (x\trl b\loo)$\Big),

\item[(CDM17)]$(\id-\tau)(\Delta_{H}(\theta(a, b)) +P(a b) )$\\
$=(\id-\tau)\Big(a_{(-1)} \otimes\theta(a_{(0)}, b)+a\ppi \otimes (a\pii\trl b)+\theta(a, b_{(0)})\otimes b_{(1)}$\\
$+(a\trr b\ppi)\otimes b\pii+b\loi \ot \tht(a, b\lmoo)+b\ppi\ot (a\trr b\pii)$\Big),

\item[(CDM18)]$(\id-\tau)(\Delta_{A}(\sigma(x, y)) +Q(x y) )$\\
$=(\id-\tau)\Big(x_{[-1]}\otimes \sigma(x_{[0]},y)+x\qi\otimes (x\qii\ppl y)+\sigma(x, y_{[0]})\otimes y_{[1]}$\\
$+(x\ppr y\qi)\otimes y\qii+y\poi \ot \si(x,y\poo)+y\qi\ot (x\ppr y\qii)$\Big),

\item[(CDM19)]
 $\phi(x \ppr b) +\gamma(x\trl b) -\tau\psi(x \ppr b) -\tau\rho(x\trl b) $\\
 $=x_{1} \otimes\left(x_{2} \ppr b\right)+x b_{(-1)} \otimes b_{(0)}+x_{[0]} \otimes x_{[1]} b$\\
$+\left(x\trl b_{1}\right) \otimes b_{2}+b_{(-1)}\ot (x\ppr b\loo)+b\ppi\ot \si(x, b\pii)$\\
$-\tau\Big((x \ppr b_{(0)}) \otimes b_{(1)}+x\qi\ot \tht(x\qii, b)+x_{[-1]} \otimes\left(x_{[0]} \trl b\right)$\\
$+\si(x, b\ppi)\ot b\pii+b\li\ot (x\trl b\lii)+b\loo\ot xb\lmi\Big)$,

\item[(CDM20)]$\psi(a\ppl y) +\rho(a \trr y)-\tau \phi(a\ppl y) -\tau\gamma(a \trr y)$\\
$=a_{(0)} \otimes a_{(1)} y+\left(a\ppl y_{1}\right) \otimes y_{2}+a_{1} \otimes\left(a_{2} \trr y\right)$\\
$+a y_{[-1]} \otimes y_{[0]}+y\boi \ot (a\trr y\poo)+y\qi \ot \tht(a, y\qii)$\\
$-\tau\Big(a_{(-1)} \otimes\left(a_{(0)}\ppl y\right)+\tht(a, y\qi)\ot y\qii+\left(a\trr y_{[0]}\right)\otimes y_{[1]}$\\
$+a\ppi \ot \si(a\pii, y)+y\li \ot (a\ppl y\lii)+ y\poo\ot ay\bi\Big)$,
\end{enumerate}

\noindent then $(A, H, \sigma, \theta, P, Q)$ is called a \emph{cocycle double matched pair}.
\end{definition}

\begin{definition}\label{cocycle-braided}
(i) A \emph{cocycle braided  left-symmetric bialgebra} $A$ is simultaneously a cocycle left-symmetric algebra $(A, \theta)$ and  a cycle left-symmetric coalgebra $(A, Q)$ satisfying the  conditions
\begin{enumerate}
\item[(CBB1)] $\Delta_{A}([a, b])+Q\theta(a, b)-Q\theta(b, a)$\\
$=  a b\li\ot b\lii+b\li\ot [a, b\lii]-b a\li\ot a\lii -a\li \ot [b, a\lii]+a_{(0)} \otimes\left(a_{(1)} \ppr b\right)+(a\ppl b_{(-1)}) \otimes b_{(0)}$\\
$+b_{(0)} \otimes (a\ppl b_{(1)}) -b_{(0)} \otimes\left(b_{(1)} \ppr a\right)-(b\ppl a_{(-1)}) \otimes a_{(0)}-a_{(0)} \otimes (b\ppl a_{(1)})$,
\end{enumerate}
\begin{enumerate}
\item[(CBB2)] $(\id-\tau)\left(\Delta_{A}(ab)+Q\theta(a,b)\right)$\\
$=(\id-\tau) \Big(a_{1} \otimes a_{2} b+a b_{1} \otimes b_{2}+ b_{1} \otimes a b_{2}$\\
$+a_{(0)} \otimes\left(a_{(1)} \ppr b\right)+(a\ppl b_{(-1)}) \otimes b_{(0)}+b_{(0)} \otimes (a\ppl b_{(1)})\Big)$.
\end{enumerate}
(ii) A \emph{cocycle braided  left-symmetric bialgebra} $H$ is simultaneously a cocycle  left-symmetric algebra $(H, \sigma)$ and a cycle  left-symmetric coalgebra $(H, P)$ satisfying the conditions
\begin{enumerate}
\item[(CBB3)] $\Delta_{H}([x, y])+P\sigma(x,y)-P\sigma(y,x)$\\
$=x_{1} \otimes [x_{2} ,y]+x y_{1} \otimes y_{2}-y_{1} \otimes [y_{2}, x]-y x_{1} \otimes x_{2}+x_{[0]} \otimes\left(x_{[1]} \trr y\right)+\left(x \trl y_{[-1]}\right) \otimes y_{[0]}$\\
 $+ y_{[0]} \otimes \left(x \trl y_{[1]}\right)-y_{[0]} \otimes\left(y_{[1]} \trr x\right)-\left(y \trl x_{[-1]}\right) \otimes x_{[0]}- x_{[0]} \otimes \left(y \trl x_{[1]}\right)$,
\end{enumerate}
\begin{enumerate}
\item[(CBB4)] $(\id-\tau)\left(\Delta_{H}(xy)+P\sigma(x,y)\right)$\\
$=(\id-\tau)\Big(x_{1} \otimes x_{2} y+x y_{1} \otimes y_{2}+ y_{1} \otimes x y_{2}+x_{[0]} \otimes\left(x_{[1]} \trr y\right)+\left(x \trl y_{[-1]}\right) \otimes y_{[0]}+ y_{[0]} \otimes \left(x \trl y_{[1]}\right)\Big)$.
\end{enumerate}
\end{definition}

It is shown that we can obtain an ordinary  left-symmetric bialgebra from two cocycle braided    left-symmetric bialgebras.
\begin{theorem}\label{main2}
Let $(A, H, \sigma, \theta, P, Q)$ be a cocycle double matched pair,  $(A, H, \sigma, \theta)$ a cocycle cross product system and $(A, H, P, Q)$ a cycle cross coproduct system.
Then the cocycle cross product   algebra and cycle cross coproduct   coalgebra fit together to become an ordinary
   left-symmetric bialgebra if and only if both $A$ and $H$ are  cocycle braided  left-symmetric bialgebras. We will call it the cocycle bicrossproduct   left-symmetric bialgebra and denote it by $A^{P}_{\sigma}\# {}^{Q}_{\theta}H$.
\end{theorem}

\begin{proof}  We  need to check the first compatibility condition $\Delta([(a, x), (b, y)])=\Delta(a, x) \cdot(b, y)+(a, x) \cdot \Delta(b, y)+(a, x) \bullet \Delta(b, y)-\Delta(b, y ) \cdot(a, x)-(b, y ) \cdot \Delta(a, x)-( b, y) \bullet \Delta(a, x)$.
The left hand side is equal to
\begin{eqnarray*}
&&\Delta([(a, x), (b, y)])\\
&=&\Delta((a, x) (b, y))-\Delta((b, y) (a, x))\\
&=&\Delta(a b+x \ppr b+a\ppl y+\sigma(x, y), x y+x \trl b+a\trr y+\theta(a, b))\\
&&-\Delta(ba+y \ppr a+b\ppl x+\sigma(y, x), yx+y \trl a+b\trr x+\theta(b, a))\\
&=&\Delta_A(a b)+\phi(a b)+\psi(a b)+P(ab)\\
&&+\Delta_A(x \ppr b)+\phi(x \ppr b)+\psi(x \ppr b)+P(x \ppr b)\\
&&+\Delta_{A}(a\ppl y)+\phi(a\ppl y)+\psi(a\ppl y)+P(a\ppl y)\\
&&+\Delta_{A}(\sigma(x, y))+\phi(\sigma(x, y))+\psi(\sigma(x, y))+P(\sigma(x, y))\\
&&+\Delta_{H}(x y)+\rho(x y)+\gamma(x y)+Q(xy)\\
&&+\Delta_{H}(x \trl b)+\rho(x \trl b)+\gamma(x \trl b)+Q(x \trl b)\\
&&+\Delta_{H}(a \trr y)+\rho(a \trr y)+\gamma(a \trr y)+Q(a \trr y)\\
&&+\Delta_{H}(\theta(a, b))+\rho(\theta(a, b))+\gamma(\theta(a, b))+Q(\theta(a, b))\\
&&- \Delta_A(ba)- \phi(ba)- \psi(ba)-  P(ba)\\
&&- \Delta_A(y \ppr a)- \phi(y \ppr a)- \psi(y \ppr a)-  P(y \ppr a)\\
&&- \Delta_{A}(b\ppl x)- \phi(b\ppl x)- \psi(b\ppl x)-  P(b\ppl x)\\
&&- \Delta_{A}(\sigma(y, x))- \phi(\sigma(y, x))- \psi(\sigma(y, x))-  P(\sigma(y, x))\\
&&- \Delta_{H}(yx)- \rho(yx)- \gamma( y x)-  Q( y x)\\
&&- \Delta_{H}(y \trl a)- \rho(y \trl a)- \gamma(y \trl a)-  Q(y \trl a)\\
&&- \Delta_{H}(b \trr x)- \rho(b \trr x)- \gamma(b \trr x)-  Q(b \trr x)\\
&&- \Delta_{H}(\theta(b, a))- \rho(\theta(b, a))- \gamma(\theta(b, a))-  Q(\theta(b, a)),
\end{eqnarray*}
and the right hand side is equal to
\begin{eqnarray*}
&&\Delta(a, x) \cdot(b, y)+(a, x) \cdot \Delta(b, y)+(a, x) \bullet \Delta(b, y)\\
&&-\Delta(b, y ) \cdot(a, x)-(b, y ) \cdot \Delta(a, x)-( b, y) \bullet \Delta(a, x)\\
&=&a_{1} \otimes a_{2} b+a_{1} \otimes\left(a_{2}\ppl y\right)+a_{1} \otimes\left(a_{2} \trr y\right)+a_{1} \otimes \theta(a_2, b)\\
&&+a_{(-1)} \otimes a_{(0)} b+a_{(-1)} \otimes\left(a_{(0)}\ppl y\right)+a_{(-1)} \otimes\left(a_{(0)} \trr y\right)+a_{(-1)} \otimes \theta(a_{(0)}, b)\\
&&+a_{(0)}\otimes\left(a_{(1)} \ppr b\right)+a_{(0)} \otimes\left(a_{(1)} \trl b\right)+a_{(0)} \otimes a_{(1)}y+a_{(0)} \otimes \si(a_{(1)}, y)\\
&&+a\ppi \otimes \left(a\pii \ppr b\right)+a\ppi \ot \sigma\left(a\pii, y\right)+a\ppi\otimes a\pii y+a\ppi \otimes \left(a\pii\trl  b\right)\\
&&+x_{1} \otimes\left(x_{2}\ppr b\right)+x_{1} \otimes\left(x_{2} \trl b\right)+x_{1} \otimes x_{2} y +x_{1} \otimes \sigma(x_{2}, y)\\
&&+x\boi\ot(x\boo\ppr b)+x\boi\ot(x\boo\trl b)+x\boi\ot x\boo y+x\boi\ot \sigma(x\boo, y)\\
&&+x\boo\ot x\bi b+x\boo\ot (x\bi\ppl y)+x\boo\ot (x\bi\trr y)+x\boo\ot \theta(x\bi, b)\\
&&+x\qi\otimes x\qii b+x\qi\otimes \left(x\qii \ppl y\right)+x\qi\otimes \left(x\qii \trr y\right)+x\qi \otimes \theta\left(x\qii, b\right)\\
&&+a b_{1} \otimes b_{2}+\left(x \ppr b_{1}\right) \otimes b_{2}+\left(x \trl b_{1}\right) \otimes b_{2}+\theta(a, b_{1})\otimes b_{2}\\
&&+(a\ppl b_{(-1)}) \otimes b_{(0)}+(a \trr b_{(-1)}) \otimes b_{(0)}+x b_{(-1)} \otimes b_{(0)}+\sigma(x, b_{(-1)})\otimes b_{(0)}\\
&&+a b_{(0)} \otimes b_{(1)}+(x \ppr b_{(0)}) \otimes b_{(1)}+(x \trl b_{(0)}) \otimes b_{(1)}+\theta(a, b_{(0)}) \otimes b_{(1)}\\
&&+(a\ppl b\ppi) \otimes b\pii+(a \trr b\ppi) \otimes b\pii+x b\ppi \otimes b\pii+\sigma(x, b\ppi)\otimes b\pii\\
&&+\left(a\ppl y_{1}\right) \otimes y_{2}+\left(a \trr y_{1}\right) \otimes y_{2}+x y_{1}\ot y_{2}+\sigma(x, y_{1})\ot y_{2}\\
&&+ay\boi\ot y\boo+(x\ppr y\boi)\ot y\boo+(x\trl y\boi)\ot y\boo+\theta(a,y\boi)\ot y\boo\\
&&+(a\ppl y\boo)\ot y\bi+(a\trr y\boo)\ot y\bi+xy\boo\ot y\bi+\sigma(x, y\boo)\ot y\bi\\
&&+(ay\qi)\ot y\qii+(x\ppr y\qi)\ot y\qii+(x\trl y\qi)\ot y\qii+\theta(a, y\qi)\ot y\qii\\
&&+ b_{1}\otimes a b_{2}+b_{1} \otimes \left(x \ppr b_{2}\right)+ b_{1}\otimes \left(x \trl b_{2}\right)+b_{1}\otimes\theta(a, b_{2}) \\
&&+ b_{(0)}\otimes (a\ppl b_{(1)})+ b_{(0)}\otimes (a \trr b_{(1)})+ b_{(0)}\otimes x b_{(1)}+b_{(0)}\otimes \sigma(x, b_{(1)})\\
&&+b_{(-1)}\otimes a b_{(0)} + b_{(-1)}\otimes(x \ppr b_{(0)})+b_{(-1)}\otimes (x \trl b_{(0)}) + b_{(-1)}\otimes\theta(a, b_{(0)}) \\
&&+b\ppi\otimes (a\ppl b\pii) + b\ppi\otimes(a \trr b\pii)+ b\ppi\otimes x b\pii+ b\ppi\otimes\sigma(x, b\pii)\\
&&+ y_{1}\otimes\left(a\ppl y_{2}\right)+ y_{1}\otimes \left(a \trr y_{2}\right)+y_{1}\ot x y_{2}+y_{1}\ot \sigma(x, y_{2})\\
&&+y\boo\ot ay\bi+ y\boo\ot(x\ppr y\bi)+ y\boo\ot(x\trl y\bi)+ y\boo\ot\theta(a,y\bi)\\
&&+ y\boi\ot(a\ppl y\boo)+y\boi\ot (a\trr y\boo)+y\boi\ot xy\boo+ y\boi\ot\sigma(x, y\boo)\\
&&+y\qi\ot (ay\qii)+ y\qi\ot x\ppr y\qii+ y\qi\ot x\trl y\qii+ y\qi\ot\theta(a, y\qii)\\
&&-b_{1} \otimes b_{2} a-b_{1} \otimes\left(b_{2}\ppl x\right)-b_{1} \otimes\left(b_{2} \trr x\right)-b_{1} \otimes \tht(b_2, a)\\
&&-b_{(-1)} \otimes b_{(0)} a-b_{(-1)} \otimes\left(b_{(0)}\ppl x\right)-b_{(-1)} \otimes\left(b_{(0)} \trr x\right)-b_{(-1)} \otimes \tht(b_{(0)}, a)\\
&&-b_{(0)}\otimes\left(b_{(1)} \ppr a\right)-b_{(0)} \otimes\left(b_{(1)} \trl a\right)-b_{(0)} \otimes b_{(1)}x-b_{(0)} \otimes \si(b_{(1)}, x)\\
&&-b\ppi \otimes \left(b\pii \ppr a\right)-b\ppi \ot \si\left(b\pii, x\right)-b\ppi\otimes b\pii x-b\ppi \otimes \left(b\pii\trl  a\right)\\
&&-y_{1} \otimes\left(y_{2}\ppr a\right)-y_{1} \otimes\left(y_{2} \trl a\right)-y_{1} \otimes y_{2} x -y_{1} \otimes \si(y_{2}, x)\\
&&-y\boi\ot(y\boo\ppr a)-y\boi\ot(y\boo\trl a)-y\boi\ot y\boo x-y\boi\ot \si(y\boo, x)\\
&&-y\boo\ot y\bi a-y\boo\ot (y\bi\ppl x)-y\boo\ot (y\bi\trr x)-y\boo\ot \tht(y\bi, a)\\
&&-y\qi\otimes y\qii a-y\qi\otimes \left(y\qii \ppl x\right)-y\qi\otimes \left(y\qii \trr x\right)-y\qi \otimes \tht\left(y\qii, a\right)\\
&&-b a_{1} \otimes a_{2}-\left(y \ppr a_{1}\right) \otimes a_{2}-\left(y \trl a_{1}\right) \otimes a_{2}-\tht(b, a_{1})\otimes a_{2}\\
&&-(b\ppl a_{(-1)}) \otimes a_{(0)}-(b \trr a_{(-1)}) \otimes a_{(0)}-y a_{(-1)} \otimes a_{(0)}-\si(y, a_{(-1)})\otimes a_{(0)}\\
&&-b a_{(0)} \otimes a_{(1)}-(y \ppr a_{(0)}) \otimes a_{(1)}-(y \trl a_{(0)}) \otimes a_{(1)}-\tht(b, a_{(0)}) \otimes a_{(1)}\\
&&-(b\ppl a\ppi) \otimes a\pii-(b \trr a\ppi) \otimes a\pii-y a\ppi \otimes a\pii-\si(y, a\ppi)\otimes a\pii\\
&&-\left(b\ppl x_{1}\right) \otimes x_{2}-\left(b \trr x_{1}\right) \otimes x_{2}-y x_{1}\ot x_{2}-\si(y, x_{1})\ot x_{2}\\
&&-bx\boi\ot x\boo-(y\ppr x\boi)\ot x\boo-(y\trl x\boi)\ot x\boo-\tht(b,x\boi)\ot x\boo\\
&&-(b\ppl x\boo)\ot x\bi-(b\trr x\boo)\ot x\bi-yx\boo\ot x\bi-\si(y, x\boo)\ot x\bi\\
&&-(bx\qi)\ot x\qii-(y\ppr x\qi)\ot x\qii-(y\trl x\qi)\ot x\qii-\tht(b, x\qi)\ot x\qii\\
&&- a_{1}\otimes b a_{2}-a_{1} \otimes \left(y \ppr a_{2}\right)- a_{1}\otimes \left(y \trl a_{2}\right)-a_{1}\otimes\tht(b, a_{2}) \\
&&- a_{(0)}\otimes (b\ppl a_{(1)})- a_{(0)}\otimes (b \trr a_{(1)})- a_{(0)}\otimes y a_{(1)}-a_{(0)}\otimes \si(y, a_{(1)})\\
&&-a_{(-1)}\otimes b a_{(0)} - a_{(-1)}\otimes(y \ppr a_{(0)})-a_{(-1)}\otimes (y \trl a_{(0)}) - a_{(-1)}\otimes\tht(b, a_{(0)}) \\
&&-a\ppi\otimes (b\ppl a\pii) - a\ppi\otimes(b \trr a\pii)- a\ppi\otimes y a\pii- a\ppi\otimes\si(y, a\pii)\\
&&- x_{1}\otimes\left(b\ppl x_{2}\right)- x_{1}\otimes \left(b \trr x_{2}\right)-x_{1}\ot y x_{2}-x_{1}\ot \si(y, x_{2})\\
&&-x\boo\ot bx\bi- x\boo\ot(y\ppr x\bi)- x\boo\ot(y\trl x\bi)- x\boo\ot\tht(b,x\bi)\\
&&- x\boi\ot(b\ppl x\boo)-x\boi\ot (b\trr x\boo)-x\boi\ot yx\boo- x\boi\ot\si(y, x\boo)\\
&&-x\qi\ot (bx\qii)- x\qi\ot (y\ppr x\qii)- x\qi\ot (y\trl x\qii)- x\qi\ot\tht(b, x\qii),
\end{eqnarray*}

By using the  cocycle double matched pair conditions (CDM1)--(CDM10) in
Definition \ref{cocycledmp}, we find that the two sides are equal to each other if and only if $A$ and $H$ satisfy the fist compatibility condition of cocycle braided left-symmetric bialgebra respectively.
Then, we check  $(\id-\tau)(\Delta( (a, x), (b, y) ))=(\id-\tau)(\Delta(a, x) \cdot(b, y)+(a, x) \cdot \Delta(b, y)+(a, x) \bullet \Delta(b, y))$,  the left hand side is equal to
\begin{eqnarray*}
&&(\id-\tau)(\Delta( (a, x), (b, y) ))\\
&=&(\id-\tau)\Delta(a b+x \ppr b+a\ppl y+\sigma(x, y), x y+x \trl b+a\trr y+\theta(a, b))\\
&=& \Delta_A(a b)+\phi(a b)+\psi(a b)+\Delta_A(x \ppr b)+\phi(x \ppr b)+\psi(x \ppr b)\\
&&+\Delta_{A}(a\ppl y)+\phi(a\ppl y)+\psi(a\ppl y)+\Delta_{H}(x y)+\rho(x y)+\gamma(x y)\\
&&+\Delta_{H}(x \trl b)+\rho(x \trl b)+\gamma(x \trl b)+\Delta_{H}(a \trr y)+\rho(a \trr y)+\gamma(a \trr y)\\
&&+\Delta_{H}(\theta(a,b))+\rho(\theta(a,b))+\gamma(\theta(a,b))+P(ab)+P(x \ppr b)+P(a\ppl y)+P(\sigma(x, y))\\
&&+\Delta_{A}(\sigma(x,y))+\phi(\sigma(x,y))+\psi(\sigma(x,y))+Q(xy)+Q(x \trl b)+Q(a\trr y)+Q(\theta(a, b))\\
&&-\tau\Delta_A(a b)-\tau\phi(a b)-\tau\psi(a b)-\tau\Delta_A(x \ppr b)-\tau\phi(x \ppr b)-\tau\psi(x \ppr b)\\
&&-\tau\Delta_{A}(a\ppl y)-\tau\phi(a\ppl y)-\tau\psi(a\ppl y)-\tau\Delta_{H}(x y)-\tau\rho(x y)-\tau\gamma(x y)\\
&&-\tau\Delta_{H}(x \trl b)-\tau\rho(x \trl b)-\tau\gamma(x \trl b)-\tau\Delta_{H}(a \trr y)-\tau\rho(a \trr y)-\tau\gamma(a \trr y)\\
&&-\tau\Delta_{H}(\theta(a,b))-\tau\rho(\theta(a,b))-\tau\gamma(\theta(a,b))-\tau P(ab)-\tau P(x \ppr b)-\tau P(a\ppl y)-\tau P(\sigma(x, y))\\
&&-\tau\Delta_{A}(\sigma(x,y))-\tau\phi(\sigma(x,y))-\tau\psi(\sigma(x,y))-\tau Q(xy)-\tau Q(x \trl b)-\tau Q(a\trr y)-\tau Q(\theta(a, b)) ,
\end{eqnarray*}
and the right hand side is equal to
\begin{eqnarray*}
&&(\id-\tau)\Delta(a, x) \cdot(b, y)+(a, x) \cdot \Delta(b, y)+(a, x) \bullet \Delta(b, y)\\
&=&(\id-\tau)\Big(a_{1} \otimes a_{2} b+a_{1} \otimes\left(a_{2}\ppl y\right)+a_{1} \otimes\left(a_{2} \trr y\right)+a_{1} \otimes \theta(a_2, b)\\
&&+a_{(-1)} \otimes a_{(0)} b+a_{(-1)} \otimes\left(a_{(0)}\ppl y\right)+a_{(-1)} \otimes\left(a_{(0)} \trr y\right)+a_{(-1)} \otimes \theta(a_{(0)}, b)\\
&&+a_{(0)}\otimes\left(a_{(1)} \ppr b\right)+a_{(0)} \otimes\left(a_{(1)} \trl b\right)+a_{(0)} \otimes a_{(1)}y+a_{(0)} \otimes \si(a_{(1)}, y)\\
&&+a\ppi \otimes \left(a\pii \ppr b\right)+a\ppi \ot \sigma\left(a\pii, y\right)+a\ppi\otimes a\pii y+a\ppi \otimes \left(a\pii\trl  b\right)\\
&&+x_{1} \otimes\left(x_{2}\ppr b\right)+x_{1} \otimes\left(x_{2} \trl b\right)+x_{1} \otimes x_{2} y +x_{1} \otimes \sigma(x_{2}, y)\\
&&+x\boi\ot(x\boo\ppr b)+x\boi\ot(x\boo\trl b)+x\boi\ot x\boo y+x\boi\ot \sigma(x\boo, y)\\
&&+x\boo\ot x\bi b+x\boo\ot (x\bi\ppl y)+x\boo\ot (x\bi\trr y)+x\boo\ot \theta(x\bi, b)\\
&&+x\qi\otimes x\qii b+x\qi\otimes \left(x\qii \ppl y\right)+x\qi\otimes \left(x\qii \trr y\right)+x\qi \otimes \theta\left(x\qii, b\right)\\
&&+a b_{1} \otimes b_{2}+\left(x \ppr b_{1}\right) \otimes b_{2}+\left(x \trl b_{1}\right) \otimes b_{2}+\theta(a, b_{1})\otimes b_{2}\\
&&+(a\ppl b_{(-1)}) \otimes b_{(0)}+(a \trr b_{(-1)}) \otimes b_{(0)}+x b_{(-1)} \otimes b_{(0)}+\sigma(x, b_{(-1)})\otimes b_{(0)}\\
&&+a b_{(0)} \otimes b_{(1)}+(x \ppr b_{(0)}) \otimes b_{(1)}+(x \trl b_{(0)}) \otimes b_{(1)}+\theta(a, b_{(0)}) \otimes b_{(1)}\\
&&+(a\ppl b\ppi) \otimes b\pii+(a \trr b\ppi) \otimes b\pii+x b\ppi \otimes b\pii+\sigma(x, b\ppi)\otimes b\pii\\
&&+\left(a\ppl y_{1}\right) \otimes y_{2}+\left(a \trr y_{1}\right) \otimes y_{2}+x y_{1}\ot y_{2}+\sigma(x, y_{1})\ot y_{2}\\
&&+ay\boi\ot y\boo+(x\ppr y\boi)\ot y\boo+(x\trl y\boi)\ot y\boo+\theta(a,y\boi)\ot y\boo\\
&&+(a\ppl y\boo)\ot y\bi+(a\trr y\boo)\ot y\bi+xy\boo\ot y\bi+\sigma(x, y\boo)\ot y\bi\\
&&+ay\qi\ot y\qii+(x\ppr y\qi)\ot y\qii+(x\trl y\qi)\ot y\qii+\theta(a, y\qi)\ot y\qii\\
&&+ b_{1}\otimes a b_{2}+b_{1} \otimes \left(x \ppr b_{2}\right)+ b_{1}\otimes \left(x \trl b_{2}\right)+b_{1}\otimes\theta(a, b_{2}) \\
&&+ b_{(0)}\otimes (a\ppl b_{(1)})+ b_{(0)}\otimes (a \trr b_{(1)})+ b_{(0)}\otimes x b_{(1)}+b_{(0)}\otimes \sigma(x, b_{(1)})\\
&&+b_{(-1)}\otimes a b_{(0)} + b_{(-1)}\otimes(x \ppr b_{(0)})+b_{(-1)}\otimes (x \trl b_{(0)}) + b_{(-1)}\otimes\theta(a, b_{(0)}) \\
&&+b\ppi\otimes (a\ppl b\pii) + b\ppi\otimes(a \trr b\pii)+ b\ppi\otimes x b\pii+ b\ppi\otimes\sigma(x, b\pii)\\
&&+ y_{1}\otimes\left(a\ppl y_{2}\right)+ y_{1}\otimes \left(a \trr y_{2}\right)+y_{1}\ot x y_{2}+y_{1}\ot \sigma(x, y_{2})\\
&&+y\boo\ot ay\bi+ y\boo\ot(x\ppr y\bi)+ y\boo\ot(x\trl y\bi)+ y\boo\ot\theta(a,y\bi)\\
&&+ y\boi\ot(a\ppl y\boo)+y\boi\ot (a\trr y\boo)+y\boi\ot xy\boo+ y\boi\ot\sigma(x, y\boo)\\
&&+y\qi\ot ay\qii+ y\qi\ot (x\ppr y\qii)+ y\qi\ot( x\trl y\qii)+ y\qi\ot\theta(a, y\qii)\Big).
\end{eqnarray*}
Now using the  cocycle double matched pair conditions (CDM11)--(CDM20) in
Definition \ref{cocycledmp}, we find that the two sides are equal to each other if and only if $A$ and $H$ satisfy the second compatibility condition of cocycle braided left-symmetric bialgebra.
 This complete the proof.
\end{proof}

\section{Extending structures for   left-symmetric bialgebras}
In this section, we will study the extending problem for  left-symmetric bialgebras.
We will find some special cases when the braided left-symmetric bialgebra is deduced into an ordinary    left-symmetric bialgebra.
It is proved that the extending problem can be solved by using of the non-abelian cohomology theory based on our cocycle bicrossedproduct for braided left-symmetric bialgebras in last section.

\subsection{Extending structures for left-symmetric algebras }
First we are going to study extending problem for left-symmetric algebras and left-symmetric coalgebras.

There are two cases for $A$ to be a left-symmetric algebra in the cocycle cross product system defined in last section, see condition (CC6). The first case is when we let $\ppr$, $\ppl$ to be trivial and $\theta\neq 0$,  then from condition (CP8) we get $\si(\theta(a, b),x)-\si(\theta(b, a), x)=0$, since $\theta\neq 0$ we assume $\sigma=0$ for simplicity, thus  we obtain the following type $(a1)$  unified product for left-symmetric algebras.

\begin{lemma}
Let ${A}$ be a left-symmetric algebra and $V$  a vector space. An extending datum of ${A}$ by $V$ of type $(a1)$  is  $\Omega^{(1)}({A}, V)=(\trr, \, \trl, \, \theta, \, \cdot)$ consisting of bilinear maps
\begin{eqnarray*}
\trr: A\otimes V \to V, \quad \trl: V\otimes A \to V, \quad\theta: A\ot A\to V, \quad\cdot: V\ot V\to V.
\end{eqnarray*}
Denote by $A_{}\#_{\theta}V$ the vector space $E={A}\oplus V$ together with the product given by
\begin{eqnarray}
(a, x)(b, y)=\big(ab, \, xy+x\trl b+a\trr y+\theta(a, b)\big).
\end{eqnarray}
Then $A_{}\# {}_{\theta}V$ is a left-symmetric algebra if and only if the following compatibility conditions hold for all $a, b\in {A}$, $x, y, z\in V$:
\begin{enumerate}
\item[(A1)] $[a, b] \trr x+(\theta(a, b)  -\theta(b, a)) x=a\trr (b\trr x)-b\trr (a\trr x) $,
\item[(A2)] $x \trl (a b)+x\theta(a, b)=a \trr (x\trl b) +(x \trl a-a\trr x) \trl b $,
\item[(A3)] $ a \trr (x y) =(a \trr x-x\trl a) y +x(a\trr y)$,
\item[(A4)] $[x, y] \trl a  =  x(y \trl a)-y(x \trl a)$,
\item[(A5)]$\theta(a b, c)-\theta(a, b c)+\theta(a, b) \triangleleft c-a\trr \theta(b, c)=\theta(ba, c)-\theta(b, a c)+\theta(b,a) \triangleleft c-b\trr \theta(a, c),$
\item[(A6)] $(x y)z-x(yz)=(yx)z-y(xz)$,
\end{enumerate}
\end{lemma}
Note that (A1)--(A4)  are deduced from (CP1)--(CP4) and by (A6)  we obtain that $V$ is a left-symmetric algebra. Furthermore, $V$ is in fact a left-symmetric subalgebra of $A_{}\#_{\theta}V$  but $A$ is not although $A$ is itself a left symmetric algebra.

Denote the set of all  algebraic extending datum of ${A}$ by $V$ of type $(a1)$  by $\Omega^{(1)}({A}, V)$.

In the following, we always assume that $A$ is a subspace of a vector space $E$, there exists a projection map $p: E \to{A}$ such that $p(a) = a$, for all $a \in {A}$.
Then the kernel space $V:= \ker(p)$ is also a subspace of $E$ and a complement of ${A}$ in $E$.

\begin{lemma}\label{lem:33-1}
Let ${A}$ be a left-symmetric  algebra and $E$ be a vector space containing ${A}$ as a   subspace.
Suppose that there is a left-symmetric algebra structure on $E$ such that $V$ is a left-symmetric subalgebra of $E$
and the canonical projection map $p: E\to A$ is a left-symmetric algebra homomorphism.
Then there exists a left-symmetric algebraic extending datum $\Omega^{(1)}({A},V)$ of ${A}$ by $V$ such that $E\cong A_{}\#_{\theta}V$.
\end{lemma}

\begin{proof}
Since $V$ is a left-symmetric subalgebra of $E$, we have $x\cdot_E y\in V$ for all $x, y\in V$.
We define the extending datum of ${A}$ through $V$ by the following formulas:
\begin{eqnarray*}
\trr: {A}\otimes V \to V, \qquad {a} \trr {x} &:=&a\cdot_E x-p(a\cdot_E {x}),\\
\trl: V\otimes {A} \to V, \qquad {x} \trl {a} &:=&{x}\cdot_E a-p({x}\cdot_E a),\\
\theta: A\otimes A \to V, \qquad \theta(a,b) &:=&p(a)\cdot_E p(b)-p \bigl(a\cdot_E b\bigl),\\
{\cdot_V}: V \otimes V \to V, \qquad {x}\cdot_V {y}&:=& {x}\cdot_E{y} .
\end{eqnarray*}
for any $a , b\in {A}$ and $x, y\in V$. It is easy to see that the above maps are  well defined and
$\Omega^{(1)}({A}, V)$ is an extending system of
${A}$ trough $V$ and
\begin{eqnarray*}
\varphi:A_{}\#_{\theta}V\to E, \qquad \varphi(a, x) := a+x
\end{eqnarray*}
is an isomorphism of  left-symmetric algebras.
\end{proof}

\begin{lemma}
Let $\Omega^{(1)}(A, V) = \bigl(\trr,  \, \trl,  \, \theta,  \, \cdot \bigl)$ and $\Omega'^{(1)}(A, V) = \bigl(\trr', \, \trl',  \, \theta',  \, \cdot' \bigl)$
be two algebraic extending datums of ${A}$ by $V$ of type (a1) and $A_{}\#_{\theta} V$, $A_{}\#_{\theta'} V$ be the corresponding unified products. Then there exists a bijection between the set of all homomorphisms of left-symmetric algebras $\varphi: A_{\theta}\#_{\trr, \trl} V\to A_{\theta'}\#_{\trr', \trl'} V$ whose restriction on ${A}$ is the identity map and the set of pairs $(r, s)$, where $r: V\rightarrow {A}$ and $s: V\rightarrow V$ are two linear maps satisfying
\begin{eqnarray}
&&{r}(x\trl a)={r}(x)\cdot' a,\\
&&{r}(a\trr x)= a\cdot'{r}(x),\\
&&a\cdot' b=ab+r\theta(a,  b),\\
&&{r}(xy)={r}(x)\cdot' {r}(y),\\
&&{s}(x)\trl' a+\theta'(r(x), a)={s}(x\trl a),\\
&&a\trr'{s}(y)+\theta'(a, r(y) )={s}(a\trr y),\\
&&\theta'(a, b)=s\theta(a, b),\\
&&{s}(xy)={s}(x)\cdot' {s}(y)+{s}(x)\trl'{r}(y)+{r}(x)\trr'{s}(y)+\theta'(r(x), r(y)).
\end{eqnarray}

for all $a\in{A}$ and $x, y\in V$.

Under the above bijection the homomorphism of  left-symmetric algebras $\varphi=\varphi_{r, s}: A_{}\#_{\theta}V\to A_{}\#_{\theta'} V$ to $(r, s)$ is given  by $\varphi(a, x)=(a+r(x), s(x))$ for all $a\in {A}$ and $x\in V$. Moreover, $\varphi=\varphi_{r, s}$ is an isomorphism if and only if $s: V\rightarrow V$ is a linear isomorphism.
\end{lemma}

\begin{proof}
Let $\varphi: A_{}\#_{\theta}V\to A_{ }\#_{\theta'} V$  be an algebra homomorphism  whose restriction on ${A}$ is the identity map. Then $\varphi$ is determined by two linear maps $r: V\rightarrow {A}$ and $s: V\rightarrow V$ such that
$\varphi(a, x)=(a+r(x), s(x))$ for all $a\in {A}$ and $x\in V$.
In fact, we have to show
$$\varphi((a, x)(b, y))=\varphi(a, x)\cdot'\varphi(b, y).$$
The left hand side is equal to
\begin{eqnarray*}
&&\varphi((a, x)(b, y))\\
&=&\varphi\left({ab}, \,  x\trl b+a\trr y+{xy}+\theta(a, b)\right)\\
&=&\big({ab}+ r(x\trl b)+r(a\trr y)+r({xy})+r\theta(a, b), \\
&&\qquad\quad s(x\trl b)+s(a\trr y)+s({xy})+s\theta(a, b)\big),
\end{eqnarray*}
and the right hand side is equal to
\begin{eqnarray*}
&&\varphi(a, x)\cdot' \varphi(b, y)\\
&=&(a+r(x),\,  s(x))\cdot'  (b+r(y), \,  s(y))\\
&=&\big((a+r(x))\cdot' (b+r(y)),  \, \,  s(x)\trl'(b+r(y))+(a+r(x))\trr's(y)\\
&&\qquad\qquad +s(x)\cdot' s(y)+\theta'(a+r(x), b+r(y))\big).
\end{eqnarray*}
Thus $\varphi$ is a homomorphism of algebras if and only if the above conditions hold.
\end{proof}

The second case is when $\theta=0$,  we obtain the following type $(a2)$  unified product for associative algebras which was developed in \cite{AM6}.

\begin{theorem}\cite{AM6}
Let $A$ be a left-symmetric algebra and $V$ be a  vector space. An \textit{extending
datum of $A$ through $V$} of type $(a2)$  is a system $\Omega^{(2)}(A, V) =
\bigl(\triangleleft, \, \triangleright, \, \leftharpoonup, \,
\rightharpoonup, \, \sigma, \, \cdot \bigl)$ consisting of six bilinear maps
\begin{eqnarray*}
&&\ppr: V \otimes A \to A, \quad \ppl: A \otimes V \to A, \quad\triangleleft : V \otimes A \to V,\\
&&\trr: A \otimes V \to V ,
 \quad\sigma: V\otimes V \to A,\quad\cdot: V\ot V\to V.
\end{eqnarray*}
Denote by $A_{\sigma}\# {}_{}H$ the vector space $E={A}\oplus V$ together with the product
\begin{align}
(a, x)(b, y)=\big(ab+x\ppr b+a\ppl y+\sigma(x, y), \, xy+x\trl b+a\trr y\big).
\end{align}
Then $A_{\sigma}\# {}_{}H$  is an algebra if and only if the following compatibility conditions hold for any $a, b, c\in A$,
$x, y, z\in V$:
\begin{enumerate}
\item[(B1)] $[a, b] \trr x  =a\trr (b\trr x)-b\trr (a\trr x)  $,
\item[(B2)] $x \trl (a b) =a \trr (x\trl b) +(x \trl a-a\trr x) \trl b $,
\item[(B3)] $ a \trr (x y) =(a \trr x-x\trl a) y+(a\ppl x-x\ppr a) \trr y+x\trl (a\ppl y)+x(a\trr y)$,
\item[(B4)] $[x, y] \trl a  =x \trl(y \ppr a)-y \trl(x \ppr a)+x(y \trl a)-y(x \trl a)$,
\item[(B5)] $[x, y] \ppr a +(\sigma(x, y)  -\sigma(y, x) )a=x\ppr(y\ppr a)-y\ppr(x\ppr a)+\sigma(x, y\trl a)-\sigma(y, x\trl a)$,
\item[(B6)] $a\ppl (x y)+a \sigma(x, y)=x\ppr (a\ppl y) +(a\ppl x-x\ppr a)\ppl y +\sigma(a \trr x-x\trl a, y)+\sigma(x, a\trr y)$,
\item[(B7)] $x \ppr(a b) =(x\ppr a-a\ppl x) b+(x\trl a-a\trr x) \ppr b+a(x\ppr b)+a\ppl(x\trl b)$,
\item[(B8)] $[a,b]\ppl x  =a(b\ppl x)+a\ppl (b \trr x)-b(a\ppl x)-b\ppl(a\trr x)$,
\item[(B9)] $\sigma(x y, z)-{\sigma}(x, y z)+\sigma(x, y)\ppl z-x \ppr \sigma(y, z)\\
=\sigma(yx, z)-\sigma(y, xz)+\sigma(y, x)\ppl z-y\ppr(x,z),$
\item[(B10)] $(x y) z-x(y z)+\sigma(x, y) \trr z-x \triangleleft \sigma(y, z)=(yx) z-y(x z)+\sigma(y, x) \trr z-y \triangleleft \sigma(x, z)$.
\end{enumerate}
\end{theorem}

\begin{theorem}\cite{AM6}
Let $A$ be a left-symmetric algebra and $E$ be a  vector space containing $A$ as a subspace.
If there is a left-symmetric algebra structure on $E$ such that $A$ is a left-symmetric subalgebra of $E$. Then there exists a left-symmetric algebraic extending structure $\Omega(A, V)^{(2)} = \bigl(\triangleleft, \, \triangleright, \,
\leftharpoonup, \, \rightharpoonup, \, \sigma \bigl)$ of $A$ through $V$ such that there is an isomorphism of   left-symmetric algebras $E\cong A_{\sigma}\#_{}H$.
\end{theorem}

\begin{lemma}\cite{AM6}
Let $\Omega^{(2)}(A, \,  V) = \bigl(\trr, \, \trl,  \, \leftharpoonup,  \,  \rightharpoonup, \,   \sigma,  \,  \cdot \bigl)$ and
$\Omega'^{(2)}(A,  \, V) = \bigl(\trr', \,  \trl', \,  \leftharpoonup', \,   \rightharpoonup', \,  \sigma', \,  \cdot' \bigl)$
be two  algebraic extending structures of $A$ through $V$ and $A{}_{\sigma}\#_{}V$, $A{}_{\sigma'}\#_{}  V$ the  associated unified
products. Then there exists a bijection between the set of all
homomorphisms of algebras $\psi: A{}_{\sigma}\#_{}V\to A{}_{\sigma'}\#_{}  V$ which
stabilize $A$ and the set of pairs $(r, s)$, where $r: V \to A$, $s: V \to V$ are linear maps satisfying the following compatibility conditions for any $x \in A$, $u$, $v \in V$:
\begin{eqnarray}
 &&r(x \cdot y) = r(x)\cdot'r(y) + \sigma ' (s(x), s(y)) - \sigma(x, y) + r(x) \ppl' s(y) + s(x) \ppr' r(y),\\
 &&s(x \cdot y) = r(x) \trr ' s(y) + s(x)\trl ' r(y) + s(x) \cdot ' s(y),\\
 &&r(x\trl  {a}) = r(x)\cdot' {a} - x \ppr {a} + s(x) \ppr' {a},\\
 &&r({a} \trr x) = {a}\cdot'r(x) - {a}\ppl x + {a} \ppl' s(x),\\
 &&s(x\trl {a}) = s(x)\trl' {a},\\
 &&s({a}\trr x) = {a} \trr' s(x).
\end{eqnarray}
Under the above bijection the homomorphism of algebras $\varphi =\varphi _{(r, s)}: A_{\sigma}\# {}_{}H \to A_{\sigma'}\# {}_{}H$ corresponding to
$(r, s)$ is given for any $a\in A$ and $x \in V$ by:
$$\varphi(a, x) = (a + r(x), s(x))$$
Moreover, $\varphi  = \varphi _{(r, s)}$ is an isomorphism if and only if $s: V \to V$ is an isomorphism linear map.
\end{lemma}

Let ${A}$ be a left-symmetric algebra and $V$ be a vector space. Two algebraic extending systems $\Omega^{(i)}({A}, V)$ and ${\Omega'^{(i)}}({A}, V)$  are called equivalent if $\varphi_{r,s}$ is an isomorphism.  We denote it by $\Omega^{(i)}({A}, V)\equiv{\Omega'^{(i)}}({A}, V)$.
From the above lemmas, we obtain the following result.

\begin{theorem}\label{thm3-1}
Let ${A}$ be a left-symmetric algebra and $E$ be a vector space containing ${A}$ as a subspace and $V$ be a complement of ${A}$ in $E$.
Denote $\mathcal{HA}(V, {A}):=\mathcal{A}^{(1)}({A}, V)\sqcup \mathcal{A}^{(2)}({A}, V) /\equiv$. Then the map
\begin{eqnarray}
\notag&&\Psi: \mathcal{HA}(V, {A})\rightarrow Extd(E, {A}),\\
&&\overline{\Omega^{(1)}({A}, V)}\mapsto A_{}\#_{\theta} V,\quad \overline{\Omega^{(2)}({A}, V)}\mapsto A_{\sigma}\# {}_{} V
\end{eqnarray}
is bijective, where $\overline{\Omega^{(i)}({A}, V)}$ is the equivalence class of $\Omega^{(i)}({A}, V)$ under $\equiv$.
\end{theorem}

\subsection{Extending structures for  left-symmetric coalgebras}

Next we consider the left-symmetric coalgebra structures on $E=A^{P}\# {}^{Q}V$.

There are two cases for $(A, \Delta_A)$ to be a left-symmetric coalgebra. The first case is  when $Q=0$,  then we obtain the following type $(c1)$ unified product for left-symmetric coalgebras.
\begin{lemma}\label{cor02}
Let $({A}, \Delta_A)$ be a left-symmetric coalgebra and $V$  be a vector space.
An  extending datum  of ${A}$ by $V$ of  type $(c1)$ is  $\mathcal{C}^{(1)}({A}, V)=(\phi, \,  {\psi}, \, \rho, \, \gamma, \,  P,  \, \Delta_V)$ with six bilinear maps
\begin{eqnarray*}
&&\phi: A \to V \otimes A, \quad  \psi: A \to A\otimes V,\quad \rho: V  \to A\otimes V,\\
&& \gamma: V \to V \otimes A, \quad {P}: A\rightarrow {V}\otimes {V},\quad\Delta_V: V\rightarrow V\otimes V.
\end{eqnarray*}
 Denote by $A^{P}\# {}^{} V$ the vector space $E={A}\oplus V$ with the linear map
$\Delta_E: E\rightarrow E\otimes E$ given by
$$\Delta_{E}(a)=(\Delta_{A}+\phi+\psi+P)(a),\quad \Delta_{E}(x)=(\Delta_{V}+\rho+\gamma)(x), $$
that is
$$\Delta_{E}(a)= a\li \ot a\lii+ a\moi \ot a\mo+a\mo\ot a\mi+a\ppi\ot a\pii,$$
$$\Delta_{E}(x)= x\li \ot x\lii+ x\boi \ot x\boo+x\boo \ot x\bi.$$
Then $A^{P}\# {}^{} V$  is a left-symmetric coalgebra with the coproduct given above if and only if the following compatibility conditions hold:
\begin{enumerate}

\item[(C1)] $\phi\left(a_{1}\right) \otimes a_{2}+\gamma\left(a_{(-1)}\right) \otimes a_{(0)}-a_{(-1)} \otimes \Delta_{A}\left(a_{(0)}\right) \\
    =\tau_{12}\left(\psi\left(a_{1}\right) \otimes a_{2}+\rho\left(a_{(-1)}\right) \otimes a_{(0)}-a_{1} \otimes \phi(a_{2})-a_{(0)} \otimes \gamma(a_{(1)})\right)$,
\item[(C2)] $\Delta_{A}\left(a_{(0)}\right) \otimes a_{(1)} -a_{1} \otimes \psi\left(a_{2}\right)-a_{(0)} \otimes \rho\left(a_{(1)}\right)\\
    =\tau_{12}\left(\Delta_{A}\left(a_{(0)}\right) \otimes a_{(1)} -a_{1} \otimes \psi\left(a_{2}\right)-a_{(0)} \otimes \rho\left(a_{(1)}\right)\right)$,
\item[(C3)] $\rho\left(x_{1}\right) \otimes x_{2}+\psi\left(x_{[-1]}\right) \otimes x_{[0]}-x_{[-1]} \otimes \Delta_{V}\left( x_{[0]}\right)  \\
    =\tau_{12}\left(\gamma\left(x_{1}\right) \otimes x_{2}+\phi\left(x_{[-1]}\right) \otimes x_{[0]}-x_{[0]} \otimes \psi\left(x_{[1]}\right)-x_{1} \otimes \rho(x_{2})\right)$,
\item[(C4)] $\Delta_{V}(x\boo)\ot x\bi -x\boo\ot \phi(x\bi)-x_1\ot \gamma(x_2)\\
=\tau_{12}\left(\Delta_{V}(x\boo)\ot x\bi -x\boo\ot \phi(x\bi)-x_1\ot \gamma(x_2) \right)$,
\item[(C5)] $\Delta_{V}\left(a_{(-1)}\right) \otimes a_{(0)}+P\left(a_{1}\right) \otimes a_{2}-a_{(-1)} \otimes \phi\left(a_{(0)}\right)-a\ppi\otimes \gamma\left(a\pii\right)\\
    =\tau_{12}\left( \Delta_{V}\left(a_{(-1)}\right) \otimes a_{(0)}+P\left(a_{1}\right) \otimes a_{2}-a_{(-1)} \otimes \phi\left(a_{(0)}\right)-a\ppi\otimes \gamma\left(a\pii\right) \right)$,
\item[(C6)] $a_{(0)}\ot \Delta_{V}\left(a_{(1)}\right)+a_{1}\otimes P\left(a_{2}\right)-\psi\left(a_{(0)}\right) \otimes a_{(1)}-\rho\left(a\ppi\right) \otimes a\pii\\
    =\tau_{12}\left(  a_{(-1)} \otimes \psi\left(a_{(0)}\right)+a\ppi \otimes \rho\left(a\pii\right)-\phi\left(a_{(0)}\right) \otimes a_{(1)}-\gamma\left(a\ppi\right) \otimes a\pii \right)$,
\item[(C7)] $x_{[-1]} \otimes \rho\left(x_{[0]}\right) )-\Delta_{A}\left(x_{[-1]}\right) \otimes x_{[0]}=\tau_{12}\left( x_{[-1]} \otimes \rho\left(x_{[0]}\right) -\Delta_{A}\left(x_{[-1]}\right) \otimes x_{[0]}\right) $,
\item[(C8)] $x\boo\ot\Delta_A(x\bi) -\gamma(x\boo)\ot x\bi=\tau_{12}\left(x\boi\ot\gamma(x\boo)-\rho(x\boo)\ot x\bi \right)$,

\item[(C9)]  $\Delta_V(a\ppi)\ot a\pii-a\ppi\ot \Delta_V(a\pii)+P(a\lmoo)\ot a\mi-a\lmoi\ot P(a\lmoo)\\
=\tau_{12}\left(\Delta_V(a\ppi)\ot a\pii-a\ppi\ot \Delta_V(a\pii)+P(a\lmoo)\ot a\mi-a\lmoi\ot P(a\lmoo)\right)$.

\item[(C10)] $\Delta_V(x\li)\ot x\lii-x_1\ot \Delta_V(x_2)+ P(x\boi) \ot x\boo-x\boo\ot P(x\bi)\\
=\tau_{12} \left(\Delta_V(x\li)\ot x\lii-x_1\ot \Delta_V(x_2)+ P(x\boi) \ot x\boo-x\boo\ot P(x\bi) \right)$.
\end{enumerate}
\end{lemma}
Denote the set of all  coalgebraic extending datum of ${A}$ by $V$ of type $(c1)$ by $\mathcal{C}^{(1)}({A}, V)$.

\begin{lemma}\label{lem:33-3}
Let $({A}, \Delta_A)$ be a left-symmetric coalgebra and $E$  a vector space containing ${A}$ as a subspace. Suppose that there is a  left-symmetric  coalgebra structure $(E, \Delta_E)$ on $E$ such that  $p: E\to {A}$ is a  left-symmetric  coalgebra homomorphism. Then there exists a  coalgebraic extending system $\mathcal{C}^{(1)}({A}, V)$ of $({A}, \Delta_A)$ by $V$ such that $(E, \Delta_E)\cong A^{P}\# {}^{} V$.
\end{lemma}

\begin{proof}
Let $p: E\to {A}$ and $\pi: E\to V$ be the projection map and $V=\ker({p})$.
Then the extending datum of $({A}, \Delta_A)$ by $V$ is defined as follows:
\begin{eqnarray*}
&&{\phi}: A\rightarrow V\ot {A},~~~~{\phi}(a)=(\pi\otimes {p})\Delta_E(a),\\
&&{\psi}: A\rightarrow A\ot V,~~~~{\psi}(a)=({p}\otimes \pi)\Delta_E(a),\\
&&{\rho}: V\rightarrow A\ot V,~~~~{\rho}(x)=({p}\otimes \pi)\Delta_E(x),\\
&&{\gamma}: V\rightarrow V\ot {A},~~~~{\gamma}(x)=(\pi\otimes {p})\Delta_E(x),\\
&&\Delta_V: V\rightarrow V\otimes V,~~~~\Delta_V(x)=(\pi\otimes \pi)\Delta_E(x),\\
&&Q: V\rightarrow {A}\otimes {A},~~~~Q(x)=({p}\otimes {p})\Delta_E(x)\\
&&P: A\rightarrow {V}\otimes {V},~~~~P(a)=({\pi}\otimes {\pi})\Delta_E(a).
\end{eqnarray*}
One check that  $\varphi: A^{P}\# {}^{} V\to E$ given by $\varphi(a, x)=a+x$ for all $a\in A, x\in V$ is a  left-symmetric  coalgebra isomorphism.
\end{proof}

\begin{lemma}\label{lem-c1}
Let $\mathcal{C}^{(1)}({A}, V)=(\phi, \, \psi, \, \rho, \, \gamma, \,  P, \, \Delta_V)$ and ${\mathcal{C}'^{(1)}}({A}, V)=(\phi', \,  \psi', \, \rho', \, \gamma', \,   P', \,  \Delta'_V)$ be two  left-symmetric  coalgebraic extending datums of $({A}, \, \Delta_A)$ by $V$. Then there exists a bijection between the set of  left-symmetric   coalgebra homomorphisms $\varphi: A^{P}\# {}^{} V\rightarrow A^{P'}\# {}^{} V$ whose restriction on ${A}$ is the identity map and the set of pairs $(r, s)$, where $r: V\rightarrow {A}$ and $s: V\rightarrow V$ are two linear maps satisfying
\begin{eqnarray}
\label{comorph11}&&P'(a)=s(a\ppi)\ot s(a\pii),\\
\label{comorph121}&&\phi'(a)={s}(a\lmoi)\ot a\lmo+s(a\ppi)\ot r(a\pii),\\
\label{comorph122}&&\psi'(a)=a\lmo\ot {s}(a\mi) +r(a\ppi)\ot s(a\pii),\\
\label{comorph13}&&\Delta'_A(a)=\Delta_A(a)+{r}(a\lmoi)\ot a\lmo+a\lmo\ot {r}(a\mi)+r(a\ppi)\ot r(a\pii)\\
\label{comorph21}&&\Delta_V'({s}(x))+P'(r(x))=({s}\otimes {s})\Delta_V(x),\\
\label{comorph221}&&{\rho}'({s}(x))+\psi'(r(x))=r(x\li)\ot s(x\lii)+x\boi\ot s(x\boo),\\
\label{comorph222}&&{\gamma}'({s}(x))+\phi'(r(x))=s(x\li)\ot r(x\lii)+s(x\boo)\ot x\bi,\\
\label{comorph23}&&\Delta'_A({r}(x))=r(x\li)\ot r(x\lii)+x\boi\ot r(x\boo)+r(x\boo)\ot x\bi.
\end{eqnarray}
Under the above bijection the  left-symmetric  coalgebra homomorphism $\varphi=\varphi_{r, s}: A^{P}\# {}^{} V\rightarrow A^{P'}\# {}^{} V$ to $(r, s)$ is given by $\varphi(a, x)=(a+r(x), s(x))$ for all $a\in {A}$ and $x\in V$. Moreover, $\varphi=\varphi_{r, s}$ is an isomorphism if and only if $s: V\rightarrow V$ is a linear isomorphism.
\end{lemma}
\begin{proof}
Let $\varphi: A^{P}\# {}^{} V\rightarrow A^{P'}\# {}^{} V$  be a   left-symmetric coalgebra homomorphism  whose restriction on ${A}$ is the identity map. Then $\varphi$ is determined by two linear maps $r: V\rightarrow {A}$ and $s: V\rightarrow V$ such that
$\varphi(a, x)=(a+r(x), s(x))$ for all $a\in {A}$ and $x\in V$. We will prove that
$\varphi$ is a homomorphism of  left-symmetric  coalgebras if and only if the above conditions hold.
First it is easy to see that  $\Delta'_E\varphi(a)=(\varphi\otimes \varphi)\Delta_E(a)$ for all $a\in {A}$.
\begin{eqnarray*}
\Delta'_E\varphi(a)&=&\Delta'_E(a)=\Delta'_A(a)+\phi'(a)+\psi'(a)+P'(a),
\end{eqnarray*}
and
\begin{eqnarray*}
&&(\varphi\otimes \varphi)\Delta_E(a)\\
&=&(\varphi\otimes \varphi)\left(\Delta_A(a)+\phi(a)+\psi(a)+P(a)\right)\\
&=&\Delta_A(a)+{r}(a\lmoi)\ot a\lmo+{s}(a\lmoi)\ot a\lmo+a\lmo\ot {r}(a\mi) +a\lmo\ot {s}(a\mi)\\
&&+r(a\ppi)\ot r(a\pii)+r(a\ppi)\ot s(a\pii)+s(a\ppi)\ot r(a\pii)+s(a\ppi)\ot s(a\pii).
\end{eqnarray*}
Thus we obtain that $\Delta'_E\varphi(a)=(\varphi\otimes \varphi)\Delta_E(a)$  if and only if the conditions \eqref{comorph11}, \eqref{comorph121}, \eqref{comorph122} and \eqref{comorph13} hold.
Then we consider that $\Delta'_E\varphi(x)=(\varphi\otimes \varphi)\Delta_E(x)$ for all $x\in V$.
\begin{eqnarray*}
\Delta'_E\varphi(x)&=&\Delta'_E({r}(x),{s}(x))=\Delta'_E({r}(x))+\Delta'_E({s}(x))\\
&=&\Delta'_A({r}(x))+\phi'(r(x))+\psi'(r(x))+P(r(x))+\Delta'_V({s}(x))+{\rho}'({s}(x))+{\gamma}'({s}(x))),
\end{eqnarray*}
and
\begin{eqnarray*}
&&(\varphi\otimes \varphi)\Delta_E(x)\\
&=&(\varphi\otimes \varphi)(\Delta_V(x)+{\rho}(x)+{\gamma}(x))\\
&=&r(x\li)\ot r(x\lii)+r(x\li)\ot s(x\lii)+s(x\li)\ot r(x\lii)+s(x\li)\ot s(x\lii)\\
&&+x\boi\ot r(x\boo)+x\boi\ot s(x\boo)+r(x\boo)\ot x\bi+s(x\boo)\ot x\bi.
\end{eqnarray*}
Thus we obtain that $\Delta'_E\varphi(x)=(\varphi\otimes \varphi)\Delta_E(x)$ if and only if the conditions  \eqref{comorph21},  \eqref{comorph221},  \eqref{comorph222} and \eqref{comorph23}  hold. By definition, we obtain that $\varphi=\varphi_{r, s}$ is an isomorphism if and only if $s: V\rightarrow V$ is a linear isomorphism.
\end{proof}

The second case is $\phi=0$ and  $\psi=0$,and we get $P=0$ when $Q\neq 0$ from (CCP1). We obtain  the following type (c2) unified coproduct for  coalgebras.
\begin{lemma}\label{cor02}
Let $({A}, \, \Delta_A)$ be a  left-symmetric  coalgebra and $V$ be a vector space.
An  extending datum  of $({A}, \, \Delta_A)$ by $V$ of type $(c2)$  is   $\mathcal{C}^{(2)}({A}, V)=(\rho, \, \gamma, \, {Q}, \, \Delta_V)$ with  linear maps
\begin{eqnarray*}
&&\rho: V  \to A\otimes V,\quad  \gamma: V \to V \otimes A,\quad Q: V \to A\otimes A,\quad\Delta_{V}: V \to V\otimes V.
\end{eqnarray*}
 Denote by $A^{}\# {}^{Q} V$ the vector space $E={A}\oplus V$ with the coproduct
$\Delta_E: E\rightarrow E\otimes E$ given by
\begin{eqnarray*}
\Delta_{E}(a)&=&\Delta_{A}(a),\quad \Delta_{E}(x)=(\Delta_{V}+\rho+\gamma+Q)(x), \\
\Delta_{E}(a)&=& a\li \ot a\lii,\quad \Delta_{E}(x)= x\li \ot x\lii+ x\boi \ot x\boo+x\boo \ot x\bi+x\qi\ot x\qii.
\end{eqnarray*}
Then $A^{}\# {}^{Q} V$  is a left-symmetric   coalgebra with the coproduct given above if and only if the following compatibility conditions hold:

\begin{enumerate}
\item[(D1)] $\rho\left(x_{1}\right) \otimes x_{2} -x_{[-1]} \otimes \Delta_{V}\left( x_{[0]}\right)=\tau_{12}\left(\gamma\left(x_{1}\right) \otimes x_{2} -x_{1} \otimes \rho(x_{2})\right)$,

\item[(D2)] $\Delta_{V}(x\boo)\ot x\bi -x_1\ot \gamma(x_2)=\tau_{12}\left(\Delta_{V}(x\boo)\ot x\bi -x_1\ot \gamma(x_2) \right)$,

\item[(D3)] $x_{[-1]} \otimes \rho\left(x_{[0]}\right) -\Delta_{A}\left(x_{[-1]}\right) \otimes x_{[0]}-Q\left(x_{1}\right) \otimes x_{2}\\
    =\tau_{12}\left( x_{[-1]} \otimes \rho\left(x_{[0]}\right) -\Delta_{A}\left(x_{[-1]}\right) \otimes x_{[0]}-Q\left(x_{1}\right) \otimes x_{2} \right)$,

\item[(D4)] $x\boo\ot\Delta_A(x\bi)+x_1\ot Q(x_2)-\gamma(x\boo)\ot x\bi=\tau_{12}\left(x\boi\ot\gamma(x\boo)-\rho(x\boo)\ot x\bi \right)$,

\item[(D5)]  $\Delta_A(x\qi)\ot x\qii-x\qi\ot \Delta_A(x\qii)+Q(x\boo)\ot x\bi-x\boi\ot Q(x\boo)\\
=\tau_{12}\left( \Delta_A(x\qi)\ot x\qii-x\qi\ot \Delta_A(x\qii)+Q(x\boo)\ot x\bi-x\boi\ot Q(x\boo)\right)$,

\item[(D6)] $\Delta_V(x\li)\ot x\lii-x_1\ot \Delta_V(x_2) =\tau_{12} \left(\Delta_V(x\li)\ot x\lii-x_1\ot \Delta_V(x_2) \right)$.
\end{enumerate}
\end{lemma}
Note that in this case $(V, \, \Delta_V)$ is a left-symmetric coalgebra.

Denote the set of all left-symmetric coalgebraic extending datum of ${A}$ by $V$ of type $(c2)$ by $\mathcal{C}^{(2)}({A}, V)$.

Similar to the left-symmetric  algebra case,  one  show that any  left-symmetric  coalgebra structure on $E$ containing ${A}$ as a  left-symmetric  subcoalgebra is isomorphic to such an unified coproduct.
\begin{lemma}\label{lem:33-4}
Let $({A}, \, \Delta_A)$ be a   left-symmetric coalgebra and $E$  a vector space containing ${A}$ as a subspace. Suppose that there is a  left-symmetric  coalgebra structure $(E, \, \Delta_E)$ on $E$ such that  $({A}, \, \Delta_A)$ is a  left-symmetric subcoalgebra of $E$. Then there exists a  coalgebraic extending system $\mathcal{C}^{(2)}({A},  V)$ of $({A}, \Delta_A)$ by $V$ such that $(E, \Delta_E)\cong A^{}\# {}^{Q} V$.
\end{lemma}

\begin{proof}
Let $p: E\to {A}$ and $\pi: E\to V$ be the projection map and $V=ker({p})$.
Then the extending datum of $({A},\Delta_A)$ by $V$ is defined as follows:
\begin{eqnarray*}
&&{\rho}: V\rightarrow A\ot V,~~~~{\phi}(x)=(p\otimes {\pi})\Delta_E(x),\\
&&{\gamma}: V\rightarrow V\ot {A},~~~~{\phi}(x)=(\pi\otimes {p})\Delta_E(x),\\
&&\Delta_V: V\rightarrow V\otimes V,~~~~\Delta_V(x)=(\pi\otimes \pi)\Delta_E(x),\\
&&Q: V\rightarrow {A}\otimes {A},~~~~Q(x)=({p}\otimes {p})\Delta_E(x).
\end{eqnarray*}
One check that  $\varphi: A^{}\# {}^{Q} V\to E$ given by $\varphi(a, x)=a+x$ for all $a\in A, x\in V$ is a  left-symmetric  coalgebra isomorphism.
\end{proof}

\begin{lemma}\label{lem-c2}
Let $\mathcal{C}^{(2)}({A}, V)=(\rho, \, \gamma, \, {Q}, \, \Delta_V)$ and ${\mathcal{C}'^{(2)}}({A}, V)=(\rho', \, \gamma', \,  {Q'}, \, \Delta'_V)$ be two  left-symmetric  coalgebraic extending datums of $({A}, \, \Delta_A)$ by $V$. Then there exists a bijection between the set of  left-symmetric   coalgebra homomorphisms $\varphi: A \# {}^{Q} V\rightarrow A \# {}^{Q'} V$ whose restriction on ${A}$ is the identity map and the set of pairs $(r, s)$, where $r: V\rightarrow {A}$ and $s: V\rightarrow V$ are two linear maps satisfying
\begin{eqnarray}
\label{comorph1}&&{\rho}'({s}(x))=r(x\li)\ot s(x\lii)+x\boi\ot s(x\boo),\\
\label{comorph2}&&{\gamma}'({s}(x))=s(x\li)\ot r(x\lii)+s(x\boo)\ot x\bi,\\
\label{comorph3}&&\Delta_V'({s}(x))=({s}\otimes {s})\Delta_V(x)\\
\label{comorph4}&&\Delta'_A({r}(x))+{Q'}({s}(x))=r(x\li)\ot r(x\lii)+x\boi\ot r(x\boo)+r(x\boo)\ot x\bi+{Q}(x).
\end{eqnarray}
Under the above bijection the   left-symmetric coalgebra homomorphism $\varphi=\varphi_{r, s}: A^{ }\# {}^{Q} V\rightarrow A^{ }\# {}^{Q'} V$ to $(r, s)$ is given by $\varphi(a, x)=(a+r(x), s(x))$ for all $a\in {A}$ and $x\in V$. Moreover, $\varphi=\varphi_{r, s}$ is an isomorphism if and only if $s: V\rightarrow V$ is a linear isomorphism.
\end{lemma}
\begin{proof} The proof is similar as the proof of Lemma \ref{lem-c1}.
Let $\varphi: A^{ }\# {}^{Q} V\rightarrow A^{}\# {}^{Q'} V$  be a  left-symmetric  coalgebra homomorphism  whose restriction on ${A}$ is the identity map.
First   it is easy to see that  $\Delta'_E\varphi(a)=(\varphi\otimes \varphi)\Delta_E(a)$ for all $a\in {A}$.
Then we consider that $\Delta'_E\varphi(x)=(\varphi\otimes \varphi)\Delta_E(x)$ for all $x\in V$.
\begin{eqnarray*}
\Delta'_E\varphi(x)&=&\Delta'_E({r}(x), {s}(x))=\Delta'_E({r}(x))+\Delta'_E({s}(x))\\
&=&\Delta'_A({r}(x))+\Delta'_V({s}(x))+{\rho}'({s}(x))+{\gamma}'({s}(x))+{Q}'({s}(x)),
\end{eqnarray*}
and
\begin{eqnarray*}
&&(\varphi\otimes \varphi)\Delta_E(x)\\
&=&(\varphi\otimes \varphi)(x\li\ot x\lii+x\boi\ot x\boo+x\boo\ot x\bi+{Q}(x))\\
&=&r(x\li)\ot r(x\lii)+r(x\li)\ot s(x\lii)+s(x\li)\ot r(x\lii)+s(x\li)\ot s(x\lii)\\
&&+x\boi\ot r(x\boo)+x\boi\ot s(x\boo)+r(x\boo)\ot x\bi+s(x\boo)\ot x\bi+{Q}(x).
\end{eqnarray*}
Thus we obtain that $\Delta'_E\varphi(x)=(\varphi\otimes \varphi)\Delta_E(x)$ if and only if the conditions \eqref{comorph1}, \eqref{comorph2},  \eqref{comorph3} and \eqref{comorph4} hold. By definition, we obtain that $\varphi=\varphi_{r,s}$ is an isomorphism if and only if $s: V\rightarrow V$ is a linear isomorphism.
\end{proof}

Let $({A},\Delta_A)$ be a   left-symmetric coalgebra and $V$ be a vector space. Two  left-symmetric  coalgebraic extending systems $\mathcal{C}^{(i)}({A}, V)$ and ${\mathcal{C}'^{(i)}}({A}, V)$  are called equivalent if $\varphi_{r, s}$ is an isomorphism.  We denote it by $\mathcal{C}^{(i)}({A}, V)\equiv{\mathcal{C}'^{(i)}}({A}, V)$.
From the above lemmas, we obtain the following result.
\begin{theorem}\label{thm3-2}
Let $({A},\Delta_A)$ be a  left-symmetric  coalgebra and $E$ a vector space containing ${A}$ as a subspace and
$V$ be a ${A}$-complement in $E$. Denote $\mathcal{HC}(V, {A}):=\mathcal{C}^{(1)}({A}, V)\sqcup\mathcal{C}^{(2)}({A}, V) /\equiv$. Then the map
\begin{eqnarray*}
&&\Psi: \mathcal{HC}(V, {A})\rightarrow CExtd(E, {A}),\\
&&\overline{\Omega^{(1)}({A}, V)}\mapsto A^{P}\# {}^{} V,
 \quad \overline{\Omega^{(2)}({A}, V)}\mapsto A^{}\# {}^{Q} V
\end{eqnarray*}
is bijective, where $\overline{\Omega^{(i)}({A}, V)}$ is the equivalence class of $\Omega^{(i)}({A}, V)$ under $\equiv$.
\end{theorem}

\subsection{Extending structures for    left-symmetric bialgebras}
Let $(A, \, \cdot, \, \Delta_A)$ be a    left-symmetric bialgebra. From $(CBB1)$ and $(CBB2)$  we have the following two cases.

The first case is that we assume $Q=0$ and $\ppr, \ppl$ to be trivial. Then by the above Theorem \ref{main2}, we obtain the following result.

\begin{theorem}\label{thm-41}
Let $(A, \, \cdot, \, \Delta_A)$ be a    left-symmetric bialgebra and $V$  be a vector space.
An extending datum of ${A}$ by $V$ of type $(I)$ is  $\mathcal{IB}^{(1)}({A}, V)=(\trr, \, \trl, \,  \phi, \, \psi,\, \rho,\, \gamma,\, \tht,\,  P, \, \cdot_V, \, \Delta_V)$ consisting of  linear maps
\begin{eqnarray*}
&&\trr: V\otimes {A}\rightarrow V,~~~~\trl:A\otimes V\rightarrow V,~~~~\phi :A \to V\otimes A,\quad{\psi}: V\to  V\otimes A,~~~~\theta:  A\otimes A \rightarrow {V},\\
 &&\rho:V\to A \otimes V,~~~~\gamma :V\to V \otimes A,~~~{P}: A\rightarrow {V}\otimes {V},~~~\cdot_V:V\otimes V \rightarrow V,~~~\Delta_V: V\rightarrow V\otimes V.
\end{eqnarray*}
Then the unified product $A^{P}_{}\# {}^{}_{\theta}\, V$ with product
\begin{align}
(a, x) (b, y):=(ab, xy+ a\trr y+x\trl b+\theta(a, b))
\end{align}
and coproduct
\begin{eqnarray}
\Delta_E(a)=\Delta_A(a)+{\phi}(a)+{\psi}(a)+P(a),\quad \Delta_E(x)=\Delta_V(x)+{\rho}(x)+{\gamma}(x)
\end{eqnarray}
forms a    left-symmetric bialgebra if and only if $A_{}\# {}_{\theta} V$ forms a left-symmetric algebra, $A^{P}\# {}^{} \, V$ forms a  left-symmetric  coalgebra and the following conditions are satisfied:
\begin{enumerate}
\item[(E1)]  $\phi([a, b]) +\gamma(\theta(a, b) -\theta(b, a))$\\
$=a_{(-1)} \otimes\left [a_{(0)}, b]\right)+(a \trr b_{(-1)}) \otimes b_{(0)}+\theta\left(a, b_{1}\right) \otimes b_{2} $\\
$ +b_{(-1)}\ot [a, b_{(0)}]-(b \trr a_{(-1)}) \otimes a_{(0)}-\theta\left(b, a_{1}\right) \otimes a_{2}  $,

\item[(E2)] $\psi([a, b]) +\rho(\theta(a, b)-\theta( b, a))$\\
$=(a b_{(0)}) \otimes b_{(1)}+a_{(0)} \otimes\left(a_{(1)} \trl b\right)
+a_{1} \otimes \theta\left(a_{2}, b\right)+b_{(0)}\ot (a\trr b_{(1)})+b_{1}\ot \tht(a, b_{2})$\\
$-(b a_{(0)}) \otimes a_{(1)}-b_{(0)} \otimes\left(b_{(1)} \trl a\right)
-b_{1} \otimes \theta\left(b_{2}, a\right)-a_{(0)}\ot (b\trr a_{(1)})-a_{1}\ot \tht(b, a_{2})$,

\item[(E3)] $\rho([x, y]) =x_{[-1]} \otimes [x_{[0]}, y]+y_{[-1]}\ot [x, y_{[0]}]$,

\item[(E4)] $\gamma([x, y])  =xy_{[0]}\otimes y_{[1]}  -yx_{[0]}\otimes x_{[1]} $,

\item[(E5)] $\Delta_{V}(a \trr y)-\Delta_{V}(y \trl a)$\\
$=a_{(-1)} \otimes\left(a_{(0)}\trr y\right)+\left(a \trr y_{1}\right) \otimes y_{2}+a\ppi \otimes [a\pii, y]+\theta\left(a, y_{[-1]}\right) \otimes y_{[0]}$\\
$+y\li\ot (a\trr y\lii)+y\poo\ot \tht(a, y\bi)-y_{1} \otimes\left(y_{2}  \trl a\right)-\left(y\trl a_{(0)}\right) \otimes a_{(1)}$\\
$-y_{[0]} \otimes \theta\left(y_{[1]}, a\right)-y a\ppi\otimes a\pii -a_{(-1)}\ot (y\trl a\loo)$,

\item[(E6)]$\Delta_{V}(\theta(a, b))-\theta(b, a))+P([a, b]) $\\
$=a_{(-1)} \otimes\theta(a_{(0)}, b)+a\ppi \otimes (a\pii\trl b)+\theta(a, b_{(0)})\otimes b_{(1)}+(a\trr b\ppi)\otimes b\pii$\\
$+b\loi \ot \tht(a, b\lmoo)+b\ppi\ot (a\trr b\pii)-b_{(-1)} \otimes\theta(b_{(0)}, a)-b\ppi \otimes (b\pii\trl a)$\\
$-\theta(b, a_{(0)})\otimes a_{(1)}-(b\trr a\ppi)\otimes a\pii-a\loi \ot \tht(b, a\lmoo)-a\ppi\ot (b\trr a\pii)$,

\item[(E7)] $\gamma(x\trl b)-\gamma(b\trr x)=x b_{(-1)} \otimes b_{(0)}+x_{[0]} \otimes [x_{[1]}, b]+\left(x\trl b_{1}\right) \otimes b_{2}-\left(b\trr x_{[0]}\right)\otimes x_{[1]} $,

\item[(E8)]
$\rho(x\trl b)-\rho(b\trr x)=x_{[-1]} \otimes\left(x_{[0]} \trl b\right)+b\li\ot (x\trl b\lii)$\\
$+b\loo\ot [x,b\lmi]-b_{1} \otimes\left(b_{2} \trr x\right)-b x_{[-1]} \otimes x_{[0]}-x\boi \ot (b\trr x\poo)$,

\item[(E9)]  $\phi(ab)+\gamma(\theta(a, b))-\tau\psi(a b)-\tau\rho(\theta(a, b))$\\
$=a_{(-1)} \otimes\left(a_{(0)} b\right)+(a \trr b_{(-1)}) \otimes b_{(0)}+\theta\left(a b_{1}\right) \otimes b_{2}+b_{(-1)})\ot ab_{(0)}) $\\
$-\tau\Big((a b_{(0)}) \otimes b_{(1)}+ a_{(0)} \otimes\left(a_{(1)} \trl b\right)+ a_{1} \otimes \theta\left(a_{2}, b\right)+ b_{(0)}\ot (a\trr b_{(1)})+ b_{1}\ot \tht(a, b_{2})\Big)$,

\item[(E10)] $\rho(x y) -\tau\gamma(x y) )=x_{[-1]} \otimes x_{[0]} y +y_{[-1]}\ot xy_{[0]}-y_{[1]}\otimes xy_{[0]}$,

\item[(E11)] $(\id-\tau)\Delta_{V}(a \trr y) $\\
$=(\id-\tau)\Big(a_{(-1)} \otimes\left(a_{(0)}\trr y\right)+\left(a \trr y_{1}\right) \otimes y_{2}+a\ppi \otimes a\pii y$\\
$+\theta\left(a, y_{[-1]}\right) \otimes y_{[0]}+y\li\ot (a\trr y\lii)+y\poo\ot \tht(a,y\bi)$\Big),

\item[(E12)] $(\id-\tau)\Delta_{V}(x \trl b)$\\
$=(\id-\tau)\Big(x_{1} \otimes\left(x_{2}  \trl b\right)+\left(x\trl b_{(0)}\right) \otimes b_{(1)}+x_{[0]} \otimes \theta\left(x_{[1]}, b\right)$\\
$+x b\ppi\otimes b\pii+b\ppi\ot xb\pii+b_{(-1)}\ot (x\trl b\loo)$\Big),

\item[(E13)]$(\id-\tau)(\Delta_{V}(\theta(a,b)) +P(a b) )$\\
$=(\id-\tau)\Big(a_{(-1)} \otimes\theta(a_{(0)}, b)+a\ppi \otimes (a\pii\trl b)+\theta(a, b_{(0)})\otimes b_{(1)}$\\
$+(a\trr b\ppi)\otimes b\pii+b\loi \ot \tht(a, b\lmoo)+b\ppi\ot (a\trr b\pii)$\Big),

\item[(E14)]
 $\gamma(x\trl b) -\tau\rho(x\trl b) $\\
 $=x b_{(-1)} \otimes b_{(0)}+x_{[0]} \otimes x_{[1]} b+\left(x\trl b_{1}\right) \otimes b_{2}-\left(x_{[0]} \trl b\right)\ot x_{[-1]}-(x\trl b\lii)\ot b\li-xb\lmi\ot b\loo$,

\item[(E15)]$ \rho(a \trr y)  -\tau\gamma(a \trr y)$\\
$=a_{(0)} \otimes a_{(1)} y +a_{1} \otimes\left(a_{2} \trr y\right) +a y_{[-1]} \otimes y_{[0]}+y\boi \ot (a\trr y\poo )-\tau\Big( \left(a\trr y_{[0]}\right)\otimes y_{[1]} + y\poo\ot ay\bi\Big)$,
\end{enumerate}
\begin{enumerate}
\item[(E16)] $\Delta_{V}([x, y]) $\\
$=x_{1} \otimes [x_{2},  y]+x y_{1} \otimes y_{2}-y x_{1} \otimes x_{2}- y_{1} \otimes [ y_{2}, x]$\\
$+x_{[0]} \otimes\left(x_{[1]} \trr y\right)+\left(x \trl y_{[-1]}\right) \otimes y_{[0]}+ y_{[0]} \otimes \left(x \trl y_{[1]}\right)$\\
$ -y_{[0]} \otimes\left(y_{[1]} \trr x\right)-\left(y \trl x_{[-1]}\right) \otimes x_{[0]}- x_{[0]} \otimes \left(y \trl x_{[1]}\right)$,
\end{enumerate}
\begin{enumerate}
\item[(E17)] $(\id-\tau)\left(\Delta_{V}(xy) \right)$\\
$=(\id-\tau)\Big(x_{1} \otimes x_{2} y+x y_{1} \otimes y_{2}+ y_{1} \otimes x y_{2}+x_{[0]} \otimes\left(x_{[1]} \trr y\right)+\left(x \trl y_{[-1]}\right) \otimes y_{[0]}+ y_{[0]} \otimes \left(x \trl y_{[1]}\right)\Big)$.
\end{enumerate}

Conversely, any    left-symmetric bialgebra structure on $E$ with the canonical projection map $p: E\to A$ both a left-symmetric algebra homomorphism and a   left-symmetric  coalgebra homomorphism is of this form.
\end{theorem}
Note that in this case, $(V, \, \cdot, \, \Delta_V)$ is a  braided    left-symmetric bialgebra. Although $(A, \, \cdot, \, \Delta_A)$ is not a left-symmetric  sub-bialgebra of $E=A^{P}_{}\# {}^{}_{\theta}\, V$, but it is indeed a    left-symmetric bialgebra and a subspace $E$.
Denote the set of all    left-symmetric bialgebraic extending datum of type $(I)$ by $\mathcal{IB}^{(1)}({A},V)$.

The second case is that we assume $P=0, \theta=0$ and $\phi, \psi$ to be trivial. Then by the above Theorem \ref{main2}, we obtain the following result.

\begin{theorem}\label{thm-42}
Let $A$ be a    left-symmetric bialgebra and $V$ be a vector space.
An extending datum of ${A}$ by $V$ of type $(II)$  is   $\mathcal{IB}^{(2)}({A}, \, V)=(\ppr, \,  \ppl, \,  \trr, \,  \trl, \,  \sigma, \,  \rho,  \, \gamma,  \, Q,  \,  \cdot_V,  \, \Delta_V)$ consisting of  linear maps
\begin{eqnarray*}
&&\ppr: V\otimes A \to A,~~~ \ppl: A\otimes V \to A,~~~\trl: V\otimes {A}\rightarrow {V},~~~\trr: A\otimes {V}\rightarrow V,~~~\sigma:  V\otimes V \rightarrow {A},\\
&&{\rho}: V\to  A\otimes V,~~{\gamma}: V\to  V\otimes A,~~{Q}: V\rightarrow {A}\otimes {A},~~\Delta_V: V\rightarrow V\otimes V,~~\cdot_V: V\otimes V \rightarrow V.
\end{eqnarray*}
Then the unified product $A^{}_{\sigma}\# {}^{Q}_{}\, V$ with product
\begin{align}
(a, x)(b, y)=\big(ab+x\ppr b+a\ppl y+\sigma(x, y), \, xy+x\trl b+a\trr y\big).
\end{align}
and coproduct
\begin{eqnarray}
\Delta_E(a)=\Delta_A(a),\quad \Delta_E(x)=\Delta_V(x)+{\rho}(x)+{\gamma}(x)+Q(x)
\end{eqnarray}
forms a    left-symmetric bialgebra if and only if $A_{\sigma}\# {}_{} V$ forms a  left-symmetric algebra, $A^{}\# {}^{Q}V$ forms a left-symmetric  coalgebra and the following conditions are satisfied:

\begin{enumerate}
\item[(F1)] $\rho([x, y]) $\\
$=x_{[-1]} \otimes [x_{[0]}, y] +\left(x \ppr y_{[-1]}\right) \otimes y_{[0]}+x\qi\otimes (x\qii\trr y)+\sigma\left(x, y_{1}\right) \otimes y_{2}+y_{[-1]}\ot [x, y_{[0]}]$\\
$ +y\qi\ot (x\trl y\qii)-\left(y \ppr x_{[-1]}\right) \otimes x_{[0]}-y\qi\otimes y\qii\trr x-\sigma\left(y, x_{1}\right) \otimes x_{2} -x\qi\ot (y\trl x\qii)$,

\item[(F2)] $\gamma([x, y]) $\\
$=x_{[0]}\otimes (x_{[1]}\ppl y)+xy_{[0]}\otimes y_{[1]}+x_{1} \otimes \sigma\left(x_{2}, y\right)+ \left(x\trl y\qi\right)\otimes y\qii$\\
$+y\li\ot \si(x, y\lii)+y\poo\ot (x\ppr y\bi)-y_{[0]}\otimes (y_{[1]}\ppl x)-yx_{[0]}\otimes x_{[1]}$\\
$-y_{1} \otimes \sigma\left(y_{2}, x\right)- \left(y\trl x\qi\right)\otimes x\qii-x\li\ot \si(y, x\lii)-x\poo\ot (y\ppr x\bi)$,

\item[(F3)] $\Delta_{A}(x \ppr b)-\Delta_{A}(b \ppl x)+Q(x\trl  b)-Q(b \trr x)$ \\
$=x_{[-1]} \otimes\left(x_{[0]} \ppr b\right)+\left(x \ppr b_{1}\right) \otimes b_{2}+x\qi \otimes [x\qii, b]  +b\li \ot (x\ppr b\lii) $\\
$-b_{1} \otimes\left(b_{2} \ppl x\right)-\left(b\ppl x_{[0]}\right) \otimes x_{[1]} -b x\qi \otimes x\qii-x\boi \ot (b\ppl x\boo) $,

\item[(F4)] $\Delta_{V}(a \trr y)-\Delta_{V}(y \trl a) = \left(a \trr y_{1}\right) \otimes y_{2}  +y\li\ot (a\trr y\lii)-y_{1} \otimes\left(y_{2}  \trl a\right) $,

\item[(F5)]$\Delta_{A}(\sigma(x, y)-\sigma(y, x))+Q([x, y]) $\\
$=x_{[-1]}\otimes \sigma(x_{[0]}, y)+x\qi\otimes (x\qii\ppl y)+\sigma(x, y_{[0]})\otimes y_{[1]}+(x\ppr y\qi)\otimes y\qii$\\
$+y\poi \ot \si(x, y\poo)+y\qi\ot (x\ppr y\qii)-y_{[-1]}\otimes \sigma(y_{[0]}, x)-y\qi\otimes (y\qii\ppl x)$\\
$-\sigma(y, x_{[0]})\otimes x_{[1]}-(y\ppr x\qi)\otimes x\qii-x\poi \ot \si(y, x\poo)-x\qi\ot (y\ppr x\qii)$,

\item[(F6)]
 $\gamma(x\trl b)-\gamma(b\trr x)$\\
 $=x_{1} \otimes\left(x_{2} \ppr b\right) +x_{[0]} \otimes [x_{[1]}, b] +\left(x\trl b_{1}\right) \otimes b_{2}
-\left(b\trr x_{[0]}\right)\otimes x_{[1]}  -x\li \ot (b\ppl x\lii)$,

\item[(F7)]
$ \rho(x\trl b)-\rho(b\trr x)$\\
$= x_{[-1]} \otimes\left(x_{[0]} \trl b\right) +b\li\ot (x\trl b\lii)
  -\left(b\ppl x_{1}\right) \otimes x_{2}-b_{1} \otimes\left(b_{2} \trr x\right)
 -b x_{[-1]} \otimes x_{[0]}-x\boi \ot (b\trr x\poo )$,

\item[(F8)] $\rho(x y)  -\tau\gamma(x y) ) $\\
$=x_{[-1]} \otimes x_{[0]} y +\left(x \ppr y_{[-1]}\right) \otimes y_{[0]}+x\qi\otimes (x\qii\trr y)+\sigma\left(x, y_{1}\right) \otimes y_{2}$\\
$+y_{[-1]}\ot xy_{[0]}+y\qi\ot (x\trl y\qii)-\tau\Big(x_{[0]}\otimes (x_{[1]}\ppl y)+xy_{[0]}\otimes y_{[1]}$\\
$+x_{1} \otimes \sigma\left(x_{2}, y\right)+ \left(x\trl y\qi\right)\otimes y\qii+y\li\ot \si(x, y\lii)+y\poo\ot (x\ppr y\bi)\Big)$,

\item[(F9)] $(\id-\tau)(\Delta_{A}(x \ppr b) +Q(x\trl  b) )$ \\
$=(\id-\tau)\Big(x_{[-1]} \otimes\left(x_{[0]} \ppr b\right)+\left(x \ppr b_{1}\right) \otimes b_{2}+x\qi \otimes x\qii b+b\li \ot (x\ppr b\lii) $\Big),

\item[(F10)] $(\id-\tau)(\Delta_{A}(a\ppl y) +Q(a\trr y) )$\\
$=(\id-\tau)\Big(a_{1} \otimes\left(a_{2} \ppl y\right)+\left(a\ppl y_{[0]}\right) \otimes y_{[1]}
 +a y\qi \otimes y\qii+y\boi \ot (a\ppl y\boo)+y\qi\ot ay\qii$\Big),

\item[(F11)] $(\id-\tau) \Delta_{H}(a \trr y)
 =(\id-\tau)\Big( \left(a \trr y_{1}\right) \otimes y_{2}
 +y\li\ot (a\trr y\lii )$\Big),

\item[(F12)] $(\id-\tau)\Delta_{H}(x \trl b)
=x_{1} \otimes\left(x_{2}  \trl b\right)-\left(x_{2}  \trl b\right)\ot x_{1}$,

\item[(F13)]$(\id-\tau)(\Delta_{A}(\sigma(x,y)) +Q(x y) )$\\
$=(\id-\tau)\Big(x_{[-1]}\otimes \sigma(x_{[0]},y)+x\qi\otimes (x\qii\ppl y)+\sigma(x,y_{[0]})\otimes y_{[-1]}$\\
$+(x\ppr y\qi)\otimes y\qii+y\poi \ot \si(x,y\poo)+y\qi\ot (x\ppr y\qii)$\Big),

\item[(F14)]
 $\gamma(x\trl b)-\tau\rho(x\trl b) $\\
 $=x_{1} \otimes\left(x_{2} \ppr b\right)+x_{[0]} \otimes x_{[1]} b+\left(x\trl b_{1}\right) \otimes b_{2}
-\tau\Big(x_{[-1]} \otimes\left(x_{[0]} \trl b\right)+b\li\ot (x\trl b\lii)\Big)$,

\item[(F15)]$\rho(a \trr y)-\tau\gamma(a \trr y)$\\
$=\left(a\ppl y_{1}\right) \otimes y_{2}+a_{1} \otimes\left(a_{2} \trr y\right)+a y_{[-1]} \otimes y_{[0]}+y\boi \ot (a\trr y\poo)$\\
$-\tau\Big(\left(a\trr y_{[0]}\right)\otimes y_{[1]}+y\li \ot (a\ppl y\lii)+ y\poo\ot ay\bi\Big)$.
\end{enumerate}
\begin{enumerate}
\item[(F16)] $\Delta_{V}([x, y]) $\\
$=x_{1} \otimes [x_{2}, y]+x y_{1} \otimes y_{2}-y x_{1} \otimes x_{2}- y_{1} \otimes [ y_{2}, x]$\\
$+x_{[0]} \otimes\left(x_{[1]} \trr y\right)+\left(x \trl y_{[-1]}\right) \otimes y_{[0]}+ y_{[0]} \otimes \left(x \trl y_{[1]}\right)$\\
$ -y_{[0]} \otimes\left(y_{[1]} \trr x\right)-\left(y \trl x_{[-1]}\right) \otimes x_{[0]}- x_{[0]} \otimes \left(y \trl x_{[1]}\right)$,
\end{enumerate}
\begin{enumerate}
\item[(F17)] $(\id-\tau)\left(\Delta_{V}(xy) \right)$\\
$=(\id-\tau)\Big(x_{1} \otimes x_{2} y+x y_{1} \otimes y_{2}+ y_{1} \otimes x y_{2}+x_{[0]} \otimes\left(x_{[1]} \trr y\right)+\left(x \trl y_{[-1]}\right) \otimes y_{[0]}+ y_{[0]} \otimes \left(x \trl y_{[1]}\right)\Big)$.
\end{enumerate}

Conversely, any    left-symmetric bialgebra structure on $E$ with the canonical injection map $i: A\to E$ both a  left-symmetric  algebra homomorphism and a  left-symmetric   coalgebra homomorphism is of this form.
\end{theorem}
Note that in this case, $(A, \, \cdot, \, \Delta_A)$ is a  left-symmetric  sub-bialgebra of $E=A^{}_{\sigma}\# {}^{Q}_{}\, V$ and $(V, \, \cdot, \, \Delta_V)$ is a  braided   left-symmetric bialgebra.
Denote the set of all     left-symmetric bialgebraic extending datum of type $(II)$ by $\mathcal{IB}^{(2)}({A}, V)$.

In the above two cases, we find that  the braided    left-symmetric bialgebra $V$ play a special role in the extending problem of   left-symmetric bialgebra $A$.
Note that $A^{P}_{}\# {}^{}_{\theta}\, V$ and $A^{}_{\sigma}\# {}^{Q}_{}\, V$ are all    left-symmetric bialgebra structures on $E$.
Conversely,  any    left-symmetric bialgebra extending system $E$ of ${A}$  through $V$ is isomorphic to such two types.
Now from Theorem \ref{thm-41}, Theorem \ref{thm-42} we obtain the main result of in this section, which solve the extending problem for    left-symmetric bialgebra.

\begin{theorem}\label{bim1}
Let $({A}, \, \cdot, \, \Delta_A)$ be a left-symmetric bialgebra and $E$ be a vector space containing ${A}$ as a subspace and $V$ be a complement of ${A}$ in $E$.
Denote by
$$\mathcal{HLB}(V, {A}):=\mathcal{IB}^{(1)}({A},V)\sqcup\mathcal{IB}^{(2)}({A},V)/\equiv.$$
Then the map
\begin{eqnarray}
&&\Upsilon: \mathcal{HLB}(V,{A})\rightarrow BExtd(E,{A}),\\
&&\overline{\mathcal{IB}^{(1)}({A},V)}\mapsto A^{P}_{}\# {}^{}_{\theta}\, V,\quad   \overline{\mathcal{IB}^{(2)}({A},V)}\mapsto A^{}_{\sigma}\# {}^{Q}_{}\, V
\end{eqnarray}
is bijective, where $\overline{\mathcal{IB}^{(i)}({A}, V)}$ is the equivalence class of $\mathcal{IB}^{(i)}({A}, V)$ under $\equiv$.
\end{theorem}

A very special case is that when $\ppr$ and $\ppl$ are trivial in the above Theorem \ref{thm-42}.    We obtain the following result.
\begin{corollary}\label{thm4}
Let $A$ be a    left-symmetric bialgebra and $V$ be a vector space.
An extending datum of ${A}$ by $V$ is  $\mathcal{IB}^{(3)}({A}, V)=(\trr, \, \trl, \, \sigma, \, \rho, \, \gamma,\,  Q, \, \cdot, \, \Delta_V)$ consisting of  eight linear maps
\begin{eqnarray*}
\trl: V\otimes {A}\rightarrow {V},~~~~\trr: A\otimes {V}\rightarrow V,~~~~\sigma:  V\otimes V \rightarrow {A},~~~\cdot_V:V\otimes V \rightarrow V,\\
{\rho}: V\to  A\otimes V,~~~~{\gamma}: V\to  V\otimes A,~~~~{Q}: V\rightarrow {A}\otimes {A},~~~~\Delta_V: V\rightarrow V\otimes V.
\end{eqnarray*}
Then the unified product $A^{}_{\sigma}\# {}^{Q}_{}\, V$ with product
\begin{align}
(a, x) (b, y):=(ab+\sigma(x, y), \,  xy+x\trl b+a\trr y)
\end{align}
and coproduct
\begin{eqnarray}
\Delta_E(a)=\Delta_A(a),\quad \Delta_E(x)=\Delta_V(x)+{\rho}(x)+{\gamma}(x)+Q(x)
\end{eqnarray}
forms a    left-symmetric bialgebra if and only if $A_{\sigma}\# {}_{} V$ forms a  left-symmetric  algebra, $A^{}\# {}^{Q} \, V$ forms a  left-symmetric  coalgebra and the following conditions are satisfied:
\begin{enumerate}
\item[(G1)] $\rho([x, y])  $\\
$=x_{[-1]} \otimes [x_{[0]}, y]  +x\qi\otimes (x\qii\trr y)
 +\sigma\left(x, y_{1}\right) \otimes y_{2}+y_{[-1]}\ot [x, y_{[0]}]$\\
$ +y\qi\ot (x\trl y\qii)-y\qi\otimes (y\qii\trr x)-\sigma\left(y, x_{1}\right) \otimes x_{2}-x\qi\ot (y\trl x\qii)$,

\item[(G2)] $\gamma([x, y])  $\\
$= xy_{[0]}\otimes y_{[1]}+x_{1} \otimes \sigma\left(x_{2}, y\right)
 + \left(x\trl y\qi\right)\otimes y\qii+y\li\ot \si(x, y\lii) $\\
$ -yx_{[0]}\otimes x_{[1]}-y_{1} \otimes \sigma\left(y_{2}, x\right)
 - \left(y\trl x\qi\right)\otimes x\qii-x\li\ot \si(y, x\lii) $,

\item[(G3)] $ Q(x\trl  b)-Q(b \trr x) = x\qi \otimes x\qii b -b x\qi \otimes x\qii -x\qi\ot bx\qii$,

\item[(G4)] $\Delta_{V}(a \trr y)-\Delta_{V}(y \trl a) = \left(a \trr y_{1}\right) \otimes y_{2}  +y\li\ot (a\trr y\lii)-y_{1} \otimes\left(y_{2}  \trl a\right) $,

\item[(G5)]$\Delta_{A}(\sigma(x, y)-\sigma(y, x))+Q([x, y])$\\
$=x_{[-1]}\otimes \sigma(x_{[0]}, y) +\sigma(x,y_{[0]})\otimes y_{[-1]} +y\poi \ot \si(x, y\poo) $\\
$-y_{[-1]}\otimes \sigma(y_{[0]}, x) -\sigma(y,x_{[0]})\otimes x_{[-1]}-x\poi \ot \si(y, x\poo) $,

\item[(G6)]
 $\gamma(x\trl b)-\gamma(b\trr x)= x_{[0]} \otimes [x_{[1]}, b] +\left(x\trl b_{1}\right) \otimes b_{2}-\left(b\trr x_{[0]}\right)\otimes x_{[1]}  $,

\item[(G7)]
$ \rho(x\trl b)-\rho(b\trr x)$\\
$= x_{[-1]} \otimes\left(x_{[0]} \trl b\right) +b\li\ot (x\trl b\lii)
  -b_{1} \otimes\left(b_{2} \trr x\right)
 -b x_{[-1]} \otimes x_{[0]}-x\boi \ot (b\trr x\poo) $,

\item[(G8)] $\rho(x y)  -\tau\gamma(x y) ) $\\
$=x_{[-1]} \otimes x_{[0]} y +x\qi\otimes (x\qii\trr y)
 +\sigma\left(x, y_{1}\right) \otimes y_{2}+y_{[-1]}\ot xy_{[0]}+y\qi\ot (x\trl y\qii)$\\
$-\tau\Big( xy_{[0]}\otimes y_{[1]}+x_{1} \otimes \sigma\left(x_{2}, y\right)
 + \left(x\trl y\qi\right)\otimes y\qii+y\li\ot \si(x,y\lii) \Big)$,

\item[(G9)] $(\id-\tau) Q(x\trl  b)=x\qi \otimes x\qii b  -x\qii b\ot x\qi$ ,

\item[(G10)] $(\id-\tau)Q(a\trr y)
=(\id-\tau)\Big(a y\qi \otimes y\qii+y\qi\ot ay\qii$\Big),

\item[(G11)] $(\id-\tau) \Delta_{V}(a \trr y)
 =(\id-\tau)\Big( \left(a \trr y_{1}\right) \otimes y_{2}
 +y\li\ot a\trr y\lii $\Big),

\item[(G12)] $(\id-\tau)\Delta_{V}(x \trl b)
=(\id-\tau)\Big(x_{1} \otimes\left(x_{2}  \trl b\right)$\Big),

\item[(G13)]$(\id-\tau)(\Delta_{A}(\sigma(x,y)) +Q(x y) )$\\
$=(\id-\tau)\Big(x_{[-1]}\otimes \sigma(x_{[0]},y) +\sigma(x,y_{[0]})\otimes y_{[-1]}  +y\poi \ot \si(x,y\poo) $\Big),

\item[(G14)]
 $\gamma(x\trl b)-\tau\rho(x\trl b)
  = x_{[0]} \otimes x_{[1]} b+\left(x\trl b_{1}\right) \otimes b_{2}
-\tau\Big(x_{[-1]} \otimes\left(x_{[0]} \trl b\right)+b\li\ot (x\trl b\lii)\Big)$,

\item[(G15)]$\rho(a \trr y)-\tau\gamma(a \trr y)$\\
$= a_{1} \otimes\left(a_{2} \trr y\right)+a y_{[-1]} \otimes y_{[0]}+y\boi \ot (a\trr y\poo ) -\tau\Big(\left(a\trr y_{[0]}\right)\otimes y_{[1]} + y\poo\ot ay\bi\Big)$,
\end{enumerate}
\begin{enumerate}
\item[(G16)] $\Delta_{V}([x, y]) $\\
$=x_{1} \otimes [x_{2}, y]+x y_{1} \otimes y_{2}-y x_{1} \otimes x_{2}- y_{1} \otimes [ y_{2}, x]$\\
$+x_{[0]} \otimes\left(x_{[1]} \trr y\right)+\left(x \trl y_{[-1]}\right) \otimes y_{[0]}+ y_{[0]} \otimes \left(x \trl y_{[1]}\right)$\\
$ -y_{[0]} \otimes\left(y_{[1]} \trr x\right)-\left(y \trl x_{[-1]}\right) \otimes x_{[0]}- x_{[0]} \otimes \left(y \trl x_{[1]}\right)$,
\end{enumerate}
\begin{enumerate}
\item[(G17)] $(\id-\tau) \Delta_{H}(xy)  $\\
$=(\id-\tau)\Big(x_{1} \otimes x_{2} y+x y_{1} \otimes y_{2}+ y_{1} \otimes x y_{2}
+x_{[0]} \otimes\left(x_{[1]} \trr y\right)+\left(x \trl y_{[-1]}\right) \otimes y_{[0]}+ y_{[0]} \otimes \left(x \trl y_{[1]}\right)\Big)$.
\end{enumerate}
\end{corollary}

\section*{Acknowledgements}
This is a primary edition, something should be modified in the future.

\vskip7pt
\footnotesize{
\noindent Tao Zhang\\
College of Mathematics and Information Science,\\
Henan Normal University, Xinxiang 453007, P. R. China;\\
 E-mail address: \texttt{{zhangtao@htu.edu.cn}}

\vskip7pt
\footnotesize{
\noindent Hui-jun Yao\\
College of Mathematics and Information Science,\\
Henan Normal University, Xinxiang 453007, P. R. China};\\
 E-mail address: \texttt{{yhjdyxa@126.com}}

\end{document}